\documentclass[a4paper]{cas-sc}

\usepackage[authoryear,longnamesfirst]{natbib}
\usepackage{tikz-cd}
\usepackage{cancel}
\usepackage{booktabs}
\usepackage{tabularx}
\usepackage{array}

\def\tsc#1{\csdef{#1}{\textsc{\lowercase{#1}}\xspace}}
\tsc{WGM}
\tsc{QE}
\newtheorem{theorem}{Theorem}
\newtheorem{lemma}[theorem]{Lemma}
\newtheorem{proposition}[theorem]{Proposition}
\newtheorem{corollary}[theorem]{Corollary}
\newdefinition{remark}{Remark}
\newproof{proof}{Proof}

\newcommand{\N}{\mathbb{N}}

\newcommand{\p}{\partial}

\makeatletter
\providecommand{\llangle}{}
\providecommand{\rrangle}{}

\renewcommand{\llangle}{\@ifnextchar\bgroup{\llangle@big}{\llangle@normal}}
\renewcommand{\rrangle}{\@ifnextchar\bgroup{\rrangle@big}{\rrangle@normal}}

\newcommand{\llangle@normal}{\mathopen{\langle\!\langle}}
\newcommand{\rrangle@normal}{\mathclose{\rangle\!\rangle}}

\newcommand{\llangle@big}[1]{\mathopen{#1\langle\!#1\langle}}
\newcommand{\rrangle@big}[1]{\mathclose{#1\rangle\!#1\rangle}}
\makeatother

\begin{document}
\let\WriteBookmarks\relax
\def\floatpagepagefraction{1}
\def\textpagefraction{.001}

\shorttitle{Connection towers and higher-order Sasaki metrics}

\shortauthors{Camarinha and Goodman}

\title{Connection Towers and Sasaki Metrics on Higher-Order Tangent Bundles}

\author[1]{Margarida Camarinha}[orcid=0000-0003-4587-7861]

\fnmark[1]

\ead{mmlsc@mat.uc.pt}

\ead[url]{https://www.cienciavitae.pt/portal/en/AC1F-473B-BEE5}

% Address/affiliation of the first author
\affiliation[1]{organization={CMUC, Department of Mathematics, University of Coimbra},
    postcode={3000-143},
    city={Coimbra},
    country={Portugal}}

\author[2]{Jacob Goodman}[orcid=0000-0002-6357-9326]

\cormark[1]

% Footnote of the second author
\fnmark[2]

% Email id of the second author
\ead{jacob.goodman@ntnu.no}

\ead[url]{https://www.ntnu.edu/employees/jacob.goodman}

% Address/affiliation of the second author
\affiliation[2]{organization={Department of Mathematical Sciences, Norwegian University of Science and Technology},
    city={Trondheim},
    country={Norway}}

\cortext[1]{Corresponding author}

\fntext[1]{M. Camarinha was supported by the Centre for Mathematics of the University of Coimbra (CMUC, https://doi.org/10.54499/UID/00324/2025) under the Portuguese Foundation for Science and Technology (FCT), Grants UID/00324/2025 and UID/PRR/00324/2025.}

\fntext[2]{J. Goodman was supported by the Marie Skłodowska-Curie grant agreement No. 101206748 (GNACS).}


\begin{abstract}
Higher-order tangent bundles possess a rich tower of fibrations, suggesting the existence of geometric structures compatible with their iterated bundle structure. In this paper, we introduce the notion of a connection tower on a higher-order tangent bundle and study the geometric structures induced by such towers. In particular, we show that connection towers determine natural multiconnections, adapted splittings of the tangent bundle, and canonical vector bundle structures on higher-order tangent bundles.
We then construct a specific connection tower induced by the Levi–Civita connection of a Riemannian manifold. This construction extends the classical Dombrowski connection map on the tangent bundle and leads naturally to a family of higher-order Sasaki metrics. We study the associated lifts of vector fields and derive explicit Lie bracket formulas for these lifts, together with structural identities for the induced multiconnection.
Finally, we determine the Levi–Civita connection of the higher-order Sasaki metrics and derive explicit geodesic equations on the second- and third-order tangent bundles. We also obtain characterization results relating geodesics of the higher-order Sasaki metrics to geodesics on the base manifold.
\end{abstract}

% The Levi–Civita connection induces a nonlinear connection in the tangent bundle. This induced structure can be described in terms of a connection map. When combined with the Riemannian metric, the connection map defines a natural Riemannian metric on the tangent bundle, known as the Sasaki metric. In this paper, we introduce a connection map in the higher-order tangent bundles of a Riemannian manifold and study the associated multiconnection. We examine how vector fields can be lifted by the multiconnection and describe the Lie bracket using these lifts. This connection map allows one to define a Riemannian metric in the higher-order tangent bundles, which coincides with the Sasaki metric in the tangent bundle. We also determine the corresponding Levi–Civita connection and derive the geodesic equations for the second- and third-order tangent bundles. Finally, we extend classical results concerning Sasaki geodesics to higher-order bundles. 



% Keywords

\maketitle

\section{Introduction}

Nonlinear connections on higher-order tangent bundles, as well as on fiber bundles in general, play a fundamental role in modern differential geometry. They provide a natural way to distinguish between vertical directions, corresponding to motions within the fibers, and horizontal directions, representing motion along the base manifold. This decomposition gives rise, in a purely geometric manner, to important notions such as covariant differentiation, curvature, and torsion. The relevance of these geometric structures extends well beyond pure mathematics, particularly when considered in the context of higher-order tangent bundles.

Higher-order tangent bundles are natural generalizations of the notion of the tangent bundle, where equivalence classes of higher-order contact curves replace the first-order ones. They  serve as evolution spaces for higher-order differential equations and, therefore, provide a natural framework for their study. Thanks to the contributions of Morimoto \cite{Morimoto1970}, León and  Rodrigues \cite{DeLeonRod1985}, Crampin et al. \cite{Crampin1986}, Saunders  \cite{Saunders1989}, Andrés et al. \cite{Andres1991}, Miron \cite{MironBook1997}, \cite{Miron2003}, Bucătaru \cite{Buc2007}, \cite{Buc2013}, Suri \cite{Suri2017} and others, who have highlighted the importance of these fiber bundles in areas such as the calculus of variations, Lagrangian and Hamiltonian mechanics, classical field theory and Finsler geometry, these geometric structures have been attracting increasing attention in recent years. 

In general, nonlinear connections do not arise from linear connections, as no linear structure is assumed on the fibers of a general fiber bundle. However, in the case of a Riemannian manifold, the Levi–Civita connection naturally induces a nonlinear connection in the tangent bundle. In fact, this can be done in a more general setting, when the fiber bundle is vectorial, or even principal, and endowed with a linear connection (\cite{ Kobayashi1957},\cite{YanoKob1966},\cite{ Vilms1967}). What is particularly interesting in the case of a Riemannian manifold is that the induced structure can be described in terms of a connection map. Furthermore, the induced connection map, when combined with the Riemannian metric, defines a natural Riemannian metric on the tangent bundle, known as the Sasaki metric (\cite{ Sasaki1958},\cite{Gudmundsson2002}). This metric is characterized by the orthogonality of the horizontal and vertical distributions determined by the induced connection. Analogous constructions also appear in Finsler geometry (\cite{BaoChernShen2000},\cite{BejancuFarran2006}).

Connection maps in the tangent bundle were originally introduced by Dombrowski \cite{ Dombrowski1962} and extended to vector bundles by Miron \cite{MironAtanasiu1996}. Later, Bucătaru characterized nonlinear connections in higher-order tangent bundles using connection maps and developed a geometric setting that associates nonlinear connections with higher-order differential equations  (\cite{Buc1997},\cite{ Buc2011}). Higher-order tangent bundles can naturally be viewed as forming a tower of fiber bundle submersions, due to the variety of fibrations they possess, with the projection onto the base manifold arising in the limit. This perspective suggests the possibility of adapting the notion of connection maps in order to define a tower of nonlinear connections. In this context, many questions may be raised. In the present paper, our purpose is to study a specific connection tower induced by the Levi--Civita connection and, subsequently, the resulting Riemannian structure in higher-order tangent bundles.

The main contributions of this paper are the following.
\begin{enumerate}
    \item We introduce the notion of a connection tower on a higher-order tangent bundle and develop its basic geometric properties. In particular, we show that connection towers induce canonical multiconnection structures and vector bundle structures on higher-order tangent bundles.
    \item We construct a specific connection tower induced by the Levi–Civita connection of a Riemannian manifold. This construction extends the classical Dombrowski connection map on the tangent bundle to the higher-order setting.
    \item Using the induced multiconnection, we study horizontal and vertical lifts of vector fields and derive explicit formulas for the Lie brackets of these lifts.
    \item We define a family of higher-order Sasaki metrics associated with the Levi–Civita-induced connection tower and determine the corresponding Levi–Civita connections.
    \item We derive explicit geodesic equations on the second- and third-order tangent bundles and obtain characterization results relating higher-order Sasaki geodesics to geodesics on the base manifold.
\end{enumerate}

The paper is organized as follows. In Section 2, we briefly recall basic definitions and notations of the theory of higher-order tangent bundles. We also provide a more detailed exposition of the theory of nonlinear connections and connection maps. Then we introduce the concept of nonlinear connection tower in terms of connection maps. The main results of this paper are stated and proved in Sections 3 and 4. In these two sections, the presentation is structured to emphasize the analogy with the results formulated in Subsections \ref{subsec TM 1}, \ref{subsec TM 2} and \ref{subsec TM 3}. In Section 3, we construct a specific connection tower induced by the Levi--Civita connection. We then examine how vector fields can be lifted by this connection tower  and describe the Lie bracket in terms of  these lifts. In Section 4, we define a family of higher-order Sasaki metrics that recover the classical Sasaki metric in the base case. We also determine the Levi–Civita connection and derive the geodesic equations for the second- and third-order tangent bundles. Finally, we extend classical results concerning Sasaki geodesics to higher-order bundles. Additional background information, computations, and technical lemmas are placed in the appendices. Table \ref{tab:notation} containing the notation frequently used in the paper can be found at the end of the manuscript. 

\section{The geometry of higher-order tangent bundles}
\label{sec:geometry-higher-order}

This section contains the geometric background on higher-order tangent bundles and nonlinear connections used throughout the paper.  We adopt the formalism of describing nonlinear connections by means of connection maps.  We first recall the tangent-bundle model, in the form needed later for comparison with the higher-order constructions.  We then review the basic geometry of $T^{(k)}M$, including its almost-tangent structure, vertical flag, connection maps, multiconnections, and adapted bases.  Finally, we introduce connection towers, which are special connection maps compatible with the canonical projections of higher-order tangent bundles, and we record the lift conventions that will be used in the sequel.

Additional background and auxiliary details are collected in Appendix~\ref{app:connection-background}.  This includes the fuller tangent-bundle discussion, the adapted-basis inversion formulas, the proof that connection towers induce vector bundle structures on the higher-order tangent bundles, and the basic technical lemmas for lifted vector fields. For convenience, Table \ref{tab:notation} at the end contains the notation commonly used throughout the article.

\subsection{The tangent-bundle model}\label{subsec TM 1}

Let $M$ be an $n$-dimensional smooth manifold and let $\tau:TM\to M$ denote the tangent-bundle projection.  We write a vector $u\in T_pM$ as $u=(p,v)$ when it is useful to distinguish the base point and fiber component, and we denote by $(q^{(0)i},q^{(1)i})$ the tangent-lifted local coordinates on $TM$ induced by local coordinates $(q^i)$ on $M$.

For $u\in TM$, the vertical space is
\[
        V_u=\ker(\tau_*\vert_u)\subset T_uTM,
\]
and the vertical bundle is $V=\bigsqcup_{u\in TM}V_u$.  If $u=(p,v)$ and $z\in T_pM$, the vertical lift of $z$ at $u$ is the tangent vector $z^v$ at $t=0$ of the curve $t\mapsto v+tz$.  Thus the vertical lift at $u$ gives a linear isomorphism $T_pM\cong V_u$.  The vertical lift of a vector field $X\in\mathfrak X(M)$ is the vector field $X^v$ on $TM$ given by
\[
        X^v(u)=(X(p))^v.
\]
In local coordinates, if $X=X^i\partial_i$, then
\[
        X^v=X^i\frac{\partial}{\partial q^{(1)i}}.
\]

The canonical almost-tangent structure $J$ on $TM$ is the $(1,1)$-tensor field
\[
        J(X)=(\tau_*X)^v,\qquad X\in T_uTM.
\]
In local coordinates,
\[
        J=\frac{\partial}{\partial q^{(1)i}}\otimes dq^{(0)i},
\]
so that $J^2=0$ and
\[
        V_u=\operatorname{Im}(J\vert_u)=\ker(J\vert_u).
\]

A nonlinear connection on $TM$ is a vector subbundle $H\subset TTM$ complementary to the vertical bundle
\begin{equation}\label{eq: direct_sum_TM}
        TTM=H\oplus V.
\end{equation}
Equivalently, a nonlinear connection can be defined via a \emph{connection map}. That is, a vector bundle morphism $K:TTM\to TM$ with base function $\tau$ such that
\[
        K\circ J=\tau_*,
\]
in which case the corresponding horizontal bundle is given by $H=\ker(K)$.  Conversely, every nonlinear connection is realized as the kernel of a unique connection map.  Locally,
\[
        K=\left(dq^{(1)i}+K^i_jdq^{(0)j}\right)\otimes \partial_i,
\]
where $K^i_j$ are the connection coefficients.  The connection map induces a canonical identification
\[
        \Phi_u:T_uTM\longrightarrow T_pM\oplus T_pM,\qquad
        \Phi_u(X)=(\tau_*X,K(X)),
\]
for each $u\in T_pM$ and $X \in T_u TM$.

A tangent vector $z\in T_pM$ has a unique horizontal lift $z^h\in H_u$ at $u = (p,v)\in T_pM$, characterized by
\[
        \tau_*(z^h)=z,\qquad K(z^h)=0.
\]
Vector fields are lifted similarly: if $X\in\mathfrak X(M)$, then $X^h(u)=(X(p))^h$ for $u\in T_pM$. 

\begin{remark}\label{rmk:pullback_bundle}
    We note that the horizontal lift depends on the chosen fiber velocity $v \in T_p M$ such that $u = (p,v)$. A more geometrically invariant construction is therefore to define lifts (and connection maps) on the \emph{pullback bundle} $\mathrm{pr}_1: \tau^\ast TM \to M$. However, we choose to proceed without making use of this formalism here, and often suppress the dependence on the fiber velocity when it is sufficiently clear from context.
\end{remark}

If $\nabla$ is a linear connection on $M$, then it induces a connection map on $TM$ by the Dombrowski construction \cite{Dombrowski1962}.  Namely, for $X\in T_uTM$, choose a family of curves $\gamma_s(t)=\gamma(s,t)$ on $M$ adapted to $X$, and set
\begin{equation}\label{eq: connection_map_TM}
        K(X)=\nabla_{\frac{\partial\gamma}{\partial t}}\frac{\partial\gamma}{\partial s}\Big\vert_{(s,t)=(0,0)}.
\end{equation}
For torsion-free $\nabla$ this agrees with the equivalent expression obtained by interchanging the two derivatives. In the present paper, the connection of interest is the Levi--Civita connection associated to a Riemannian metric $g$ on $M$, which is torsion-free by definition. We note, however, that many of the constructions and results obtained in the paper can be studied for general linear connections as well, with additional torsion components appearing when commutation of covariant derivatives are applied. 

\subsection{Higher-order tangent bundles}

We now recall the definitions and notation for tangent bundles of order $k$ that will be used in the rest of the paper; see \cite{DeLeonRod1985} for further details. Let $M$ be a smooth manifold of dimension $n$.  Two curves $q$ and $\tilde q$ on $M$ with a common initial point $p$ are said to be equivalent up to order $k$ at $p$ if there is a chart $(U,\varphi)$ around $p$ such that
\[
        \frac{d^\alpha(\varphi\circ q)}{dt^\alpha}(0)
        =
        \frac{d^\alpha(\varphi\circ \tilde q)}{dt^\alpha}(0),
        \qquad \alpha=1,\ldots,k.
\]
Denote by $j^k_0(q)$ the equivalence class of $q$, and define the \emph{tangent space of order} $k$ at $p$ by
\[
        T_p^{(k)}M=\{j^k_0(q) \ \vert \ q(0)=p\}.
\]
The \emph{tangent bundle of order} $k$ is then given by the disjoint union of all $k^{\text{th}}$-order tangent spaces
\[
        T^{(k)}M=\bigsqcup_{p\in M}T_p^{(k)}M,
\]
together with the fiber bundle structure $\tau_k: T^{(k)}M \to M$ defined by $\tau_k(j^k_0(q)) = q(0)$. These fiber bundles were introduced by Ehresmann \cite{Ehresmann1951} within the broader setting of jet manifolds.

A chart $(U,\varphi)$ on $M$, with local coordinates $(q^i)$, induces local coordinates $(q^{(0)i},q^{(1)i},\ldots,q^{(k)i})$
on $\tau_k^{-1}(U)$ by
\[
        q^{(0)i}=q^i,\qquad
        q^{(\alpha)i}(j^k_0(q)):=\frac{1}{\alpha!}\frac{d^\alpha q^i}{dt^\alpha}(0),
        \qquad \alpha=1,\ldots,k.
\]
Given a curve $q$ on $M$, its $k$-jet is the curve $j^kq:I\to T^{(k)}M$ defined by
\[
        (j^kq)(t)=j^k_0(q_t),\qquad q_t(s)=q(t+s).
\]
In local coordinates,
\[
        (j^kq)(t)=
        \left(
        q^i(t),\frac{dq^i}{dt}(t),\frac1{2}\frac{d^2q^i}{dt^2}(t),\ldots,
        \frac1{k!}\frac{d^kq^i}{dt^k}(t)
        \right).
\]

For $u\in T^{(k)}M$, every tangent vector $X\in T_uT^{(k)}M$ can be written locally as
\[
        X=
        X^{(0)i}\frac{\partial}{\partial q^{(0)i}}\Big\vert_u
        +\cdots+
        X^{(k)i}\frac{\partial}{\partial q^{(k)i}}\Big\vert_u.
\]
A family of curves $\gamma_s(t)=\gamma(s,t)$ on $M$ is said to be $k$-adapted to $X$ if the curve
        $s\longmapsto (j^k\gamma_s)(0)$
on $T^{(k)}M$ is adapted to $X$. That is,
\[
        (j^k\gamma_s)(0)\Big\vert_{s=0}=u,\qquad
        \frac{\partial}{\partial s}\Big\vert_{s=0}(j^k\gamma_s)(0)=X.
\]
In local coordinates this gives
\[
        X^{(\alpha)i}
        =
        \frac1{\alpha!}
        \frac{\partial^{\alpha+1}\gamma^i}{\partial s\,\partial t^\alpha}
        \Big\vert_{(s,t)=(0,0)},
        \qquad \alpha=0,\ldots,k.
\]

The almost-tangent structure on $T^{(k)}M$ is the $(1,1)$-tensor field given in coordinates by
\[
        J=
        \frac{\partial}{\partial q^{(1)i}}\otimes dq^{(0)i}
        +\cdots+
        \frac{\partial}{\partial q^{(k)i}}\otimes dq^{(k-1)i}.
\]
Thus
\[
        J(X)=
        X^{(0)i}\frac{\partial}{\partial q^{(1)i}}\Big\vert_u
        +\cdots+
        X^{(k-1)i}\frac{\partial}{\partial q^{(k)i}}\Big\vert_u.
\]
The endomorphism $J\vert_u:T_uT^{(k)}M\to T_uT^{(k)}M$ is nilpotent of index $k+1$, and
\[
        \operatorname{Im}(J^\alpha)=\ker(J^{k+1-\alpha}),\qquad \alpha=1,\ldots,k.
\]
We denote the \emph{vertical space of order} $\alpha$ at $u \in T^{(k)}M$ by 
$$V_\alpha(u) = \operatorname{Im}((J\vert_u)^\alpha) =
        \operatorname{span}\left\{
        \frac{\partial}{\partial q^{(\alpha)i}}\Big\vert_u,\ldots,
        \frac{\partial}{\partial q^{(k)i}}\Big\vert_u
        \right\}.$$  
In particular, $V_1(u)=\ker(\tau_{k*}\vert_u)$ is the vertical subspace with respect to $\tau_k$. It is easily seen that the vertical spaces satisfy the following chain of inclusions
\[
        0\subset V_k(u)\subset V_{k-1}(u)\subset\cdots\subset V_1(u).
\]

For $0\leq \alpha\leq k$, let
\[
        \overset{(k)}{\tau}_{\!\alpha}:T^{(k)}M\to T^{(\alpha)}M,\qquad
        \overset{(k)}{\tau}_{\!\alpha}(j^k_0q)=j^\alpha_0q,
\]
where $T^{(0)}M:=M$ and $\overset{(k)}{\tau}_0=\tau_k$.  Then the vertical bundle of the fibration
\[
        T^{(k)}M\stackrel{\overset{(k)}{\tau}_{\!\alpha}}{\longrightarrow}T^{(\alpha)}M
\]
is $V_{\alpha+1}$, for $\alpha=0,\ldots,k-1$, i.e.
\[
        V_{\alpha+1}(u)=\ker(\overset{(k)}{\tau}_{\!\alpha\ast}\vert_u).
\]
This gives exact sequences
\[
        0\longrightarrow V_{\alpha+1}
        \longrightarrow TT^{(k)}M
        \stackrel{\overset{(k)}{\tau}_{\!\alpha\ast}}{\longrightarrow}
        TT^{(\alpha)}M
        \longrightarrow 0,
        \qquad \alpha=0,\ldots,k-1.
\]

\subsection{Connection maps and multiconnections}

A nonlinear connection on $T^{(k)}M$ is a vector subbundle $H\subset TT^{(k)}M$, again refered to as the horizontal bundle, which is complementary to the vertical bundle $V:=V_1$
\begin{equation}\label{eq: nonlinear-connection-direct-sum}
        TT^{(k)}M=H\oplus V.
\end{equation}
As before, nonlinear connections can be equivalently described by connection maps.  A connection map on $T^{(k)}M$ is a $\tau_k$-morphism
\[
        K=(K_1,\ldots,K_k):TT^{(k)}M\to (TM)_\oplus^{k}
\]
such that
\begin{enumerate}
    \item $K_\alpha:T_uT^{(k)}M\to T_{\tau_k(u)}M$ is linear for $\alpha=1,\ldots,k$;
    \item $K_{\alpha+1}\circ J=K_\alpha$ for $\alpha=1,\ldots,k-1$;
    \item $K_1\circ J=\tau_{k*}$.
\end{enumerate}
Here, the notation $(TM)_\oplus^k$ refers to the Whitney sum of $k$ tangent bundles, i.e. $(TM)_\oplus^k = \bigoplus_{i=1}^k TM$.
The corresponding horizontal bundle is then given by $H=\ker(K)$, and
\begin{equation}\label{eq:directsum}
        TT^{(k)}M=\ker(K)\oplus \operatorname{Im}(J).
\end{equation}

In local coordinates a connection map is determined by $kn^2$ smooth functions $(K_\alpha)^i_j:T^{(k)}M\to\mathbb R$, called connection coefficients, through
\begin{equation}\label{eq: connection map coefficients}
        K_\alpha=
        \left(
        dq^{(\alpha)i}
        +(K_1)^i_jdq^{(\alpha-1)j}
        +\cdots+
        (K_\alpha)^i_jdq^{(0)j}
        \right)\otimes\partial_i,
        \qquad \alpha=1,\ldots,k.
\end{equation}
Equivalently,
\[
        K_\alpha\left(\frac{\partial}{\partial q^{(\alpha)i}}\Big\vert_u\right)=\partial_i\Big\vert_{\tau_k(u)},
        \qquad
        K_\alpha\left(\frac{\partial}{\partial q^{(\beta)i}}\Big\vert_u\right)=0
        \quad\text{for }\beta>\alpha.
\]

\begin{theorem}[Bucătaru, \cite{Buc1997}]\label{thm: nonlinear connection equivalence}
A nonlinear connection on $T^{(k)}M$ uniquely determines a connection map of order $k$ in $M$. Conversely, if
\[
        K:TT^{(k)}M\to (TM)_\oplus^k
\]
satisfies the three conditions above, then $\ker(K)$ is a nonlinear connection.
\end{theorem}

A connection map gives a canonical identification
\[
        \Phi:TT^{(k)}M\to (TM)_\oplus^{k+1},\qquad
        \Phi_u(X)=(\tau_{k*}X,K_1(X),\ldots,K_k(X)).
\]
For each $u\in T^{(k)}M$, the map $\Phi_u$ is a linear isomorphism.  Under this isomorphism, horizontal vectors are identified with vectors of the form $(v,0,\ldots,0)$, whereas vertical vectors with respect to $\tau_k$ are identified with vectors of the form $(0,v_1,\ldots,v_k)$.

Given a nonlinear connection $H$, the associated multiconnection is the family of subspaces
\[
        H_0(u),\ldots,H_{k-1}(u)\subset T_uT^{(k)}M
\]
defined by
\[
        H_0(u)=H(u),\qquad H_i(u)=J^i(H_0(u)),\quad i=1,\ldots,k-1.
\]

\begin{theorem}[Bucătaru, \cite{Buc2011}]\label{teo: decomp}
Suppose that $H$ is a nonlinear connection.  Then, for all $\alpha=1,\ldots,k$,
\begin{equation}\label{eq: multiconnection direct sum}
        TT^{(k)}M=\bigoplus_{i=0}^{\alpha-1}H_i\oplus V_\alpha.
\end{equation}
\end{theorem}
Taking $\alpha=k$ in Equation \eqref{eq: multiconnection direct sum}, every vector $X\in T_uT^{(k)}M$ can be uniquely decomposed as
\[
        X=h_0(X)+h_1(X)+\cdots+h_k(X),
\]
where $h_\alpha$ is the projection onto $H_\alpha$ for $\alpha=0,\ldots,k-1$, and $h_k=v_k$ is the projection onto $V_k$.  If
\[
        X=X_0+J(X_1)+\cdots+J^{k-1}(X_{k-1})+J^k(X_k),
        \qquad X_i\in H_0,
\]
then
\[
        \Phi_u(X)=\big(\tau_{k*}X_0,\tau_{k*}X_1,\ldots,\tau_{k*}X_k\big).
\]

For each $\alpha=1,\ldots,k$, define
\[
        N_\alpha=\bigoplus_{i=0}^{\alpha-1}H_i.
\]
Then
\begin{equation}\label{eq:horizontal_splittings}
    TT^{(k)}M=N_\alpha\oplus V_\alpha,
\end{equation}
        
where
\[
        N_\alpha=\ker(K_\alpha,\ldots,K_k),\qquad
        V_\alpha=\operatorname{Im}(J^\alpha)=\ker(\tau_{k*},K_1,\ldots,K_{\alpha-1}).
\]
A vector $X\in T_uT^{(k)}M$ is called $\alpha$-vertical if $X\in V_\alpha$, and $\alpha$-horizontal if $X\in N_\alpha$.  We denote by $n_\alpha$ and $v_\alpha$ the projections onto $N_\alpha$ and $V_\alpha$, respectively.  These definitions are extended to vector fields and curves in the usual way.

\subsection{Adapted bases}

Define covectors $\delta q^{(\alpha)i}$ by
\[
        \delta q^{(\alpha)i}
        =
        dq^{(\alpha)i}
        +(K_1)^i_jdq^{(\alpha-1)j}
        +\cdots+
        (K_\alpha)^i_jdq^{(0)j},
        \qquad \alpha=1,\ldots,k,
\]
and set $\delta q^{(0)i}=dq^{(0)i}$. We refer to the collection $\{\delta q^{(\alpha)i} \ \vert \ 0 \le \alpha \le k, \, 1 \le i \le n\}$ as the \textit{adapted basis}. With respect to the adapted basis, the connection map and canonical projection take the simplied forms: 
\[
        K_\alpha=\delta q^{(\alpha)i}\otimes \partial_i,\qquad
        \tau_{k*}=\delta q^{(0)i}\otimes\partial_i.
\]
Consider the local frame of vector fields $\displaystyle \left\{\frac{\delta}{\delta q^{(\alpha)i}}\right\}$ dual to the dapated basis. This frame is adapted to the decomposition given in Equation \eqref{eq: multiconnection direct sum} in the sense that, for each fixed $\alpha$, the vector fields $\displaystyle \left\{\frac{\delta}{\delta q^{(\alpha)i}}\right\}_{i=1}^n$
span $H_\alpha$, while
        $\displaystyle \left\{\frac{\partial}{{\partial q^{(k)i}}}\right\}_{i=1}^n$
span $V_k$.  With respect to these bases, the projections similarly take the form
\[
        h_\alpha=\delta q^{(\alpha)i}\otimes\frac{\delta}{\delta q^{(\alpha)i}},
\]
while the almost-tangent structure is given by
\[
        J=
        \frac{\delta}{\delta q^{(1)i}}\otimes\delta q^{(0)i}
        +\cdots+
        \frac{\delta}{\delta q^{(k)i}}\otimes\delta q^{(k-1)i}.
\]

The coordinate basis and adapted basis are related by dual coefficient functions $(C_\alpha)^i_j$.  The details are recorded in Appendix \ref{app:adapted-basis-details}; the formulas most often used later are
\begin{equation}\label{eq: coordinate basis to adapted basis}
        dq^{(\alpha)i}
        =
        \delta q^{(\alpha)i}
        -(C_1)^i_j\delta q^{(\alpha-1)j}
        -\cdots
        -(C_\alpha)^i_j\delta q^{(0)j},
        \qquad \alpha=1,\ldots,k,
\end{equation}
and
\begin{equation}\label{eq: coordinate basis to adapted basis vector field}
        \frac{\partial}{\partial q^{(\alpha)i}}
        =
        \frac{\delta}{\delta q^{(\alpha)i}}
        +(K_1)^j_i\frac{\delta}{\delta q^{(\alpha+1)j}}
        +\cdots+
        (K_{k-\alpha})^j_i\frac{\delta}{\delta q^{(k)j}}.
\end{equation}

\subsection{Connection towers}\label{Subsec: connection tower}

Throughout the remainder of the article, we use the notation $\overset{(\alpha)}{\cdot}$ to signify that an object is being considered with respect to $T^{(\alpha)}M$, with the exception of the projection $\tau_k$ and the coefficients of certain objects in local coordinates for simplicity. In some places, the overset notation may be dropped when it is sufficiently clear from context which space an object pertains to. 

The higher-order tangent bundle $T^{(k)}M$ carries a natural tower of fibrations
\begin{equation}\label{eq:fiber_bundle_tower}
        T^{(k)}M
        \stackrel{\overset{(k)}{\tau}_{\!k-1}}{\longrightarrow}
        T^{(k-1)}M
        \stackrel{\overset{(k-1)}{\tau}_{\!\!\!\!\!k-2}}{\longrightarrow}
        \cdots
        \stackrel{\overset{(2)}{\tau}_{\!1}}{\longrightarrow}
        TM
        \stackrel{\overset{(1)}{\tau}_{\!0}}{\longrightarrow}
        M.
\end{equation}
A standard nonlinear connection on $T^{(k)}M$ is adapted only to the projection onto the base manifold.  However, the iterated bundle structure suggests the existence of compatible connection data simultaneously adapted to all intermediate fibrations.  This leads to the following definition.

We say that a connection map $\overset{(k)}{K}=(\overset{(k)}{K}_1,\ldots,\overset{(k)}{K}_k)$ on $T^{(k)}M$ is a \emph{connection tower of order} $k$ provided that, for each $\alpha=1,\ldots,k$, there exists a unique map
\[
        \overset{(\alpha)}{K}_\alpha:TT^{(\alpha)}M\to TM
\]
such that
\begin{equation}\label{eq: tower}
        \overset{(k)}{K}_\alpha=\overset{(\alpha)}{K}_\alpha\circ \overset{(k)}{\tau}_{\!\alpha\ast}.
\end{equation}
Given a connection tower $\overset{(k)}{K}$ of order $k$, define
\[
        \overset{(\alpha)}{K}=(\overset{(\alpha)}{K}_1,\ldots,\overset{(\alpha)}{K}_\alpha),
        \qquad
        \overset{(\alpha)}{K}_\mu=\overset{(\mu)}{K}_\mu\circ \overset{(\alpha)}{\tau}_{\!\mu*},
        \quad \mu=1,\ldots,\alpha.
\]
Then $\overset{(\alpha)}{K}$ is a connection map on $T^{(\alpha)}M$ for each $\alpha=1,\ldots,k$.  By construction, $\overset{(k)}{K}_\mu$ and $\overset{(\alpha)}{K}_\mu$ have the same connection coefficients pulled back through $\overset{(k)}{\tau}_{\!\alpha}$.  In local coordinates, this compatibility means that
\[
        \frac{\partial (K_\alpha)^i_j}{\partial q^{(\mu)l}}=0
        \qquad\text{whenever }\mu>\alpha.
\]
Conversely, a connection map whose connection coefficients satisfy such a relation uniquely determines a connection tower.

The adapted bases on $T^{(k)}M$ and $T^{(\alpha)}M$ are naturally related through the projections.  For $\mu\leq \alpha$,
\[
        \overset{(k)}{\tau}_{\!\alpha\ast}
        \left(
        \frac{\overset{(k)}{\delta}}{\delta q^{(\mu)i}}
        \right)
        =
        \frac{\overset{(\alpha)}{\delta}}{\delta q^{(\mu)i}},
\]
whereas the adapted basis vectors with $\mu>\alpha$ are annihilated by $\overset{(k)}{\tau}_{\!\alpha\ast}$. Equivalently,
\[
        \overset{(k)}{\tau}_{\!\alpha\ast}
        =
        \frac{\overset{(\alpha)}{\delta}}{\delta q^{(0)i}}
        \otimes
        \overset{(k)}{\delta}q^{(0)i}
        +\cdots+
        \frac{\overset{(\alpha)}{\delta}}{\delta q^{(\alpha)i}}
        \otimes
        \overset{(k)}{\delta}q^{(\alpha)i}.
\] 

A connection tower gives the commutative diagram:
\[
\begin{tikzcd}
TT^{(k)}M  \arrow[r, "\overset{(k)}{\tau}_{\!k-1\ast}"] \arrow[d, "\overset{(k)}{K}"] &
TT^{(k-1)}M \arrow[r, "\overset{(k-1)}{\tau}_{\!\!\!\!\!k-2\ast}"]  \arrow[d, "\overset{(k-1)}{K}"] &
TT^{(k-2)}M \arrow[r, "\overset{(k-2)}{\tau}_{\!\!\!\!\!k-3\ast}"]  \arrow[d, "\overset{(k-2)}{K}"] &
\cdots \arrow[r, "\overset{(3)}{\tau}_{\!2\ast}"]  &
TT^{(2)}M \arrow[r, "\overset{(2)}{\tau}_{\!1\ast}"] \arrow[d, "\overset{(2)}{K}"]&
TTM  \arrow[d, "\overset{(1)}{K}"]   \\
(TM)_\oplus^k \arrow[r, "\overset{(k)}{\mathrm{pr}}_{k-1}"] &
(TM)_\oplus^{k-1} \arrow[r, "\overset{(k-1)}{\mathrm{pr}}_{\!\!\!\! k-2}"] &
(TM)_\oplus^{k-2} \arrow[r, "\overset{(k-2)}{\mathrm{pr}}_{\!\!\!\! k-3}"] &
\cdots \arrow[r, "\overset{(3)}{\mathrm{pr}}_{2}"] &
(TM)_\oplus^2  \arrow[r, "\overset{(2)}{\mathrm{pr}}_{1}"] &
TM
\end{tikzcd}
\]
where $\overset{(\alpha)}{\mathrm{pr}}_{\alpha-1}$ denotes the projection of $(TM)_\oplus^\alpha$ onto $(TM)_\oplus^{\alpha-1}$ obtained by dropping the last component.

Applying $\overset{(k)}{\tau}_{\!\alpha\ast}$ to the multiconnection decomposition on $T^{(k)}M$ gives the corresponding decomposition on $T^{(\alpha)}M$.  In particular, for all $\mu=0,\ldots,\alpha$,
\[
        \overset{(k)}{\tau}_{\!\alpha\ast}\big(\overset{(k)}{H}_\mu(u)\big)
        =
        \overset{(\alpha)}{H}_\mu(\overset{(k)}{\tau}_\alpha(u)), \qquad \overset{(k)}{\tau}_{\!\alpha\ast}\big(\overset{(k)}{V}_{\alpha+1}(u)\big)=0.
\]
Thus a connection tower induces a tower of $\alpha$-horizontal subbundles
\[
        \overset{(k)}{N}_k
        \stackrel{\overset{(k)}{\tau}_{\!k-1\ast}}{\longrightarrow}
        \overset{(k-1)}{N}_{\!\!\!\! k-1}
        \stackrel{\overset{(k-1)}{\tau}_{\!\!\!\!k-2\ast}}{\longrightarrow}
        \cdots
        \stackrel{\overset{(2)}{\tau}_{\!1\ast}}{\longrightarrow}
        \overset{(1)}{N}_1,
\]
where $\overset{(\alpha)}{N}_\alpha=\bigoplus_{i=0}^{\alpha-1}\overset{(\alpha)}{H}_i$.  This justifies the terminology connection tower.

\subsection{Vector bundle structures induced by connection towers}

Connection towers also induce vector bundle structures on the higher-order tangent bundles themselves. For \(\alpha\geq 1\), we define the \emph{Tulczyjew operator} $\overset{(\alpha)}{d_T}:T^{(\alpha)}M\longrightarrow TT^{(\alpha-1)}M$
by
\[
        \overset{(\alpha)}{d_T}(u)
        =
        \frac{d}{dt}\Big\vert_{t=0}(j^{\alpha-1}q)(t),
        \qquad
        u=j^\alpha_0(q),
\]
which is well-defined since \(\overset{(\alpha)}{d_T}(u)\) depends only on the \(\alpha\)-jet of \(q\) \cite{Tulczyjew1976}.  In local coordinates,
\[
        \overset{(\alpha)}{d_T}
        =
        q^{(1)i}\frac{\partial}{\partial q^{(0)i}}
        +2q^{(2)i}\frac{\partial}{\partial q^{(1)i}}
        +\cdots
        +\alpha q^{(\alpha)i}\frac{\partial}{\partial q^{(\alpha-1)i}}.
\]

\begin{theorem}\label{thm:connection-tower-vector-bundle}
Let \(\overset{(k)}{K}\) be a connection tower of order \(k\) in \(M\).  For each \(\alpha=1,\ldots,k\), let
\[
        \sigma:(TM)_\oplus^\alpha\to M
\]
denote the natural projection, and define
\[
        \overset{(\alpha)}{F}:T^{(\alpha)}M\to (TM)_\oplus^\alpha
\]
by
\begin{equation}\label{eq: F}
        \overset{(\alpha)}{F}(u)
        =
        \Big(
        \overset{(1)}{d_T}(\overset{(\alpha)}{\tau}_1(u)),\,
        \overset{(1)}{K}_1\big(\overset{(2)}{d_T}(\overset{(\alpha)}{\tau}_2(u))\big),\,
        2!\,\overset{(2)}{K}_2\big(\overset{(3)}{d_T}(\overset{(\alpha)}{\tau}_3(u))\big),\,
        \ldots,\,
        (\alpha-1)!\,\overset{(\alpha-1)}{K}_{\alpha-1}
        \big(\overset{(\alpha)}{d_T}(u)\big)
        \Big).
\end{equation}
Then \(\overset{(\alpha)}{F}\) is a diffeomorphism for all \(\alpha=1,\ldots,k\).  Moreover,
\[
        \sigma\circ \overset{(\alpha)}{F}=\tau_\alpha .
\]
Consequently, the vector bundle structure of \((TM)_\oplus^\alpha\) over \(M\) pulls back through
\(\overset{(\alpha)}{F}\) to a vector bundle structure on \(T^{(\alpha)}M\).
\end{theorem}

\begin{proof}
Let
\[
        (X_0,\ldots,X_{\alpha-1})\in (TM)_\oplus^\alpha
\]
be a point over \(p\in M\).  In local coordinates, the equation
\[
        \overset{(\alpha)}{F}(u)=(X_0,\ldots,X_{\alpha-1})
\]
is triangular.  Indeed, if \(u=(q^{(0)},\ldots,q^{(\alpha)})\), then
        $X_0^i=q^{(1)i}$,
and for \(\mu=1,\ldots,\alpha-1\),
\[
        X_\mu^i
        =
        \mu!\left(
        (\mu+1)q^{(\mu+1)i}
        +\mu q^{(\mu)j}(K_1)^i_j
        +\cdots
        +q^{(1)j}(K_\mu)^i_j
        \right).
\]
Here \((K_l)^i_j\), for \(l=1,\ldots,\mu\), denotes the corresponding coefficient of the connection map on
\(T^{(\mu)}M\), pulled back to \(T^{(\alpha)}M\) through \(\overset{(\alpha)}{\tau}_\mu\).  The first equations are
\[
        X_0^i=q^{(1)i},
\]
\[
        X_1^i=2q^{(2)i}+q^{(1)j}(K_1)^i_j,
\]
\[
        X_2^i=
        2!\left(
        3q^{(3)i}
        +2q^{(2)j}(K_1)^i_j
        +q^{(1)j}(K_2)^i_j
        \right).
\]
Since the coefficient of \(q^{(\mu+1)i}\) in the \(\mu\)th equation is \((\mu+1)!\), these equations can be solved recursively:
\[
        q^{(1)i}=X_0^i,
\]
\[
        q^{(2)i}
        =
        \frac{1}{2}
        \left(
        X_1^i-(K_1)^i_jX_0^j
        \right),
\]
and, in general,
\[
        q^{(\mu+1)i}
        =
        \frac{1}{(\mu+1)!}
        \left(
        X_\mu^i-H_{\mu+1}^i(X_0,\ldots,X_{\mu-1})
        \right)
\]
for smooth functions \(H_{\mu+1}^i\) determined by the connection coefficients and the previously solved coordinates.
Hence the fiber coordinates
        $q^{(1)i},\ldots,q^{(\alpha)i}$
are uniquely and smoothly determined by \((X_0,\ldots,X_{\alpha-1})\), while the base coordinates \(q^{(0)i}\) are
determined by the common base point \(p\).  Therefore \(\overset{(\alpha)}{F}\) is a bijective local diffeomorphism, hence a diffeomorphism.

Finally, each component of \(\overset{(\alpha)}{F}(u)\) lies in \(T_{\tau_\alpha(u)}M\), so
        $\sigma\circ \overset{(\alpha)}{F}=\tau_\alpha $.
The vector bundle structure of \((TM)_\oplus^\alpha\) over \(M\) therefore pulls back through \(\overset{(\alpha)}{F}\) to a vector bundle
structure on \(T^{(\alpha)}M\).
\qed
\end{proof}

\subsection{Lifting vector fields on $M$ to higher-order tangent bundles}\label{subsec lifts}

Let $\overset{(k)}{K}$ be a connection tower of order $k$.  By Theorem \ref{teo: decomp},
\begin{equation}\label{eq: multiconnection direct sum2}
        TT^{(k)}M=\bigoplus_{\alpha=0}^{k-1}\overset{(k)}{H}_\alpha\oplus \overset{(k)}{V}_k.
\end{equation}
We use the convention
\[
        \overset{(k)}{K}_0:=\tau_{k*},\qquad \overset{(k)}{h}_k:=\overset{(k)}{v}_k.
\]
Let $z\in T_pM$ and let $u\in T^{(k)}M$ with $\tau_k(u)=p$.  For each $\alpha=0,\ldots,k$, there exists a unique vector
\[
        z^{\overset{(k)}{h}_\alpha}\in T_uT^{(k)}M
\]
such that
\[
        \overset{(k)}{K}_\alpha(z^{\overset{(k)}{h}_\alpha})=z,\qquad
        \overset{(k)}{K}_\beta(z^{\overset{(k)}{h}_\alpha})=0
        \quad\text{for all }\beta\neq\alpha.
\]
For $\alpha=0,\ldots,k-1$, the vector $z^{\overset{(k)}{h}_\alpha}$ belongs to $\overset{(k)}{H}_\alpha$ and is called the \emph{$\overset{(k)}{h}_\alpha$-lift} of $z$ to $T^{(k)}M$ at $u$. In particular, we see that $z^{\overset{(k)}{h}_0}$ is $1$-horizontal while $z^{\overset{(k)}{v}_k}$ is $k$-vertical, so that we alternatively refer to the $\overset{(k)}{h}_0$-lift as the \emph{$1$-horizontal lift} and the $\overset{(k)}{v}_k$-lift as the \emph{$k$-vertical lift}. 

Similarly, for vector fields $X\in\mathfrak X(M)$, we define the $\overset{(k)}{h}_\alpha$-lift of $X$ as the vector field $X^{\overset{(k)}{h}_\alpha}\in\mathfrak X(T^{(k)}M)$ satisfying
\[
        \overset{(k)}{K}_\alpha\circ X^{\overset{(k)}{h}_\alpha}=X\circ \tau_k,\qquad
        \overset{(k)}{K}_\beta\circ X^{\overset{(k)}{h}_\alpha}=0
        \quad\text{for }\beta\neq\alpha.
\]
As before, we refer to the $\overset{(k)}{v}_k$-lift of $X$ as the $k$-vertical lift, and the $\overset{(k)}{h}_0$ lift as the $1$-horizontal lift. The adapted basis is described by lifts of coordinate vector fields
\[
        \frac{\overset{(k)}{\delta}}{\delta q^{(\alpha)j}}
        =
        (\partial_j)^{\overset{(k)}{h}_\alpha},
        \qquad \alpha=0,\ldots,k.
\]

\begin{lemma}\label{lemma: fX}\label{lemma: prop}
Let \(X,Y\in\mathfrak X(M)\) and \(f\in C^\infty(M)\).  Then
\[
        (X+Y)^{\overset{(k)}{h}_\alpha}=X^{\overset{(k)}{h}_\alpha}+Y^{\overset{(k)}{h}_\alpha},
        \qquad
        (fX)^{\overset{(k)}{h}_\alpha}=(f\circ\tau_k)X^{\overset{(k)}{h}_\alpha},
\]
for \(\alpha=0,\ldots,k\).  If \(X=X^j\partial_j\), then
\[
        X^{\overset{(k)}{h}_\alpha}=(X^j\circ\tau_k)\frac{\delta}{\delta q^{(\alpha)j}},
        \qquad \alpha=0,\ldots,k.
\]
Moreover,
\[
        X^{\overset{(k)}{h}_0}(f\circ\tau_k)=(Xf)\circ\tau_k,
        \qquad
        X^{\overset{(k)}{h}_\alpha}(f\circ\tau_k)=0,
        \quad \alpha=1,\ldots,k.
\]
\end{lemma}

\begin{proof}
    See Appendix \ref{app:lift-technical-lemmas}.
    \qed
\end{proof}

Every vector field $W$ on $T^{(k)}M$ admits the decomposition
\begin{equation}\label{eq: decomposition VF}
        W
        =
        (\tau_{k*}\circ W)^{\overset{(k)}{h}_0}
        +(\overset{(k)}{K}_1\circ W)^{\overset{(k)}{h}_1}
        +\cdots
        +(\overset{(k)}{K}_{k-1}\circ W)^{\overset{(k)}{h}_{k-1}}
        +(\overset{(k)}{K}_k\circ W)^{\overset{(k)}{v}_k}.
\end{equation}
Similarly, if $\Gamma$ is a curve on $T^{(k)}M$ and $q=\tau_k\circ\Gamma$, then
\[
        \dot q=\tau_{k*}\circ\dot\Gamma,\qquad
        Y^{(\alpha)}=\overset{(k)}{K}_\alpha\circ\dot\Gamma,\quad \alpha=1,\ldots,k,
\]
so that
\begin{equation}\label{eq: decomposition gamma T^k}
        \dot\Gamma
        =
        \dot q^{\overset{(k)}{h}_0}
        +(Y^{(1)})^{\overset{(k)}{h}_1}
        +\cdots
        +(Y^{(k-1)})^{\overset{(k)}{h}_{k-1}}
        +(Y^{(k)})^{\overset{(k)}{v}_k}.
\end{equation}
The curve $\Gamma$ is $\alpha$-horizontal if and only if
\[
        Y^{(\beta)}=0,\qquad \beta=\alpha,\ldots,k,
\]
and is $\alpha$-vertical if and only if
\[
        Y^{(\beta)}=0,\qquad \beta=1,\ldots,\alpha-1.
\]
A large structural advantage to connection towers is that order-$k$ computations can be reused in all higher orders, which is described by the following lemmas.
\begin{lemma}\label{lemma: projected lifts}
Let $X\in\mathfrak X(M)$.  Then $X^{\overset{(k)}{h}_\alpha}$ is $\overset{(k)}{\tau}_{\!k-1}$-related to $X^{\!\!\overset{(k-1)}{h}_{\!\!\!\alpha}}$:
\begin{equation}\label{eqtt}
        X^{\!\!\overset{(k-1)}{h}_{\!\!\!\alpha}}\, \circ\, \overset{(k)}{\tau}_{\!k-1}
        =
        \overset{(k)}{\tau}_{\!k-1\ast} \, \circ \,X^{\overset{(k)}{h}_\alpha},
        \qquad \alpha=0,\ldots,k-1.
\end{equation}
\end{lemma}

\begin{proof}
This follows immediately from the connection tower compatibility
\begin{equation}\label{eqt}
        \overset{(k-1)}{K}_{\!\!\!\alpha}\circ\overset{(k)}{\tau}_{\!k-1\ast}=\overset{(k)}{K}_\alpha,
        \qquad \alpha=0,\ldots,k-1.
\end{equation}
\qed
\end{proof}
Moreover, for the Lie brackets of lifts, we have the following lemma.
\begin{lemma}\label{lemma: K mu}
For each $X,Y\in\mathfrak X(M)$,
\[
        \overset{(k)}{K}_\mu\circ[X^{\overset{(k)}{h}_\alpha},Y^{\overset{(k)}{h}_\beta}]
        =
        \left(
        \overset{(k-1)}{K}_{\!\!\!\!\mu}\circ[X^{\overset{(k-1)}{h}_{\!\!\!\alpha}},Y^{\overset{(k-1)}{h}_{\!\!\!\beta}}]
        \right)\circ\overset{(k)}{\tau}_{\!k-1},
\]
for $\alpha,\beta,\mu=0,\ldots,k-1$.
\end{lemma}
\begin{proof}
By Lemma~\ref{lemma: projected lifts}, the lifted vector fields
$X^{\overset{(k)}{h}_\alpha}$ and $Y^{\overset{(k)}{h}_\beta}$ are $\overset{(k)}{\tau}_{\!k-1}$-related to
$X^{\!\!\overset{(k-1)}{h}_{\!\!\!\alpha}}$ and $Y^{\!\!\overset{(k-1)}{h}_{\!\!\!\!\beta}}$, respectively.  Their brackets are
therefore also $\overset{(k)}{\tau}_{\!k-1}$-related.  Applying the connection tower compatibility
$\overset{(k-1)}{K}_{\!\!\!\mu}\circ\overset{(k)}{\tau}_{\!k-1\ast}=\overset{(k)}{K}_\mu$ gives the identity. \qed
\end{proof}
In particular, the $\overset{(k)}{h}_\mu$-components of the Lie brackets $[X^{\overset{(k)}{h}_\alpha},Y^{\overset{(k)}{h}_\beta}]$ with $\mu<k$ are obtained from the corresponding brackets on $T^{(k-1)}M$, so that the genuinely new information in order $k$ is the $k$-vertical components.
Hence, if
\[
        \left[
        \frac{\overset{(k)}{\delta}}{\delta q^{(\alpha)i}},
        \frac{\overset{(k)}{\delta}}{\delta q^{(\beta)l}}
        \right]
        =
        \sum_{j=1}^n\sum_{\mu=0}^k
        {}^kS^{(\mu)j}_{(\alpha)i,(\beta)l}
        \frac{\overset{(k)}{\delta}}{\delta q^{(\mu)j}},
\]
then, for $\mu=0,\ldots,k-1$,
\[
        {}^kS^{(\mu)j}_{(\alpha)i,(\beta)l}
        =
        {}^{k-1}S^{(\mu)j}_{(\alpha)i,(\beta)l}\circ\overset{(k)}{\tau}_{\!k-1},
\]
and only the coefficients ${}^kS^{(k)j}_{(\alpha)i,(\beta)l}$ must be computed anew.

The Lie brackets of lifted vector fields will be studied in more detail in the next section.  

\section{Connection maps on $T^{(k)}M$ induced by the Riemannian metric}
\label{sec:riemannian-induced-connection-tower}

In this section, we introduce a particular connection tower on $T^{(k)}M$ induced by the Riemannian metric $g$ on $M$.  This construction extends the Dombrowski splitting of $TTM$ to higher-order tangent bundles. In Section~\ref{sec:generalized-sasaki-metrics}, we use this tower to orthogonalize the adapted splitting and thereby obtain a family of higher-order Sasaki metrics, recovering the classical Sasaki metric when $k=1$.

The curvature identities used in this section are recorded in Appendix~\ref{app:riemannian-identities}.  We keep the construction of the Levi--Civita-induced connection tower and the full second-order bracket verification in the main text, since these are the basic computations on which the later formulas rest.  The longer third-order bracket computations and the general coordinate reduction for the $k$-vertical bracket components are collected in Appendix~\ref{app:bracket-computations}.

\subsection{Riemannian conventions}\label{Subsec: Riem geo}

Consider a Riemannian manifold $(M,\langle\cdot,\cdot\rangle)$, equipped with its Levi--Civita connection $\nabla$.  Thus $\nabla$ is torsion-free and compatible with the metric:
\begin{enumerate}
    \item $\nabla_XY-\nabla_YX=[X,Y]$,
    \item $X\langle Y,Z\rangle=\langle\nabla_XY,Z\rangle+\langle Y,\nabla_XZ\rangle$, for all $X,Y,Z\in\mathfrak X(M)$.
\end{enumerate}
The Levi-Civita connection is uniquely defined by the \emph{Koszul formula}
\begin{equation}\label{eq: Koszul}
\begin{aligned}
        2\langle\nabla_XY,Z\rangle
        &=X\langle Y,Z\rangle+Y\langle X,Z\rangle-Z\langle X,Y\rangle\\
        &\quad+\langle[X,Y],Z\rangle-\langle[X,Z],Y\rangle-\langle[Y,Z],X\rangle,
\end{aligned}
\end{equation}
for all $X,Y,Z\in\mathfrak X(T^{(k)}M)$.

We denote by $\nabla Y$ the linear map defined by $(\nabla Y)(X)=\nabla_XY$, and we define the second covariant derivative by
\[
        \nabla^2_{XY}=\nabla_X\nabla_Y-\nabla_{\nabla_XY}.
\]
The curvature tensor is defined by
\begin{equation}\label{eq:rg1}
        R(X,Y)Z=\nabla_X\nabla_YZ-\nabla_Y\nabla_XZ-\nabla_{[X,Y]}Z.
\end{equation}
We also write $R(X,Y)$ for the endomorphism $Z\mapsto R(X,Y)Z$.

We use the standard curvature symmetries, the corresponding identities for $\nabla R$ and $\nabla^2R$, and the basic commutation formula for covariant derivatives along a two-parameter family of curves.  These identities are listed in Appendix~\ref{app:riemannian-identities}.  Throughout the paper, $\Gamma^k_{ij}$ denotes the Christoffel symbols of $\nabla$, so that $\nabla_{\partial_i}\partial_j=\Gamma^k_{ij}\partial_k$, and $R^l_{ijk}$ denotes the components of the curvature tensor, defined by $R(\partial_i,\partial_j)\partial_k=R^l_{ijk}\partial_l$.
The curvature tensor components of the Levi--Civita connection can be expressed in terms of the Christoffel symbols via
\begin{equation}\label{eq:rg3}
        R^i_{jkl}=\partial_j\Gamma^i_{kl}-\partial_k\Gamma^i_{jl}
        +\Gamma^m_{kl}\Gamma^i_{jm}-\Gamma^m_{jl}\Gamma^i_{km}.
\end{equation}
The components of the covariant derivative of the curvature tensor are denoted by $R^m_{i;jkl}$ and are given by
\[
        (\nabla_{\partial_i}R)(\partial_j,\partial_k)\partial_l=R^m_{i;jkl}\partial_m.
\]

\subsection{Connection map on $TM$ induced by the Riemannian structure}\label{subsec TM 2}

Given a Riemannian metric $g$ on $M$, there is a canonical choice for a connection map on $TM$.  The Levi--Civita connection of $(M,g)$ induces the connection map ${K}:TTM\to TM$ given by
\begin{equation}\label{eq: connection_map_TM2}
        {K}(X)=\nabla_{\frac{\partial\gamma}{\partial t}}\frac{\partial\gamma}{\partial s}\Big\vert_{(s,t)=(0,0)},
\end{equation}
where $\gamma_s:t\mapsto\gamma_s(t)=\gamma(s,t)$ is a family of curves on $M$ adapted to $X\in TTM$.  The connection map is well-defined, since ${K}(X)$ does not depend on the particular choice of $\gamma$.  The connection coefficients of ${K}$ are given by $K^i_j=q^{(1)l}\Gamma^i_{jl}$, and hence
\[
        {K}=\left(dq^{(1)i}+q^{(1)l}\Gamma^i_{jl}dq^{(0)j}\right)\otimes\partial_i.
\]

The nonlinear connection ${H}=\ker({K})$ leads to the decomposition given in Equation \eqref{eq: direct_sum_TM}.  To standardize the notation with the general case on $T^{(k)}M$, we write $\overset{(1)}{H}_0={H}$ and $\overset{(1)}{V}_1={V}$.  The horizontal and vertical lifts of a vector field $X$ on $M$ are denoted by $X^{\overset{(1)}{h}_0}$ and $X^{\overset{(1)}{v}_1}$, respectively.

\begin{theorem}[\cite{Dombrowski1962}]\label{k1}
Let $X, Y \in \mathfrak{X}(M)$. Then
\begin{align*}
[X^{\overset{(1)}{h}_0}, Y^{\overset{(1)}{v}_1}] &= (\nabla_X Y)^{\overset{(1)}{v}_1}, \quad [X^{\overset{(1)}{v}_1}, Y^{\overset{(1)}{v}_1}] = 0, \quad  [X^{\overset{(1)}{h}_0}, Y^{\overset{(1)}{h}_0}]   = [X, Y]^{\overset{(1)}{h}_0} -\left(R(X,Y)u\right)^{\overset{(1)}{v}_1},
    \end{align*}
   with $u\in TM$.
    \end{theorem}

\begin{remark}
The last equation is written in abbreviated form.  Precisely, if $u\in T_pM$, then
\[
        [X^{\overset{(1)}{h}_0},Y^{\overset{(1)}{h}_0}]_u=[X,Y]^{\overset{(1)}{h}_0}_u-\big(R(X_p,Y_p)u\big)^{\overset{(1)}{v}_1}_u.
\]
We use the same convention throughout the paper for tensorial expressions evaluated at the point $u\in T^{(k)}M$.
\end{remark}

\subsection{The Levi--Civita-induced connection tower}\label{subsec:LC-connection-tower}

The goal of this subsection is to define a connection tower depending only on the Riemannian metric \(g\), as happens with the connection map \(\overset{(1)}{K}\) on \(TM\).  This tower should be related to the connection map on \(TM\) by the rule
\[
        \overset{(\alpha)}{K}_1=\overset{(1)}{K}\circ\overset{(\alpha)}{\tau}_{\!1*}.
\]
Equivalently, the connection maps \(\overset{(\alpha)}{K}=(\overset{(\alpha)}{K}_1,\ldots,\overset{(\alpha)}{K}_\alpha)\) on \(T^{(\alpha)}M\) should satisfy
\begin{equation}\label{eq: t}
        \overset{(\alpha)}{K}_\beta=\overset{(\alpha-1)}{K}_\beta\circ\overset{(\alpha)}{\tau}_{\!\alpha-1*},
        \qquad \beta=1,\ldots,\alpha-1,
\end{equation}
for \(\alpha=2,\ldots,k\), starting with \(\overset{(1)}{K}_1=\overset{(1)}{K}\).  Thus, it remains to specify the highest component \(\overset{(\alpha)}{K}_\alpha\) at each order.

Let \(u\in T^{(k)}M\) and \(X\in T_uT^{(k)}M\).  Choose a \(k\)-adapted family of curves \(\gamma_s(t)=\gamma(s,t)\) on \(M\), so that
\[
        (j^k\gamma_s)(0)\big\vert_{s=0}=u,
        \qquad
        \frac{\partial}{\partial s}\Big\vert_{s=0}(j^k\gamma_s)(0)=X.
\]
For \(\alpha=1,\ldots,k\), define
\begin{equation}\label{eq: LC tower K alpha}
        \overset{(k)}{K}_\alpha(X)
        =
        \frac1{\alpha!}\nabla^\alpha_{\frac{\partial\gamma}{\partial t}}
        \frac{\partial\gamma}{\partial s}
        \Big\vert_{(s,t)=(0,0)}.
\end{equation}
Here \(\nabla^\alpha_{\frac{\partial\gamma}{\partial t}}\) denotes repeated covariant differentiation along the \(t\)-curves. The main result of this subsection is the following theorem.

\begin{theorem}\label{thm:LC-induced-connection-tower}
The map \(\overset{(k)}{K}=(\overset{(k)}{K}_1,\ldots,\overset{(k)}{K}_k)\) defined by Equation \eqref{eq: LC tower K alpha} forms a connection map on \(T^{(k)}M\).  Moreover, the family of maps \(\{\overset{(\alpha)}{K}\}_{\alpha=1}^k\) obtained by applying the same construction at each order is a connection tower.
\end{theorem}

The proof proceeds in three steps.  We first show that each component of \(\overset{(k)}{K}\) is a well-defined linear map with the required normalization on vertical directions.  The same computation also identifies the corresponding connection coefficients.  We then record the resulting recursion among these coefficients, and finally use this recursion to prove that \(\overset{(k)}{K}\) has the local form of a connection map and is compatible with the tower projections.

\begin{lemma}\label{lemma: LC tower preliminary conditions}
The map \(\overset{(k)}{K}=(\overset{(k)}{K}_1,\ldots,\overset{(k)}{K}_k)\) satisfies the following conditions:
\begin{enumerate}
    \item \(\overset{(k)}{K}_{\alpha}:T_uT^{(k)}M\to T_{\tau_k(u)}M\) is linear for all \(\alpha=1,\ldots,k\).
    \item \(\overset{(k)}{K}_{\alpha}\circ J^\alpha=\tau_{k*}\) for all \(\alpha=1,\ldots,k\).
\end{enumerate}
\end{lemma}

\begin{proof}
We define the connection coefficients $(K_\alpha)_{(\mu)j}^i$ of $K_\alpha$ implicitly by $K_\alpha\left(\frac{\partial}{\partial q^{(\mu)i}}\right) = (K_\alpha)_{(\mu)j}^i \partial_j$ for all $\mu = 0, \dots, k$ and $i,j = 1, \dots, n$.

We prove by induction on \(\alpha\) that
\[
        \nabla^\alpha_{\frac{\partial \gamma}{\partial t}}
        \frac{\partial \gamma}{\partial s}
        =
        \alpha!
        \sum_{\mu=0}^{\alpha}
        (L_\alpha)^i_{(\mu)j}
        \frac1{\mu!}
        \frac{\partial^{\mu+1}\gamma^j}{\partial s\partial t^\mu}
        \partial_i,
\]
where, for \(\alpha=1,\ldots,k\),
\[
        (L_\alpha)^i_{(\alpha)j}=\delta^i_j,
        \qquad
        (L_1)^i_{(0)j}
        =
        \frac{\partial \gamma^l}{\partial t}\Gamma^i_{lj}.
\]
For \(\alpha=1\), we have
\[
        \nabla_{\frac{\partial \gamma}{\partial t}}
        \frac{\partial \gamma}{\partial s}
        =
        \left(
        \frac{\partial^2 \gamma^i}{\partial s\partial t}
        +
        \frac{\partial \gamma^l}{\partial t}
        \frac{\partial \gamma^j}{\partial s}
        \Gamma^i_{lj}
        \right)\partial_i,
\]
so the asserted form follows immediately.

Now suppose that the formula holds for some fixed \(\alpha-1<k\), so that
\[
        \nabla^{\alpha-1}_{\frac{\partial \gamma}{\partial t}}
        \frac{\partial \gamma}{\partial s}
        =
        (\alpha-1)!
        \sum_{\mu=0}^{\alpha-1}
        (L_{\alpha-1})^i_{(\mu)j}
        \frac1{\mu!}
        \frac{\partial^{\mu+1}\gamma^j}{\partial s\partial t^\mu}
        \partial_i,
\]
with \((L_{\alpha-1})^i_{(\alpha-1)j}=\delta^i_j\).  Applying one more covariant derivative in the \(t\)-direction gives
\[
\begin{aligned}
        \nabla^\alpha_{\frac{\partial \gamma}{\partial t}}
        \frac{\partial \gamma}{\partial s}
        &=
        (\alpha-1)!
        \sum_{\mu=0}^{\alpha-1}
        \frac1{\mu!}
        \Bigg[
        \frac{\partial (L_{\alpha-1})^i_{(\mu)j}}{\partial t}
        \frac{\partial^{\mu+1}\gamma^j}{\partial s\partial t^\mu}
        +(L_{\alpha-1})^i_{(\mu)j}
        \frac{\partial^{\mu+2}\gamma^j}{\partial s\partial t^{\mu+1}}
        +(L_{\alpha-1})^m_{(\mu)j}
        \frac{\partial^{\mu+1}\gamma^j}{\partial s\partial t^\mu}
        \frac{\partial \gamma^l}{\partial t}\Gamma^i_{lm}
        \Bigg]\partial_i
        \\
        &=
        \frac{\partial^{\alpha+1}\gamma^i}{\partial s\partial t^\alpha}\partial_i
        +(\alpha-1)!
        \sum_{\mu=0}^{\alpha-1}
        \frac1{\mu!}
        \Bigg[
        \frac{\partial (L_{\alpha-1})^i_{(\mu)j}}{\partial t}
        +\mu(L_{\alpha-1})^i_{(\mu-1)j}
        +(L_{\alpha-1})^m_{(\mu)j}
        \frac{\partial \gamma^l}{\partial t}\Gamma^i_{lm}
        \Bigg]
        \frac{\partial^{\mu+1}\gamma^j}{\partial s\partial t^\mu}
        \partial_i
        \\
        &=
        \frac{\partial^{\alpha+1}\gamma^i}{\partial s\partial t^\alpha}\partial_i 
        +(\alpha-1)!
        \sum_{\mu=0}^{\alpha-1}
        \frac1{\mu!}
        \Bigg[
        \frac{\partial (L_{\alpha-1})^i_{(\mu)j}}{\partial t}
        +(L_{\alpha-1})^m_{(\mu)j}(L_1)^i_{(0)m}
        +\mu(L_{\alpha-1})^i_{(\mu-1)j}
        \Bigg]
        \frac{\partial^{\mu+1}\gamma^j}{\partial s\partial t^\mu}
        \partial_i .
\end{aligned}
\]
Comparison with the asserted form at order \(\alpha\) gives the induction step.

Evaluating at \((s,t)=(0,0)\), we obtain
\begin{equation}\label{eq: K and Lalpha}
        \overset{(k)}{K}_\alpha(X)
        =
        \sum_{\mu=0}^{\alpha}
        (L_\alpha)^i_{(\mu)j}\Big\vert_{(s,t)=(0,0)}
        X^{(\mu)j}\partial_i
\end{equation}
Moreover,
\[
\begin{aligned}
        (L_\alpha)^i_{(\mu)j}\Big\vert_{(s,t)=(0,0)}
        &=
        \frac1\alpha\left[
        \sum_{\nu=0}^{\alpha-1}
        (\nu+1)q^{(\nu+1)m}
        \frac{\partial (L_{\alpha-1})^i_{(\mu)j}}{\partial q^{(\nu)m}}
        +\mu(L_{\alpha-1})^i_{(\mu-1)j}
        +(L_{\alpha-1})^m_{(\mu)j}(L_1)^i_{(0)m}
        \right]
        \\
        &=
        \frac1\alpha\left[
        d_T(L_{\alpha-1})^i_{(\mu)j}
        +(L_{\alpha-1})^m_{(\mu)j}(L_1)^i_{(0)m}
        +\mu(L_{\alpha-1})^i_{(\mu-1)j}
        \right].
\end{aligned}
\]
Thus the coefficients \((L_\alpha)^i_{(\mu)j}|_{(s,t)=(0,0)}\) depend only on \(u\), so that the maps
\[
        \overset{(k)}{K}_\alpha:T_uT^{(k)}M\to T_{\tau_k(u)}M
\]
are well-defined and linear.  In coordinates, the connection coefficients satisfy
\[
        (K_\alpha)^i_{(\mu)j}
        =
        (L_\alpha)^i_{(\mu)j}\Big\vert_{(s,t)=(0,0)}
        \quad\text{for }\mu\leq\alpha, \qquad (K_\alpha)^i_{(\mu)j} = 0 \quad\text{for }\mu > \alpha
\]
Since \((L_\alpha)^i_{(\alpha)j}=\delta^i_j\), it follows that $(K_\alpha)^i_{(\alpha)j}=\delta^i_j$.
Therefore \(\overset{(k)}{K}_\alpha\circ J^\alpha=\tau_{k*}\), proving the second condition.
\qed
\end{proof}

The preceding computation also gives the following recursion satisfied by the connection coefficients.

\begin{lemma}\label{lemma: recursive connection coef}
The connection coefficients of \(\overset{(k)}{K}\) satisfy
\begin{equation}\label{eq: semispray pre-connection coefficients recursion}
        (K_\alpha)^i_{(\mu)j}
        =
        \frac1\alpha\left[
        d_T(K_{\alpha-1})^i_{(\mu)j}
        +(K_{\alpha-1})^m_{(\mu)j}(K_1)^i_{(0)m}
        +\mu(K_{\alpha-1})^i_{(\mu-1)j}
        \right],
\end{equation}
for \(\alpha=2,\ldots,k\), \(\mu=0,\ldots,\alpha-1\), and \(1\leq i,j\leq n\), where
\[
        (K_{\alpha-1})^i_{(-1)j}:=0,
        \qquad
        (K_1)^i_{(0)j}=q^{(1)l}\Gamma^i_{lj}.
\]
\end{lemma}

\begin{proof}
By Equation \eqref{eq: K and Lalpha}, the coefficients of \(\overset{(k)}{K}_\alpha\) are obtained by evaluating the coefficients \((L_\alpha)^i_{(\mu)j}\) at \((s,t)=(0,0)\).  The final recursion derived in the proof of Lemma~\ref{lemma: LC tower preliminary conditions} therefore gives
\[
        (K_\alpha)^i_{(\mu)j}
        =
        \frac1\alpha\left[
        d_T(K_{\alpha-1})^i_{(\mu)j}
        +(K_{\alpha-1})^m_{(\mu)j}(K_1)^i_{(0)m}
        +\mu(K_{\alpha-1})^i_{(\mu-1)j}
        \right],
\]
for \(\alpha=2,\ldots,k\) and \(\mu=0,\ldots,\alpha-1\), with the convention \((K_{\alpha-1})^i_{(-1)j}=0\).  The expression for \((K_1)^i_{(0)j}\) follows from the base case \(\alpha=1\).
\qed
\end{proof}

In particular, for \((K_\alpha)^i_j=(K_\alpha)^i_{(0)j}\), we have
\begin{equation}\label{eq: recursive coef K}
        (K_\alpha)^i_j
        =
        \frac1\alpha\left(
        d_T(K_{\alpha-1})^i_j
        +(K_{\alpha-1})^l_j(K_1)^i_l
        \right),
        \qquad \alpha=2,\ldots,k.
\end{equation}
This relation was introduced by Miron in \cite[Theorem 9.1.1]{MironBook1997}, where it was shown to describe a connection map for any smooth choice of \((K_1)^i_j\) satisfying the corresponding transformation laws under changes of coordinates.

We now complete the proof of Theorem~\ref{thm:LC-induced-connection-tower}.

\begin{proof}[Proof of Theorem \ref{thm:LC-induced-connection-tower}]
By Lemma~\ref{lemma: LC tower preliminary conditions}, each
\[
        \overset{(k)}{K}_\alpha:T_uT^{(k)}M\to T_{\tau_k(u)}M
\]
is well-defined and linear, and satisfies \(\overset{(k)}{K}_{\alpha}\circ J^\alpha=\tau_{k*}\).  It remains to prove that \(\overset{(k)}{K}=(\overset{(k)}{K}_1,\ldots,\overset{(k)}{K}_k)\) has the local form of a connection map and that the resulting family is compatible with the tower projections.

For \(\alpha=1,\ldots,k\), let \(P(\alpha)\) denote the statement
\begin{equation}\label{eq: LC-tower-coefficient-compatibility}
        (K_\alpha)^i_{(\alpha-\beta)j}
        =
        (K_\beta)^i_{(0)j},
        \qquad
        \beta=1,\ldots,\alpha-1 .
\end{equation}
We prove \(P(\alpha)\) by induction on \(\alpha\).  The case \(\alpha=1\) is vacuous.  Suppose that \(P(\alpha-1)\) holds for some \(1<\alpha\leq k\).

First consider \(\beta=1\).  By Equation \eqref{eq: semispray pre-connection coefficients recursion},
\[
\begin{aligned}
        (K_\alpha)^i_{(\alpha-1)j}
        &=
        \frac1\alpha\left[
        d_T(K_{\alpha-1})^i_{(\alpha-1)j}
        +(K_{\alpha-1})^m_{(\alpha-1)j}(K_1)^i_{(0)m}
        +(\alpha-1)(K_{\alpha-1})^i_{(\alpha-2)j}
        \right]
        \\
        &=
        \frac1\alpha\left[
        (K_1)^i_{(0)j}
        +(\alpha-1)(K_1)^i_{(0)j}
        \right]
        =
        (K_1)^i_{(0)j}.
\end{aligned}
\]
Now let \(\beta\geq2\).  The induction hypothesis gives
\[
        (K_{\alpha-1})^i_{(\alpha-\beta)j}
        =
        (K_{\beta-1})^i_{(0)j}, \quad \text{and} \quad (K_{\alpha-1})^i_{(\alpha-\beta-1)j}
        =
        (K_\beta)^i_{(0)j}.
\]
when \(\beta\leq\alpha-2\), while the case \(\beta=\alpha-1\) is tautological for the latter.  Therefore, again by Equation \eqref{eq: semispray pre-connection coefficients recursion},
\[
\begin{aligned}
        (K_\alpha)^i_{(\alpha-\beta)j}
        &=
        \frac1\alpha\left[
        d_T(K_{\beta-1})^i_{(0)j}
        +(K_{\beta-1})^m_{(0)j}(K_1)^i_{(0)m}
        +(\alpha-\beta)(K_\beta)^i_{(0)j}
        \right]
        \\
        &=
        \frac1\alpha\left[
        \beta(K_\beta)^i_{(0)j}
        +(\alpha-\beta)(K_\beta)^i_{(0)j}
        \right]
        =
        (K_\beta)^i_{(0)j},
\end{aligned}
\]
where the second equality uses Equation \eqref{eq: recursive coef K}.  This proves Equation \eqref{eq: LC-tower-coefficient-compatibility}.

Together with
\[
        (K_\alpha)^i_{(\alpha)j}=\delta^i_j,
        \qquad
        (K_\alpha)^i_{(\mu)j}=0\quad\text{for }\mu>\alpha,
\]
Equation \eqref{eq: LC-tower-coefficient-compatibility} shows that
\[
        \overset{(k)}{K}_\alpha
        =
        \left(
        dq^{(\alpha)i}
        +(K_1)^i_jdq^{(\alpha-1)j}
        +\cdots
        +(K_\alpha)^i_jdq^{(0)j}
        \right)\otimes\partial_i .
\]
Hence \(\overset{(k)}{K}=(\overset{(k)}{K}_1,\ldots,\overset{(k)}{K}_k)\) is a connection map on \(T^{(k)}M\).

Finally, we have already shown that \((K_\alpha)^i_j\) depends only on the coordinates \(q^{(\mu)a}\) with \(0\leq\mu\leq\alpha\).  Therefore \(\overset{(k)}{K}_\alpha\) descends through \(\overset{(k)}{\tau}_\alpha\) for each \(\alpha\), and the resulting family is compatible with the tower projections.  Hence \(\{\overset{(\alpha)}{K}\}_{\alpha=1}^k\) is a connection tower.
\qed
\end{proof}

The connection coefficients satisfy
\begin{equation}\label{eq: partialK}
        \frac{\partial (K_\alpha)^j_i}{\partial q^{(\alpha)l}}=\Gamma^j_{li},
        \qquad \alpha=1,\ldots,k,
\end{equation}
and, for \(\alpha\geq2\),
\begin{equation}\label{eq: partialK-1}
        \frac{\partial (K_\alpha)^j_i}{\partial q^{(\alpha-1)l}}
        =
        \frac1\alpha q^{(1)a}dq^j\left(
        \nabla_{\partial_l}\nabla_{\partial_a}\partial_i
        +(\alpha-1)\nabla_{\partial_a}\nabla_{\partial_l}\partial_i
        \right).
\end{equation}
In general, the connection coefficients admit the intrinsic expression
\[
        (K_\alpha)^j_i
        =
        dq^j\left(\overset{(k)}{K}_\alpha\left(\frac{\partial}{\partial q^{(0)i}}\right)\right)
        =
        \frac1{\alpha!}dq^j\left(
        \nabla^\alpha_{\frac{\partial\gamma}{\partial t}}
        \partial_i
        \right)\Big\vert_{t=0}.
\]
For the first values of \(\alpha\),
\begin{align*}
        (K_1)^j_i
        &=q^{(1)a}\Gamma^j_{ai},\\
        (K_2)^j_i
        &=q^{(2)a}dq^j(\nabla_{\partial_a}\partial_i)
          +\frac12q^{(1)a}q^{(1)b}dq^j(\nabla_{\partial_a}\nabla_{\partial_b}\partial_i),\\
        (K_3)^j_i
        &=q^{(3)a}dq^j(\nabla_{\partial_a}\partial_i)
          +q^{(2)a}q^{(1)b}dq^j(\nabla_{\partial_b}\nabla_{\partial_a}\partial_i)\\
        &\qquad
          +\frac16q^{(1)a}q^{(1)b}q^{(1)c}dq^j(\nabla_{\partial_c}\nabla_{\partial_a}\nabla_{\partial_b}\partial_i)
          +\frac13q^{(2)a}q^{(1)b}R^j_{abi}.
\end{align*}

We shall also use the diffeomorphism \(\overset{(k)}{F}\) from Equation \eqref{eq: F} associated with the present connection tower.  Let \(\overset{(k)}{F}_\beta:T^{(k)}M\to TM\) be its components functions, i.e.
\[
        \overset{(k)}{F}(u)=(\overset{(k)}{F}_0(u),\overset{(k)}{F}_1(u),\ldots,\overset{(k)}{F}_{k-1}(u)).
\]
Then, if \(u=j^k_0(q)\),
\[
        \overset{(k)}{F}_\alpha(u)
        =
        \nabla^\alpha_{\frac{dq}{dt}}\frac{dq}{dt}\Big\vert_{t=0},
        \qquad \alpha=0,\ldots,k-1.
\]
If \(F_\alpha^i\) are the coordinate functions of \(\overset{(k)}{F}_\alpha\), then
\begin{equation}\label{eq: cov_alpha_coef}
        F_\alpha^i=\overset{(k)}{d_T}F_{\alpha-1}^i+F_{\alpha-1}^lq^{(1)j}\Gamma^i_{lj},
        \qquad F_0^i=q^{(1)i}.
\end{equation}
Moreover,
\begin{equation}\label{eq: coef der F}
        \frac{\partial F_\alpha^i}{\partial q^{(\alpha)j}}=(\alpha+1)!q^{(1)l}\Gamma^i_{jl},
        \qquad
        \frac{\partial F_\alpha^i}{\partial q^{(\alpha+1)j}}=(\alpha+1)!\delta^i_j,
        \qquad
        \frac{\partial F_\alpha^i}{\partial q^{(\mu)j}}=0
\end{equation}
for \(\alpha+2\leq\mu\leq k\), and consequently
\begin{equation}\label{eq: coef delta der F vanish}
        \frac{\delta F_\alpha^i}{\delta q^{(\alpha)j}}=0.
\end{equation}
We denote \(\overset{(k)}{F}_\alpha(u)\) by \(u^{(\alpha+1)}\).  Thus, through \(\overset{(k)}{F}\), we identify a point \(u\in T^{(k)}M\) with
\[
        (u^{(1)},u^{(2)},\ldots,u^{(k)})\in (TM)_\oplus^k.
\]

\begin{remark}
The above construction is not the only formal way one might try to extend the first-order Dombrowski connection map, since there is no a priori preference for the order of covariant differentiation in the components of the connection maps. For instance, a natural alternative candidate is
$$\overset{(k)}{\widetilde K}_\alpha(X)
=
\frac{1}{\alpha!}
\nabla_{\frac{\partial \gamma}{\partial s}}
\nabla_{\frac{\partial \gamma}{\partial t}}^{\alpha-1}
\frac{\partial \gamma}{\partial t}
\Big|_{(s,t)=(0,0)} ,$$
which agrees with the usual Dombrowski connection map in order one. However, we claim that already for $k=3$, this candidate fails to produce a valid connection map without additional curvature correction terms, due to the failure of the compatibility condition with the canonical almost-tangent structure $J$. Thus, among the simple iterated-covariant-derivative constructions without correction terms, the order of differentiation used in this manuscript is the natural one. It is in this sense that the resulting Levi--Civita-induced connection tower can be regarded as canonical. We leave a detailed analysis of alternative constructions to future work.
\end{remark}

\subsection{Lie brackets of lifted vector fields}\label{subsec:LC-lift-brackets}

We now calculate Lie brackets of $\overset{(k)}{h}_\alpha$-lifts of vector fields on $M$ for the Levi--Civita-induced connection tower.  The $k$-vertical components of the brackets between $\overset{(k)}{h}_0$-lifts are the most involved; the full computation for $k=2$ is given below, while the longer $k=3$ computation is collected in Appendix~\ref{app:third-order-brackets}.

We begin with the brackets whose second factor is the $k$-vertical lift of a vector field on $M$.

\begin{lemma}\label{lemma: Brackets vk}
Let $X, Y \in \mathfrak{X}(M)$. Then
    \begin{equation*}
 [X^{\overset{(k)}{h}_0}, Y^{\overset{(k)}{v}_k}] = (\nabla_X Y)^{\overset{(k)}{v}_k}, \quad [X^{\overset{(k)}{h}_\alpha}, Y^{\overset{(k)}{v}_k}]= 0, \; \forall \alpha=1, 2, \ldots, k-1, \quad [X^{\overset{(k)}{v}_k}, Y^{\overset{(k)}{v}_k}]= 0.
    \end{equation*}
\end{lemma}

\begin{proof}
From Equation \eqref{eq: coordinate basis to adapted basis vector field}, the adapted basis vectors
$\delta/\delta q^{(\alpha)i}$ are independent of $q^{(k)l}$ for $\alpha=1,\ldots,k-1$. Hence
\[
        \left[\frac{\delta}{\delta q^{(\alpha)i}},\frac{\partial}{\partial q^{(k)l}}\right]=0,
        \qquad \alpha=1,\ldots,k-1.
\]
It follows from Lemma~\ref{lemma: fX} that $[X^{\overset{(k)}{h}_\alpha},Y^{\overset{(k)}{v}_k}]=0$ for $\alpha=1,\ldots,k-1$. The identity $[X^{\overset{(k)}{v}_k},Y^{\overset{(k)}{v}_k}]=0$ is immediate. Finally,
\[
\begin{aligned}
        \left[\frac{\delta}{\delta q^{(0)i}},\frac{\partial}{\partial q^{(k)l}}\right]
        &=\frac{\partial (K_k)^j_i}{\partial q^{(k)l}}\frac{\partial}{\partial q^{(k)j}}\\
        &=\Gamma^j_{li}\frac{\partial}{\partial q^{(k)j}}
        =\left(\nabla_{\partial_i}\partial_l\right)^{\overset{(k)}{v}_k}.
\end{aligned}
\]
Therefore, by Lemma~\ref{lemma: Brackets fX gY},
\[
        [X^{\overset{(k)}{h}_0},Y^{\overset{(k)}{v}_k}]
        =\left(X^i\partial_iY^j+\Gamma^j_{il}X^iY^l\right)\frac{\partial}{\partial q^{(k)j}}
        =(\nabla_XY)^{\overset{(k)}{v}_k}.
\]
\qed
\end{proof}
For $k=2$, the complete bracket table is as follows.
\begin{theorem}\label{prop: liebrac2}
Let $X, Y \in \mathfrak{X}(M)$. Then
\begin{align*}
[X^{\overset{(2)}{h}_0}, Y^{\overset{(2)}{v}_2}] &= (\nabla_X Y)^{\overset{(2)}{v}_2}, \quad [X^{\overset{(2)}{h}_1}, Y^{\overset{(2)}{v}_2}] = [X^{\overset{(2)}{v}_2}, Y^{\overset{(2)}{v}_2}] = 0\\
        [X^{\overset{(2)}{h}_0}, Y^{\overset{(2)}{h}_1}] &= (\nabla_X Y)^{\overset{(2)}{h}_1} + \frac12\left(  R(u^{(1)},X)Y -  R(X,Y)u^{(1)} \right)_u^{\overset{(2)}{v}_2}, \quad [X^{\overset{(2)}{h}_1}, Y^{\overset{(2)}{h}_1}] =  0,\\
        [X^{\overset{(2)}{h}_0}, Y^{\overset{(2)}{h}_0}] & =  [X, Y]^{\overset{(2)}{h}_0} -\big(R(X,Y)u^{(1)}\big)^{\overset{(2)}{h}_1} -\big(\frac12 R(X,Y)u^{(2)}+ (\nabla_{u^{(1)}} R)(X, Y)u^{(1)})\big)^{\overset{(2)}{v}_2},
    \end{align*}\normalsize
     where $u^{(1)}=\overset{(2)}{F}_0(u)$, $u^{(2)}= \overset{(2)}{F}_1(u)$, $u \in T^{(2)}M$.
    \end{theorem}
From Lemma \ref{lemma: K mu}, many of the Lie brackets are inherited directly from the first-order case studied by Dombrowski (Theorem \ref{k1}). The reamining brackets will be handled now in a series of technical lemmas. 

\begin{lemma}\label{lemma: Brackets h1,1}
Let  $X, Y \in \mathfrak{X}(M)$. Then $\overset{(2)}{K}_2([X^{\overset{(2)}{h}_1}, Y^{\overset{(2)}{h}_1}]_u) = 0$, for $u \in T^{(2)}M$.
\end{lemma}

\begin{proof}
Using Equation \eqref{eq: coordinate basis to adapted basis vector field} and Lemma \ref{lemma: Brackets fX gY}, we get
         \begin{align*}
        \left[ \frac{\delta}{\delta q^{(1)i}}, \frac{\delta}{\delta q^{(1)l}}\right] &= \left[\frac{\partial}{\partial q^{(1)i}} - (K_1)^j_i \frac{\delta}{\delta q^{(2)j}}, \ \frac{\delta}{\delta q^{(1)l}}\right] \\
        &= \left[\frac{\partial}{\partial q^{(1)i}}, \frac{\delta}{\delta q^{(1)l}}\right] + \frac{\delta (K_1)^j_i}{\delta q^{(1)l}}\frac{\delta}{\delta q^{(2)j}} \\
        &= \left[\frac{\partial}{\partial q^{(1)i}}, \frac{\partial}{\partial q^{(1)l}} - (K_1)_l^j \frac{\delta}{\delta q^{(2)j}} \right] + \frac{\delta (K_1)^j_i}{\delta q^{(1)l}}\frac{\delta}{\delta q^{(2)j}} \\
        &= -(K_1)^j_l\left[\frac{\partial}{\partial q^{(1)i}}, \frac{\delta}{\delta q^{(2)j}} \right] + \left(\frac{\partial (K_1)^j_i}{\partial q^{(1)l}} -\frac{\partial (K_1)^j_l}{\partial q^{(1)i}}\right)\frac{\delta}{\delta q^{(2)j}}  \\
        &= \left(\frac{\partial (K_1)^j_i}{\partial q^{(1)l}} -\frac{\partial (K_1)^j_l}{\partial q^{(1)i}}\right)\frac{\delta}{\delta q^{(2)j}}.
        \end{align*}
Thus, since $\nabla$ is torsion-free,
   $ \overset{(2)}{K}_2\circ  \left[ \frac{\delta}{\delta q^{(1)i}}, \frac{\delta}{\delta q^{(1)l}}\right] = 0$. Taking into account Lemma \ref{lemma: Brackets fX gY}, the result follows.
   \qed
\end{proof}

\begin{lemma}\label{lemma: Brackets h0,1}
Let  $X, Y \in \mathfrak{X}(M)$. Then $   \overset{(2)}{K}_2([X^{\overset{(2)}{h}_0}, Y^{\overset{(2)}{h}_1}]_u) = \frac12\left(  R(u^{(1)},X_p)Y_p -  R(X_p,Y_p)u^{(1)} \right)$,  where $u^{(1)}=\overset{(2)}{F}_0(u)$, $u \in T^{(2)}_pM$, $p\in M$.
\end{lemma}

\begin{proof}
Using again Equation \eqref{eq: coordinate basis to adapted basis vector field} and Lemma \ref{lemma: Brackets fX gY}, we get

        \begin{align*}
    \left[\frac{\delta}{\delta q^{(0)i}}, \frac{\delta}{\delta q^{(1)l}} \right] &= \left[\frac{\partial}{\partial q^{(0)i}} - (K_1)^j_i \frac{\delta}{\delta q^{(1)j}} - (K_2)^j_i \frac{\delta}{\delta q^{(2)j}} , \  \frac{\delta}{\delta q^{(1)l}} \right] \\
    &= \left[\frac{\partial}{\partial q^{(0)i}}, \frac{\delta}{\delta q^{(1)l}}  \right] + \frac{\delta (K_1)^j_i}{\delta q^{(1)l}} \frac{\delta}{\delta q^{(1)j}} + \frac{\delta (K_2)^j_i}{\delta q^{(1)l}} \frac{\delta}{\delta q^{(2)j}}  - (K_1)^j_i \left[\frac{\delta}{\delta q^{(1)j}}, \ \frac{\delta}{\delta q^{(1)l}} \right] \\
    &= \left[\frac{\partial}{\partial q^{(0)i}}, \frac{\delta}{\delta q^{(1)l}}  \right] + \frac{\delta (K_1)^j_i}{\delta q^{(1)l}} \frac{\delta}{\delta q^{(1)j}} + \frac{\delta (K_2)^j_i}{\delta q^{(1)l}} \frac{\delta}{\delta q^{(2)j}}.
\end{align*}
Moreover,
\begin{align*}
    \left[\frac{\partial}{\partial q^{(0)i}},\ \frac{\delta}{\delta q^{(1)l}}  \right] &= \left[\frac{\partial}{\partial q^{(0)i}}, \frac{\partial}{\partial q^{(1)l}} - (K_1)^j_l \frac{\delta}{\delta q^{(2)j}}  \right] \\
    &=  -\frac{\partial(K_1)^j_l}{\partial q^{(0)i}} \frac{\delta}{\delta q^{(2)j}}- (K_1)^j_l \left[\frac{\partial}{\partial q^{(0)i}}, \frac{\delta}{\delta q^{(2)j}}\right] \\
    &= -\frac{\partial (K_1)^j_l}{\partial q^{(0)i}} \frac{\delta}{\delta q^{(2)j}}.
\end{align*}
Hence, we obtain
\begin{align*}
    \left[\frac{\delta}{\delta q^{(0)i}}, \frac{\delta}{\delta q^{(1)l}} \right] = &\frac{\partial (K_1)^j_i}{\partial q^{(1)l}} \frac{\delta}{\delta q^{(1)j}} + \left(\frac{\delta (K_2)^j_i}{\delta q^{(1)l}} -  \frac{\partial (K_1)^j_l}{\partial q^{(0)i}}\right)\frac{\delta}{\delta q^{(2)j}},
\end{align*}
from which it follows that
\begin{align*}
    \overset{(2)}{K}_2\circ \left[\frac{\delta}{\delta q^{(0)i}}, \frac{\delta}{\delta q^{(1)l}} \right]&= \left(\frac{\delta (K_2)^j_i}{\delta q^{(1)l}} -  \frac{\partial (K_1)^j_l}{\partial q^{(0)i}}\right)\partial_j \\
    &= \left(\frac{\partial (K_2)^j_i}{\partial q^{(1)l}} - (K_1)^m_l \frac{\partial (K_1)^j_i}{\partial q^{(1)m}} - \frac{\partial (K_1)^j_l}{\partial q^{(0)i}}\right)\partial_j.
\end{align*}
By direct calculation,
    \begin{align*}
        \frac{\partial (K_2)^j_i}{\partial q^{(1)l}} &= \frac12 q^{(1)a} dq^j(\nabla_{\partial_a} \nabla_{\partial_l} \partial_i) + \frac12 q^{(1)a} dq^j(\nabla_{\partial_l} \nabla_{\partial_a} \partial_i) \\
        &= q^{(1)a} dq^j (\nabla_{\partial_a} \nabla_{\partial_l} \partial_i) + \frac12 q^{(1)a} R_{lai}^j,
    \end{align*}
    and
    \begin{align*}
        (K_1)^m_l \frac{\partial (K_1)^j_i}{\partial q^{(1)m}} + \frac{\partial (K_1)^j_l}{\partial q^{(0)i}} &= q^{(1)a}dq^j(\nabla_{\partial_i}\nabla_{\partial_a} \partial_l).
    \end{align*}
Therefore,
\begin{align*}
    \overset{(2)}{K}_2\left(\left[\frac{\delta}{\delta q^{(0)i}}, \frac{\delta}{\delta q^{(1)l}} \right]_u\right) &=
   q^{(1)a} \left( dq^j\left(\nabla_{\partial_a} \nabla_{\partial_l} \partial_i - \nabla_{\partial_i}\nabla_{\partial_a} \partial_l\right) + \frac12 R_{lai}^j\right)\partial_j\Big{\vert}_{p} \\
    &= q^{(1)a}(R_{ail}^j + \frac12 R^j_{lai})\partial_j\Big{\vert}_{p} \\
    &= \frac12 q^{(1)a} (R^j_{ail} - R^j_{ila})\partial_j\Big{\vert}_{p}.
\end{align*}
Finally, taking into account Lemma \ref{lemma: Brackets fX gY}, the result follows.
\qed
\end{proof}

\begin{lemma}\label{lemma: Brackets h0h0}
Let  $X, Y \in \mathfrak{X}(M)$. Then $\overset{(2)}{K}_2([X^{\overset{(2)}{h}_0}, Y^{\overset{(2)}{h}_0}]_u)= -\frac12 R(X_p,Y_p)u^{(2)} - (\nabla_{u^{(1)}} R)(X_p, Y_p)u^{(1)},$
 where $u^{(1)}=\overset{(2)}{F}_0(u)$, $u^{(2)}= \overset{(2)}{F}_1(u)$, $u \in T^{(2)}_pM$, $p\in M$.
    \end{lemma}

\begin{proof}
It suffices to  calculate $\overset{(2)}{v}_2\circ\left[\frac{\delta}{\delta q^{(0)i}},  \frac{\delta}{\delta q^{(0)j}}\right] $. Due to the complexity of the calculations involving the derivatives in order to  $q^{(0)i}$, we use another method to determine it. Since $\overset{(2)}{v}_2\circ \left[\frac{\delta}{\delta q^{(0)i}},  \frac{\delta}{\delta q^{(0)j}}\right] $ is the last component of $\left[\frac{\delta}{\delta q^{(0)i}},  \frac{\delta}{\delta q^{(0)j}}\right]$ to determine, we will  use the Jacobi identity and combine it with all the other components.

We first observe that \begin{equation}\label{eq: h0 h0 S}
\left[\frac{\delta}{\delta q^{(0)i}},  \frac{\delta}{\delta q^{(0)j}}\right]_u =R_{ija}^l(p)q^{(1)a}\frac{\delta}{\delta q^{(1)l}}\Big{\vert}_{u}+S_{(0)i(0)j}^{(2)l}(u)\frac{\delta}{\delta q^{(2)l}}\Big{\vert}_{u},   \; u\in T_p^{(2)}M,p\in M,
\end{equation}
where $S_{(0)i(0)j}^{(2)l}\in C^\infty(T^{(2)}M)$. Hence, it is enough  to obtain  the   smooth functions $S_{(0)i(0)j}^{(2)l}$.

From the Jacobi identity, we get

\begin{align*}
    \left[ \frac{\delta}{\delta q^{(2)m}}, \left[\frac{\delta}{\delta q^{(0)i}},  \frac{\delta}{\delta q^{(0)j}}\right]\right] &= \left[\frac{\delta}{\delta q^{(0)j}}, \left[\frac{\delta}{\delta q^{(0)i}}, \frac{\delta}{\delta q^{(2)m}}\right] \right] -\left[\frac{\delta}{\delta q^{(0)i}}, \left[\frac{\delta}{\delta q^{(0)j}}, \frac{\delta}{\delta q^{(2)m}}\right] \right] \\
    &= \left[\frac{\delta}{\delta q^{(0)j}}, \left(\nabla_{\partial_i}\partial_m \right)^{\overset{(2)}{v}_2}\right]  - \left[\frac{\delta}{\delta q^{(0)i}}, \left(\nabla_{\partial_j}\partial_m \right)^{\overset{(2)}{v}_2} \right] \\
    &= (R(\p_j, \p_i)\p_m)^{\overset{(2)}{v}_2}
\end{align*} \normalsize

On the other hand, using Equation \eqref{eq: h0 h0 S}, we obtain
\begin{align*}
     \left[ \frac{\delta}{\delta q^{(2)m}}, \left[\frac{\delta}{\delta q^{(0)i}},  \frac{\delta}{\delta q^{(0)j}}\right]\right] &=  R_{ija}^l(p)q^{(1)a}\left[ \frac{\delta}{\delta q^{(2)m}},\frac{\delta}{\delta q^{(1)l}}\right]+\frac{\delta S_{(0)i(0)j}^{(2)l}}{\delta q^{(2)m}} \frac{\delta}{\delta q^{(2)l}}+S_{(0)i(0)j}^{(2)l}\left[ \frac{\delta}{\delta q^{(2)m}}, \frac{\delta}{\delta q^{(2)l}}\right]\\
     &=  \frac{\partial S_{(0)i(0)j}^{(2)l}}{\partial q^{(2)m}} \frac{\partial}{\partial q^{(2)l}}.
\end{align*}
Combining the two last equations, it follows that
\begin{align*}
    S_{(0)i(0)j}^{(2)l} &= -\frac 1 2 R_{ijm}^l F_1^m + \Phi^{l}_{ij},
\end{align*}
where the functions $\Phi^{l}_{ij} \in C^\infty(T^{(2)}M)$ satisfy $\frac{\p \Phi^{l}_{ij}}{\p q^{(2)\lambda}}=  0$, for all $\lambda=1,\dots,n$.

From the Jacobi identity, we also have

\begin{align*}
    \left[ \frac{\delta}{\delta q^{(1)m}}, \left[\frac{\delta}{\delta q^{(0)i}},  \frac{\delta}{\delta q^{(0)j}}\right]\right] &= \left[\frac{\delta}{\delta q^{(0)j}}, \left[\frac{\delta}{\delta q^{(0)i}}, \frac{\delta}{\delta q^{(1)m}}\right] \right] -\left[\frac{\delta}{\delta q^{(0)i}}, \left[\frac{\delta}{\delta q^{(0)j}}, \frac{\delta}{\delta q^{(1)m}}\right] \right].
\end{align*} \normalsize
Let $u \in T_p^{(2)}M$, with $p\in M$. Then, from Lemma \ref{lemma: Brackets h0,1},
\begin{align*}
    \left[\frac{\delta}{\delta q^{(0)j}}, \left[\frac{\delta}{\delta q^{(0)i}}, \frac{\delta}{\delta q^{(1)m}} \right] \right]_u&= \left[\frac{\delta}{\delta q^{(0)j}}, \left(\nabla_{\p_i} \p_m \right)^{\overset{(2)}{h}_1} \right]_u + \frac12 \left[\frac{\delta}{\delta q^{(0)j}}, \left(R(u^{(1)}, \p_i)\p_m - R(\p_i, \p_m)u^{(1)}\right)^{\overset{(2)}{v}_2} \right]_u.
    \end{align*}
 We again use Lemma \ref{lemma: Brackets h0,1} to simplify  the first term.
\begin{align*}
\left[\frac{\delta}{\delta q^{(0)j}}, \left(\nabla_{\p_i} \p_m \right)^{\overset{(2)}{h}_1} \right]_u &
= (\nabla_{\p_j} \nabla_{\p_i} \p_m)_u^{\overset{(2)}{h}_1} + \frac12\left(R(u^{(1)}, \p_j)\nabla_{\p_i}\p_m - R(\p_j, \nabla_{\p_i}\p_m)u^{(1)} \right)_u^{\overset{(2)}{v}_2} \end{align*}
On the other hand, to calculate the Lie bracket
$\Big{[}\frac{\delta}{\delta q^{(0)j}}, \Big{(}R(u^{(1)}, \p_i)\p_m - R(\p_i, \p_m)u^{(1)} \Big{)}^{\overset{(2)}{v}_2} \Big{]}_u $, we first need to rewrite it
as
$\Big{[}\frac{\delta}{\delta q^{(0)j}}, F_0^a\Big{(}R(\p_a, \p_i)\p_m - R(\p_i, \p_m)\p_a \Big{)}^{\overset{(2)}{v}_2} \Big{]}_u $. Then
\begin{align*} \Big{[}\frac{\delta}{\delta q^{(0)j}}, F_0^a\Big{(}R(\p_a, \p_i)\p_m - R(\p_i, \p_m)\p_a \Big{)}^{\overset{(2)}{v}_2} \Big{]}= &
\frac{\delta F_0^a}{\delta q^{(0)j}}\Big{(}R(\p_a, \p_i)\p_m - R(\p_i, \p_m)\p_a \Big{)}^{\overset{(2)}{v}_2}\\
&+F_0^a\Big{[}\frac{\delta}{\delta q^{(0)j}}, \Big{(}R(\p_a, \p_i)\p_m - R(\p_i, \p_m)\p_a \Big{)}^{\overset{(2)}{v}_2} \Big{]}.
\end{align*}
Now, we  simplify the first term using
$\frac{\delta F_0^a}{\delta q^{(0)j}}(u)=-dq^j(\nabla_{u^{(1)}}\p_a)$ and apply Lemma \ref{lemma: Brackets vk} to the second term.
\begin{align*}
\Big{[}\frac{\delta}{\delta q^{(0)j}}, \Big{(}R(u^{(1)}, \p_i)\p_m - R(\p_i, \p_m)u^{(1)} \Big{)}^{\overset{(2)}{v}_2} \Big{]}_u=&\Big{(}(\nabla_{\p_j} R)(u^{(1)}, \p_i)\p_m + R(u^{(1)}, \nabla_{\p_j}\p_i)\p_m \\ & + R(u^{(1)}, \p_i)\nabla_{\p_j}\p_m -  (\nabla_{\p_j}R)(\p_i, \p_m)u^{(1)} \\ & - R(\nabla_{\p_j} \p_i, \p_m)u^{(1)} - R(\p_i, \nabla_{\p_j}\p_m)u^{(1)}\Big{)}_u^{\overset{(2)}{v}_2}
\end{align*}
Combining the two Lie brackets we obtain

{\small
\begin{align*}
    \overset{(2)}{v}_2\left(\left[\frac{\delta}{\delta q^{(0)j}}, \left[\frac{\delta}{\delta q^{(0)i}}, \frac{\delta}{\delta q^{(1)m}} \right] \right]_u\right)&= \frac12\left(R(u^{(1)}, \p_j)\nabla_{\p_i}\p_m - R(\p_j, \nabla_{\p_i}\p_m)u^{(1)} +(\nabla_{\p_j} R)(u^{(1)}, \p_i)\p_m + R(u^{(1)}, \nabla_{\p_j}\p_i)\p_m \right.\\
    & \quad \left. + R(u^{(1)}, \p_i)\nabla_{\p_j}\p_m -  (\nabla_{\p_j}R)(\p_i, \p_m)u^{(1)} - R(\nabla_{\p_j} \p_i, \p_m)u^{(1)} - R(\p_i, \nabla_{\p_j}\p_m)u^{(1)}\right)_u^{\overset{(2)}{v}_2} .
\end{align*}}
Applying this result twice, we finally conclude that
\begin{align*}
&\overset{(2)}{K}_2\left(\left[ \frac{\delta}{\delta q^{(1)m}}, \left[\frac{\delta}{\delta q^{(0)i}},  \frac{\delta}{\delta q^{(0)j}}\right]\right]_u\right)\\
&= \overset{(2)}{K}_2\left(\left[\frac{\delta}{\delta q^{(0)j}}, \left[\frac{\delta}{\delta q^{(0)i}}, \frac{\delta}{\delta q^{(1)m}} \right] \right]_u - \left[\frac{\delta}{\delta q^{(0)i}}, \left[\frac{\delta}{\delta q^{(0)j}}, \frac{\delta}{\delta q^{(1)m}} \right] \right]_u\right)\\
&= \frac12 \left[(\nabla_{\p_j} R)(u^{(1)}, \p_i)\p_m - (\nabla_{\p_i} R)(u^{(1)}, \p_j)\p_m - (\nabla_{\p_j}R)(\p_i, \p_m)u^{(1)} + (\nabla_{\p_i}R)(\p_j, \p_m)u^{(1)} \right]_p \\
    &= -\frac12 \left[(\nabla_{u^{(1)}} R)(\p_i, \p_j)\p_m + (\nabla_{\p_m} R)(\p_i, \p_j)u^{(1)} \right]_p.
\end{align*}
On the other hand, using again Equation \eqref{eq: h0 h0 S}, we have
\begin{align*}
    \overset{(2)}{v}_2 \circ \left[ \frac{\delta}{\delta q^{(1)m}}, \left[\frac{\delta}{\delta q^{(0)i}},  \frac{\delta}{\delta q^{(0)j}}\right]\right]
    &= \frac{\delta S_{(0)i(0)j}^{(2)l}}{\delta q^{(1)m}} \frac{\delta}{\delta q^{(2)l}}  \\
    &=  \frac{\p\Phi^{l}_{ij}}{\p q^{(1)m}} \frac{\delta}{\delta q^{(2)l}}.
\end{align*}
Hence, combining the two last equations, it follows that

\begin{align*}
    \Phi^{l}_{ij}(u)&= -\frac12 q^{(1)m}dq^l\left[(\nabla_{u^{(1)}} R)(\p_i, \p_j)\p_m + (\nabla_{\p_m} R)(\p_i, \p_j)u^{(1)} \right]_p + h^{l}_{ij}(u)\\
    &= - dq^l\left[(\nabla_{u^{(1)}} R)(\partial_i,\partial_j)u^{(1)}\right]_p+ h^{l}_{ij}(u).
\end{align*}
where the functions $h^{l}_{ij} \in C^\infty(T^{(2)}M)$  satisfy $\frac{\p h^{l}_{ij}}{\p q^{(2)\lambda}}= \frac{\p h^{l}_{ij}}{\p q^{(1)\lambda}}= 0$, for all $\lambda=1,\dots,n$.

Therefore,
\begin{align*}
    \overset{(2)}{K}_2\left( \left[\frac{\delta}{\delta q^{(0)i}},  \frac{\delta}{\delta q^{(0)j}}\right]_u \right)
    &=\left[-\frac12 R(\partial_i,\partial_j)u^{(2)} - (\nabla_{u^{(1)}} R)(\partial_i,\partial_j)u^{(1)}\right]_p+ h^{l}_{ij}(u).
\end{align*}

Finally, using  Lemma \ref{lemma: fiberlinear}, we conclude that
\begin{align*}
    \overset{(2)}{K}_2([\frac{\delta}{\delta q^{(0)i}},  \frac{\delta}{\delta q^{(0)j}}]_u) &= \left[-\frac12 R(\partial_i,\partial_j)u^{(2)} - (\nabla_{u^{(1)}} R)(\partial_i,\partial_j)u^{(1)}\right]_p,
\end{align*}
which, taking into account Lemma \ref{lemma: Brackets fX gY}, leads to the result.
\qed
\end{proof}

\begin{proof}[Proof of Theorem \ref{prop: liebrac2}]
The components below order $2$ are inherited from the fist-order bracket computations given in Theorem \ref{k1} together with Lemma \ref{lemma: K mu}, whereas the identities involving $\overset{(2)}{v}_2$ follow from Lemma~\ref{lemma: Brackets vk}.  The remaining components are given by Lemmas~\ref{lemma: Brackets h1,1}, \ref{lemma: Brackets h0,1}, and \ref{lemma: Brackets h0h0}, together with Lemma~\ref{lemma: Brackets fX gY} from Appendix~\ref{app:lift-technical-lemmas}.
\qed
\end{proof}

For $k=3$, one obtains the following higher-order analogue.

\begin{theorem}\label{thm: liebrac3}
Let $X, Y \in \mathfrak{X}(M)$. Then
\small{\begin{align*}
  [X^{\overset{(3)}{h}_0}, Y^{\overset{(3)}{v}_3}] = &(\nabla_X Y)^{\overset{(3)}{v}_3}, \quad [X^{\overset{(3)}{h}_1}, Y^{\overset{(3)}{v}_3}] = [X^{\overset{(3)}{h}_2}, Y^{\overset{(3)}{v}_3}] = [X^{\overset{(3)}{h}_3}, Y^{\overset{(3)}{v}_3}] = 0,\\
  [X^{\overset{(3)}{h}_0}, Y^{\overset{(3)}{h}_2}] = & (\nabla_X Y)^{\overset{(3)}{h}_2} - \left(  R(X,u^{(1)})Y + \frac13 R(u^{(1)},Y)X \right)^{\overset{(3)}{v}_3}, \quad [X^{\overset{(3)}{h}_1}, Y^{\overset{(3)}{h}_2}] = [X^{\overset{(3)}{h}_2}, Y^{\overset{(3)}{h}_2}] = 0,\\
   [X^{\overset{(3)}{h}_0}, Y^{\overset{(3)}{h}_1}] = &(\nabla_X Y)^{\overset{(3)}{h}_1} + \frac12 \left(R(u^{(1)}, X)Y - R(X, Y)u^{(1)}\right)^{\overset{(3)}{h}_2} \\
        & -\frac12\left(R(X, Y)u^{(2)} + \frac13 R(Y, u^{(2)})X +(\nabla_{u^{(1)}} R)(X, Y)u^{(1)} + \frac13 (\nabla_{u^{(1)}} R)(Y, u^{(1)})X\right)^{\overset{(3)}{v}_3},\\
    [X^{\overset{(3)}{h}_1}, Y^{\overset{(3)}{h}_1}] = & -\frac12 \left(R(X,Y)u^{(1)}\right)^{\overset{(3)}{v}_3}, \\
        [X^{\overset{(3)}{h}_0}, Y^{\overset{(3)}{h}_0}] = & [X, Y]^{\overset{(3)}{h}_0} -\left(R(X,Y)u^{(1)}\right)^{\overset{(3)}{h}_1} -\left(\frac12 R(X,Y)u^{(2)} + (\nabla_{u^{(1)}} R)(X, Y)u^{(1)})\right)^{\overset{(3)}{h}_2} \\
        & -\left(\frac1{3!} R(X,Y)u^{(3)} + \frac13 (\nabla_{u^{(1)}} R)(X,Y)u^{(2)} - \frac13 (\nabla_{u^{(2)}}R)(X,Y)u^{(1)} + \frac12 (\nabla^2_{u^{(1)}} R)(X,Y)u^{(1)}\right)^{\overset{(3)}{v}_3},
    \end{align*}}\normalsize
   where  $u^{(1)}=\overset{(3)}{F}_0(u)$, $u^{(2)}=\overset{(3)}{F}_1(u)$, $u^{(3)}=\overset{(3)}{F}_2(u)$, $u \in T^{(3)}M$.
    \end{theorem}

\begin{proof}[Proof sketch]
The components below order $3$ are inherited from the second-order bracket formulas given in Theorem \ref{prop: liebrac2} together with Lemma~\ref{lemma: K mu}. Thus only the $3$-vertical components have to be computed. These are obtained by applying the adapted-basis bracket formula to the pairs $(\overset{(3)}{h}_0,\overset{(3)}{h}_2)$, $(\overset{(3)}{h}_1,\overset{(3)}{h}_1)$, $(\overset{(3)}{h}_0,\overset{(3)}{h}_1)$, and $(\overset{(3)}{h}_0,\overset{(3)}{h}_0)$. The first three computations use only the coefficient identities given in Equations \eqref{eq: partialK}--\eqref{eq: partialK-1} and the first Bianchi identity. The last computation, namely the $3$-vertical component of $[X^{\overset{(3)}{h}_0},Y^{\overset{(3)}{h}_0}]$, requires repeated use of the Jacobi identity and introduces the terms involving $\nabla R$ and $\nabla^2R$. These longer reductions are given in Appendix~\ref{app:third-order-brackets}. Combining them with Lemma~\ref{lemma: Brackets fX gY} gives the displayed formulas.
\qed
\end{proof}

Notice that, already in order $3$, covariant derivatives of curvature and second covariant derivatives of curvature appear. This is the main source of the combinatorial growth in the general bracket formulas. We conclude with a general reduction method for the $k$-vertical components of Lie brackets of $\overset{(k)}{h}_\alpha$-lifts on $T^{(k)}M$.  The method uses auxiliary vector fields in $T^{(k)}M$ that decompose the $\overset{(k)}{h}_\alpha$-lifts into lifts from $T^{(k-1)}M$ and $k$-vertical vector fields.  The coordinate formula is recorded in Appendix~\ref{app:general-bracket-reduction}.  We will need the following two consequences later.

\begin{proposition}\label{prop: Brackets hk-1}
Let $X, Y \in \mathfrak{X}(M)$.   Then
    \begin{equation*}
 [X^{\overset{(k)}{h}_{k-\alpha}}, Y^{\overset{(k)}{h}_{k-1}}] = 0, \alpha=1, 2, \ldots, k-1.
    \end{equation*}
\end{proposition}

The proof is given in Appendix~\ref{app:general-bracket-reduction}.

\begin{proposition}\label{prop: Brackets h_0 hk-1}
Let $X, Y \in \mathfrak{X}(M)$.   Then
\begin{equation}\label{equation: Brackets h_0 hk-1}
 [X^{\overset{(k)}{h}_0}, Y^{\overset{(k)}{h}_{k-1}}] = \left(\nabla_XY\right)^{\overset{(k)}{h}_{k-1}}-\frac 1 k\left(R(X,Y)u^{(1)}+(k-1)R(X,u^{(1)})Y\right)^{\overset{(k)}{v}_k},
\end{equation}
where $u^{(1)}=\overset{(k)}{F}_0(u)$, $u \in T^{(k)}M$.
\end{proposition}

The proof is given in Appendix~\ref{app:general-bracket-reduction}.

\section{Higher-order Sasaki metrics}
\label{sec:generalized-sasaki-metrics}

In this section, we use the Levi--Civita-induced connection tower constructed in Subsection~\ref{subsec:LC-connection-tower} to define a higher-order Sasaki metric on each higher-order tangent bundle.  We first recall the classical Sasaki metric on $TM$, since it provides the model for the higher-order construction.  We then define the $k$-Sasaki metric on $T^{(k)}M$, record the basic metric properties needed later, and compute the corresponding Levi--Civita connection in the cases $k=2$ and $k=3$.  Finally, we derive the geodesic equations on $(T^{(2)}M,\overset{(2)}g)$ and $(T^{(3)}M,\overset{(3)}g)$, and prove the general jet lift characterization for arbitrary order.

Throughout this section, $\overset{(k)}{K}=(\overset{(k)}{K}_1,\ldots,\overset{(k)}{K}_k)$ denotes the Levi--Civita-induced connection tower on $T^{(k)}M$. 

\subsection{The Sasaki metric on $TM$}\label{subsec TM 3}

There is also a canonical choice for the Riemannian metric on $TM$, called the \textit{Sasaki metric}, with respect to which the splitting Equation \eqref{eq: direct_sum_TM} is an orthogonal decomposition. The Sasaki metric $\overset{(1)}g$ is defined by
\[
        \overset{(1)}g(X,Y)=g(\tau_{1*}X,\tau_{1*}Y)+g(\overset{(1)}{K}(X),\overset{(1)}{K}(Y)),
\]
where $X,Y\in T_uTM$, $u\in TM$.

The following standard formulas will serve as the first-order model for the higher-order computations.

\begin{theorem}[Dombrowski, \cite{Dombrowski1962}]\label{prop: Levi--Civita Sasaki}
Let $X,Y\in\mathfrak X(M)$. Then
\begin{align*}
        \overset{(1)}{\nabla}_{X^{\overset{(1)}{v}_1}}Y^{\overset{(1)}{v}_1}&=0,\\
        \overset{(1)}{\nabla}_{X^{\overset{(1)}{h}_0}}Y^{\overset{(1)}{v}_1}&=\frac12\left(R(u,Y)X\right)^{\overset{(1)}{h}_0}+(\nabla_XY)^{\overset{(1)}{v}_1},\\
        \overset{(1)}{\nabla}_{X^{\overset{(1)}{h}_0}}Y^{\overset{(1)}{h}_0}&=(\nabla_XY)^{\overset{(1)}{h}_0}-\frac12\left(R(X,Y)u\right)^{\overset{(1)}{v}_1},
\end{align*}
where $u\in TM$.
\end{theorem}

\begin{proof}
Applying
Koszul's formula \eqref{eq: Koszul} to triples of horizontal and vertical lifts, and using the bracket
identities in Theorem~\ref{k1}, gives the displayed components after pairing with
arbitrary horizontal and vertical lifts.  The metric nondegeneracy of the horizontal--vertical
splitting then determines the covariant derivatives uniquely.
\qed \end{proof}

Let $\Gamma$ be a curve on $TM$. Denote by $q$ the curve obtained by projecting $\Gamma$ onto $M$, $q:=\tau_1\circ\Gamma$, and by $V$ the vector field along $q$ defined by the curve $\Gamma$. The velocity vector field $\dot \Gamma$ gives rise to the vector fields along $q$
\[
        \dot q=\tau_{1*}\circ \dot\Gamma,\qquad Y=\overset{(1)}{K}\circ \dot\Gamma,
\]
with respect to which
\[
        \dot\Gamma=\dot q^{\overset{(1)}{h}_0}+Y^{\overset{(1)}{v}_1}.
\]

\begin{theorem}\label{thm: Sasaki geodesic TM}
A curve $\Gamma$ on $TM$ is a Riemannian geodesic of the Sasaki metric if and only if
\begin{align*}
        0&=\nabla_{\dot q}\dot q+R(V,Y)\dot q,\\
        0&=\nabla_{\dot q}Y.
\end{align*}
\end{theorem}

\begin{proof}
Using the decomposition $\dot\Gamma=\dot q^{\overset{(1)}{h}_0}+Y^{\overset{(1)}{v}_1}$ and the formulas of
Theorem~\ref{prop: Levi--Civita Sasaki}, together with torsion-freeness to compute
$\overset{(1)}{\nabla}_{Y^{\overset{(1)}{v}_1}}\dot q^{\overset{(1)}{h}_0}$, one obtains
\[
        \overset{(1)}{\nabla}_{\dot\Gamma}\dot\Gamma
        =
        \left(\nabla_{\dot q}\dot q+R(V,Y)\dot q\right)^{\overset{(1)}{h}_0}
        +
        \left(\nabla_{\dot q}Y\right)^{\overset{(1)}{v}_1}.
\]
Since the horizontal and vertical distributions are orthogonal and complementary, the
covariant acceleration vanishes if and only if both displayed components vanish.
\qed \end{proof}

\begin{remark}\label{rem: hgeod}
If $\Gamma$ is a horizontal geodesic on $TM$, then $\tau_1\circ\Gamma$ is a geodesic on $M$.
\end{remark}

Given a curve $q$ on $M$, the curve $j^1q=(q,\dot q)$ is called the \textit{tangent lift} of $q$ (or \textit{first-jet lift}). If $\Gamma=j^1q$, then
\[
        \dot\Gamma=\dot q^{\overset{(1)}{h}_0}+(\nabla_{\dot q}\dot q)^{\overset{(1)}{v}_1}.
\]

\begin{corollary}\label{cor: 1jetgeod}
Let $q$ be a curve on $M$. The first-jet lift $\Gamma=j^1q$ is a geodesic on $TM$ if and only if $q$ is a geodesic on $M$.
\end{corollary}

\subsection{The $k$-Sasaki metric}

For each $u\in T^{(k)}M$, we consider the inner product $\llangle\cdot,\cdot\rrangle_u$ on $T_uT^{(k)}M$ defined by
\begin{equation}\label{eq: Sasaki metric}
        \llangle X,Y\rrangle_u
        =\langle\tau_{k*}X,\tau_{k*}Y\rangle
        +\langle \overset{(k)}{K}_1(X),\overset{(k)}{K}_1(Y)\rangle+\cdots+
        \langle \overset{(k)}{K}_k(X),\overset{(k)}{K}_k(Y)\rangle,
\end{equation}
where $X,Y\in T_uT^{(k)}M$.  These inner products define a Riemannian metric on $T^{(k)}M$ \cite[Theorem 9.1.3]{MironBook1997}. We call this metric the \textit{$k$-Sasaki metric} associated with the connection tower $\overset{(k)}{K}$, and denote it by $\overset{(k)}g$.

The $k$-Sasaki metric makes the multiconnection decomposition given in Equation \eqref{eq: multiconnection direct sum} orthogonal. Indeed, in terms of the adapted covector basis,
\[
        \overset{(k)}g
        =\sum_{\alpha=0}^k\sum_{i,j=1}^{\dim M}g_{ij}\,
        \delta q^{(\alpha)i}\otimes\delta q^{(\alpha)j},
\]
where $g_{ij}$ are the coordinates of the Riemannian metric on $M$.

\begin{remark}\label{rem: Rsubm}
The maps $\overset{(k)}{\tau}_{\!\alpha}$ are Riemannian submersions from $(T^{(k)}M,\overset{(k)}g)$ to $(T^{(\alpha)}M,\overset{(\alpha)}g)$, for $\alpha=0,1,\ldots,k-1$; see \cite{ONeill1966}. This will be relevant to the study of geodesics on these jet bundles.
\end{remark}

\begin{remark}\label{rem: flat-check}
If the base manifold is flat, then all curvature terms in the formulas below vanish. In this case the $k$-Sasaki metric is locally the product metric induced by the connection coordinates, and the geodesic equations reduce to
\[
        \nabla_{\dot q}\dot q=0,
        \qquad
        \nabla_{\dot q}Y^{(\alpha)}=0,
        \quad \alpha=1,\ldots,k.
\]
This provides a useful check on the curvature terms in the second- and third-order formulas.
\end{remark}

We now study Riemannian geodesics on $(T^{(k)}M,\overset{(k)}g)$, especially when the curves are $k$-horizontal or $k$-jets of curves on $M$.  If $q$ is a curve on $M$ and $\Gamma=j^kq$, then the velocity vector field of $\Gamma$ is
\begin{equation}\label{eq: decomposition curve T^k}
        \dot\Gamma
        =\dot q^{\overset{(k)}{h}_0}
        +\left(\nabla_{\dot q}\dot q\right)^{\overset{(k)}{h}_1}
        +\frac12\left(\nabla_{\dot q}^2\dot q\right)^{\overset{(k)}{h}_2}
        +\cdots+
        \frac1{k!}\left(\nabla_{\dot q}^k\dot q\right)^{\overset{(k)}{v}_k}.
\end{equation}

Let $\overset{(k)}\nabla$ denote the Levi--Civita connection of $(T^{(k)}M,\overset{(k)}g)$. Our aim is to now establish results relating the Levi--Civita connections on tangent bundles of different orders, as we previously did for Lie brackets in Lemma \ref{lemma: K mu}. The main computational tool here is the Koszul formula \ref{eq: Koszul}, together with the Lie bracket identities we established in the previous section. 

Due to the direct sum decomposition given in Equation \eqref{eq: multiconnection direct sum} and the $C^\infty(T^{(k)}M)$-bilinearity and non-degeneracy of the $k$-Sasaki metric, it suffices to apply the Koszul formula when $X,Y,Z$ are lifts of vector fields on $M$. The following three results isolate the general computations that are used repeatedly.  Their proofs are given in Appendix \ref{app:sasaki-koszul-reductions}.

\begin{lemma}\label{lem: appKoszul}
Let $X,Y,Z\in\mathfrak X(M)$. Then
\[
        \left(X^{\overset{(k)}{h}_\alpha}\llangle Y^{\overset{(k)}{h}_\beta},Z^{\overset{(k)}{h}_\mu}\rrangle\right)\circ\tau_k
        =\delta_{\alpha0}\delta_{\beta\mu}
        \left(\langle\nabla_XY,Z\rangle+\langle Y,\nabla_XZ\rangle\right).
\]
\end{lemma}

\begin{lemma}\label{lemma: nablahh}
Let $X,Y\in\mathfrak X(M)$. Then
\[
        \overset{(k)}{K}_\mu\circ\left(\overset{(k)}\nabla_{X^{\overset{(k)}{h}_\alpha}}Y^{\overset{(k)}{h}_\beta}\right)
        =\overset{(k-1)}{K}_\mu\circ\left(\overset{(k-1)}\nabla_{X^{\overset{(k-1)}{h}_\alpha}}Y^{\overset{(k-1)}{h}_\beta}\right)\circ\overset{(k)}{\tau}_{\!k-1},
\]
for $\mu=0,1,\ldots,k-1$, and
\[
        \overset{(k)}{K}_k\circ\left(\overset{(k)}\nabla_{X^{\overset{(k)}{h}_\alpha}}Y^{\overset{(k)}{h}_\beta}\right)
        =\frac12\overset{(k)}{K}_k\circ\left([X^{\overset{(k)}{h}_\alpha},Y^{\overset{(k)}{h}_\beta}]\right),
\]
for all $\alpha,\beta=0,1,\ldots,k-1$.
\end{lemma}

\begin{proposition}\label{prop: Levi--Civita kSasaki}
Let $\overset{(k)}\nabla$ be the Levi--Civita connection of $(T^{(k)}M,\overset{(k)}g)$. Then
\begin{align*}
        \overset{(k)}{K}_k\circ\left(\overset{(k)}\nabla_{X^{\overset{(k)}{v}_k}}Y^{\overset{(k)}{v}_k}\right)
        &=\overset{(k)}{K}_k\circ\left(\overset{(k)}\nabla_{X^{\overset{(k)}{h}_1}}Y^{\overset{(k)}{v}_k}\right)=\cdots=\overset{(k)}{K}_k\circ\left(\overset{(k)}\nabla_{X^{\overset{(k)}{h}_{k-1}}}Y^{\overset{(k)}{v}_k}\right)=0,\\
        \overset{(k)}{K}_k\circ\left(\overset{(k)}\nabla_{X^{\overset{(k)}{h}_0}}Y^{\overset{(k)}{v}_k}\right)&=(\nabla_XY)\circ\tau_k,\\
        \overset{(k)}{K}_\mu\circ\left(\overset{(k)}\nabla_{X^{\overset{(k)}{v}_k}}Y^{\overset{(k)}{v}_k}\right)&=0,
        \qquad \mu=0,1,\ldots,k-1,
\end{align*}
for all $X,Y\in\mathfrak X(M)$.
\end{proposition}

\subsection{Geodesics on $(T^{(2)}M,\overset{(2)}g)$}

We first calculate the covariant derivatives of $\overset{(2)}\nabla$ and then obtain the geodesic equations on $(T^{(2)}M,\overset{(2)}g)$.  The proof of the following theorem is included below, since it gives the model for the higher-order Koszul computations.

\begin{theorem}\label{teo: Levi--Civita Sasaki 2}
Let $\overset{(2)}\nabla$ be the Levi--Civita connection of $(T^{(2)}M,\overset{(2)}g)$. Then
\begin{align*}
        \overset{(2)}\nabla_{X^{\overset{(2)}{v}_2}}Y^{\overset{(2)}{v}_2}&=0,\\
        \overset{(2)}\nabla_{X^{\overset{(2)}{h}_1}}Y^{\overset{(2)}{v}_2}&=\frac14\left(R(X,Y)u^{(1)}+R(u^{(1)},Y)X\right)^{\overset{(2)}{h}_0},\\
        \overset{(2)}\nabla_{X^{\overset{(2)}{h}_0}}Y^{\overset{(2)}{v}_2}
        &=\left(\frac14R(u^{(2)},Y)X+\frac12(\nabla_{u^{(1)}}R)(u^{(1)},Y)X\right)^{\overset{(2)}{h}_0}\\
        &\quad+\left(\frac14R(u^{(1)},X)Y+\frac14R(u^{(1)},Y)X\right)^{\overset{(2)}{h}_1}
        +(\nabla_XY)^{\overset{(2)}{v}_2},\\
        \overset{(2)}\nabla_{X^{\overset{(2)}{h}_1}}Y^{\overset{(2)}{h}_1}&=0,\\
        \overset{(2)}\nabla_{X^{\overset{(2)}{h}_0}}Y^{\overset{(2)}{h}_1}
        &=\frac12\left(R(u^{(1)},Y)X\right)^{\overset{(2)}{h}_0}+(\nabla_XY)^{\overset{(2)}{h}_1}
        -\left(\frac14R(X,Y)u^{(1)}+\frac14R(X,u^{(1)})Y\right)^{\overset{(2)}{v}_2},\\
        \overset{(2)}\nabla_{X^{\overset{(2)}{h}_0}}Y^{\overset{(2)}{h}_0}
        &=(\nabla_XY)^{\overset{(2)}{h}_0}-\frac12\left(R(X,Y)u^{(1)}\right)^{\overset{(2)}{h}_1}
        -\left(\frac14R(X,Y)u^{(2)}+\frac12(\nabla_{u^{(1)}}R)(X,Y)u^{(1)}\right)^{\overset{(2)}{v}_2},
\end{align*}
for all $X,Y\in\mathfrak X(M)$, where $u^{(1)}=\overset{(2)}{F}_0(u)$, $u^{(2)}=\overset{(2)}{F}_1(u)$, and $u\in T^{(2)}M$.
\end{theorem}

\begin{proof}
The components involving only $\overset{(2)}{h}_0$- and $\overset{(2)}{h}_1$-lifts follow from Lemma \ref{lemma: nablahh} and Theorem \ref{prop: liebrac2}.  The $2$-vertical components involving one $\overset{(2)}{v}_2$-lift are given by Proposition \ref{prop: Levi--Civita kSasaki}.  It remains to obtain the horizontal components of the derivatives involving one $2$-vertical lift.

For example, for $Z\in\mathfrak X(M)$,
\begin{align*}
        \llangle \overset{(2)}\nabla_{X^{\overset{(2)}{h}_0}}Y^{\overset{(2)}{v}_2},Z^{\overset{(2)}{h}_1}\rrangle
        &=-\llangle Y^{\overset{(2)}{v}_2},\overset{(2)}\nabla_{X^{\overset{(2)}{h}_0}}Z^{\overset{(2)}{h}_1}\rrangle\\
        &=\left\langle Y,
        \frac14R(X,Z)u^{(1)}+\frac14R(X,u^{(1)})Z
        \right\rangle\\
        &=\left\langle
        \frac14R(u^{(1)},X)Y+\frac14R(u^{(1)},Y)X,
        Z\right\rangle.
\end{align*}
Thus
\[
        \overset{(2)}{K}_1\left(\overset{(2)}\nabla_{X^{\overset{(2)}{h}_0}}Y^{\overset{(2)}{v}_2}\right)
        =\frac14R(u^{(1)},X)Y+\frac14R(u^{(1)},Y)X.
\]
The remaining components are obtained analogously by applying the Koszul formula against $\overset{(2)}{h}_0$-, $\overset{(2)}{h}_1$-, and $\overset{(2)}{v}_2$-lifts and using the second-order bracket formulas from Theorem \ref{prop: liebrac2}.  Combining these components gives the displayed formulas.
\qed \end{proof}

Let $\Gamma$ be a curve on $T^{(2)}M$. Let $q=\tau_2\circ\Gamma$, and let
\[
        V^{(1)}=\overset{(2)}{F}_0\circ\Gamma,
        \qquad
        V^{(2)}=\overset{(2)}{F}_1\circ\Gamma.
\]
We also write
\[
        Y^{(\alpha)}=\overset{(2)}{K}_\alpha\circ\dot\Gamma,
        \qquad \alpha=1,2.
\]
Then
\begin{equation}\label{eq: decomposition gamma T^2}
        \dot\Gamma=\dot q^{\overset{(2)}{h}_0}+(Y^{(1)})^{\overset{(2)}{h}_1}+(Y^{(2)})^{\overset{(2)}{v}_2}.
\end{equation}

\begin{theorem}\label{teo: geod eq 2}
A curve $\Gamma$ is a Riemannian geodesic on $(T^{(2)}M,\overset{(2)}g)$ if and only if
\begin{align*}
        0&=\nabla_{\dot q}\dot q
        +R(V^{(1)},Y^{(1)})\dot q
        +\frac12R(V^{(2)},Y^{(2)})\dot q
        +(\nabla_{V^{(1)}}R)(V^{(1)},Y^{(2)})\dot q\\
        &\quad+\frac12R(Y^{(1)},Y^{(2)})V^{(1)}
        +\frac12R(V^{(1)},Y^{(2)})Y^{(1)},\\
        0&=\nabla_{\dot q}Y^{(1)}
        +\frac12R(V^{(1)},\dot q)Y^{(2)}
        +\frac12R(V^{(1)},Y^{(2)})\dot q,\\
        0&=\nabla_{\dot q}Y^{(2)}.
\end{align*}
\end{theorem}

\begin{proof}
Using the decomposition given in Equation \eqref{eq: decomposition gamma T^2} and the torsion-free property of $\overset{(2)}\nabla$, the covariant acceleration of $\Gamma$ can be written as
\begin{align*}
        \overset{(2)}\nabla_{\dot\Gamma}\dot\Gamma
        &=\overset{(2)}\nabla_{\dot q^{\overset{(2)}{h}_0}}\dot q^{\overset{(2)}{h}_0}
        +2\overset{(2)}\nabla_{\dot q^{\overset{(2)}{h}_0}}(Y^{(1)})^{\overset{(2)}{h}_1}
        -[\dot q^{\overset{(2)}{h}_0},(Y^{(1)})^{\overset{(2)}{h}_1}]\\
        &\quad+2\overset{(2)}\nabla_{\dot q^{\overset{(2)}{h}_0}}(Y^{(2)})^{\overset{(2)}{v}_2}
        -[\dot q^{\overset{(2)}{h}_0},(Y^{(2)})^{\overset{(2)}{v}_2}]
        +2\overset{(2)}\nabla_{(Y^{(1)})^{\overset{(2)}{h}_1}}(Y^{(2)})^{\overset{(2)}{v}_2}.
\end{align*}
Applying Theorem \ref{teo: Levi--Civita Sasaki 2} and the corresponding bracket formulas gives
\begin{align*}
        \tau_{2*}\left(\overset{(2)}\nabla_{\dot\Gamma}\dot\Gamma\right)
        &=\nabla_{\dot q}\dot q
        +R(V^{(1)},Y^{(1)})\dot q
        +\frac12R(V^{(2)},Y^{(2)})\dot q
        +(\nabla_{V^{(1)}}R)(V^{(1)},Y^{(2)})\dot q\\
        &\quad+\frac12R(Y^{(1)},Y^{(2)})V^{(1)}
        +\frac12R(V^{(1)},Y^{(2)})Y^{(1)},\\
        \overset{(2)}{K}_1\left(\overset{(2)}\nabla_{\dot\Gamma}\dot\Gamma\right)
        &=\nabla_{\dot q}Y^{(1)}+\frac12R(V^{(1)},\dot q)Y^{(2)}+\frac12R(V^{(1)},Y^{(2)})\dot q,\\
        \overset{(2)}{K}_2\left(\overset{(2)}\nabla_{\dot\Gamma}\dot\Gamma\right)
        &=\nabla_{\dot q}Y^{(2)}.
\end{align*}
Since the maps $\tau_{2*},\overset{(2)}{K}_1,\overset{(2)}{K}_2$ determine the decomposition of $TT^{(2)}M$, the vanishing of the covariant acceleration is equivalent to the three displayed equations in Theorem \ref{teo: geod eq 2}.
\qed \end{proof}

\begin{remark}\label{rem: horizgeo}
If $\Gamma$ is a $2$-horizontal geodesic on $T^{(2)}M$, then $\overset{(2)}{\tau}_1\circ\Gamma$ is a geodesic on $TM$.
\end{remark}

Now let $\Gamma=j^2q$ be the second-jet lift of a curve $q$ on $M$. Then
\[
        V^{(1)}=\dot q,
        \qquad
        V^{(2)}=\nabla_{\dot q}\dot q,
\]
and
\begin{equation}\label{eq: decomposition curve T^2}
        \dot\Gamma=\dot q^{\overset{(2)}{h}_0}+(\nabla_{\dot q}\dot q)^{\overset{(2)}{h}_1}
        +\frac12\left(\nabla_{\dot q}^2\dot q\right)^{\overset{(2)}{v}_2}.
\end{equation}

\begin{corollary}\label{cor: 2jetcurve}
Let $q$ be a curve on $M$. The second-jet lift $\Gamma=j^2q$ is a geodesic on $T^{(2)}M$ if and only if $q$ is a geodesic on $M$.
\end{corollary}

\begin{proof}
If $q$ is a geodesic on $M$, then Equation \eqref{eq: decomposition curve T^2} gives $\dot\Gamma=\dot q^{\overset{(2)}{h}_0}$ and Theorem \ref{teo: Levi--Civita Sasaki 2} gives $\overset{(2)}\nabla_{\dot\Gamma}\dot\Gamma=0$.

Conversely, suppose that $\Gamma=j^2q$ is a geodesic.  Then Theorem \ref{teo: geod eq 2} gives, after substituting
\[
        V^{(1)}=\dot q,
        \qquad
        V^{(2)}=\nabla_{\dot q}\dot q,
        \qquad
        Y^{(1)}=\nabla_{\dot q}\dot q,
        \qquad
        Y^{(2)}=\frac12\nabla^2_{\dot q}\dot q,
\]
a system whose final two equations imply
\[
        \nabla_{\dot q}^2\dot q+\frac14R(\dot q,\nabla^2_{\dot q}\dot q)\dot q=0,
        \qquad
        \nabla_{\dot q}^3\dot q=0.
\]
Let $A=\nabla^2_{\dot q}\dot q$.  Taking the inner product of the first equation with $\dot q$ gives $\langle A,\dot q\rangle=0$, and the second equation gives $\nabla_{\dot q}A=0$.  Differentiating $\langle A,\dot q\rangle=0$ once gives $\langle A,\nabla_{\dot q}\dot q\rangle=0$.  The first component equation, paired with $\dot q$, then gives $\langle\nabla_{\dot q}\dot q,\dot q\rangle=0$.  Differentiating this last identity and using $\langle A,\dot q\rangle=0$ yields $\|\nabla_{\dot q}\dot q\|^2=0$.  Hence $q$ is a geodesic.
\qed \end{proof}

\subsection{Geodesics on $(T^{(3)}M,\overset{(3)}g)$}

We now calculate the covariant derivatives of $\overset{(3)}\nabla$ and obtain the geodesic equations on $(T^{(3)}M,\overset{(3)}g)$.  The proof of the next theorem is given in Appendix \ref{app:sasaki-third-order-lc}.

\begin{theorem}\label{teo: Levi--Civita Sasaki 3}
Let $\overset{(3)}\nabla$ be the Levi--Civita connection of $(T^{(3)}M,\overset{(3)}g)$. Then
\begin{align*}
        \overset{(3)}\nabla_{X^{\overset{(3)}{v}_3}}Y^{\overset{(3)}{v}_3}&=0,\\
        \overset{(3)}\nabla_{X^{\overset{(3)}{h}_2}}Y^{\overset{(3)}{v}_3}
        &=\left(\frac12R(X,Y)u^{(1)}+\frac16R(u^{(1)},X)Y\right)^{\overset{(3)}{h}_0},\\
        \overset{(3)}\nabla_{X^{\overset{(3)}{h}_1}}Y^{\overset{(3)}{v}_3}
        &=\left(\frac14R(u^{(2)},Y)X+\frac14(\nabla_{u^{(1)}}R)(u^{(1)},Y)X+\frac1{12}R(X,u^{(2)})Y\right.\\
        &\qquad\left.+\frac1{12}(\nabla_{u^{(1)}}R)(X,u^{(1)})Y\right)^{\overset{(3)}{h}_0}
        +\frac14\left(R(u^{(1)},Y)X\right)^{\overset{(3)}{h}_1},\\
        \overset{(3)}\nabla_{X^{\overset{(3)}{h}_0}}Y^{\overset{(3)}{v}_3}
        &=\left(\frac1{12}R(u^{(3)},Y)X+\frac16(\nabla_{u^{(1)}}R)(u^{(2)},Y)X-\frac16(\nabla_{u^{(2)}}R)(u^{(1)},Y)X\right.\\
        &\qquad\left.+\frac14(\nabla_{u^{(1)}}^2R)(u^{(1)},Y)X\right)^{\overset{(3)}{h}_0}\\
        &\quad+\left(\frac14R(u^{(2)},Y)X+\frac14(\nabla_{u^{(1)}}R)(u^{(1)},Y)X-\frac1{12}R(X,Y)u^{(2)}\right.\\
        &\qquad\left.-\frac1{12}(\nabla_{u^{(1)}}R)(X,Y)u^{(1)}\right)^{\overset{(3)}{h}_1}\\
        &\quad+\left(\frac16R(X,Y)u^{(1)}+\frac12R(u^{(1)},X)Y\right)^{\overset{(3)}{h}_2}
        +(\nabla_XY)^{\overset{(3)}{v}_3},\\
        \overset{(3)}\nabla_{X^{\overset{(3)}{h}_2}}Y^{\overset{(3)}{h}_2}&=0,\\
        \overset{(3)}\nabla_{X^{\overset{(3)}{h}_1}}Y^{\overset{(3)}{h}_2}&=\frac14\left(R(X,Y)u^{(1)}+R(u^{(1)},Y)X\right)^{\overset{(3)}{h}_0},\\
        \overset{(3)}\nabla_{X^{\overset{(3)}{h}_0}}Y^{\overset{(3)}{h}_2}
        &=\left(\frac14R(u^{(2)},Y)X+\frac12(\nabla_{u^{(1)}}R)(u^{(1)},Y)X\right)^{\overset{(3)}{h}_0}\\
        &\quad+\left(\frac14R(u^{(1)},X)Y+\frac14R(u^{(1)},Y)X\right)^{\overset{(3)}{h}_1}
        +(\nabla_XY)^{\overset{(3)}{h}_2}\\
        &\quad+\left(\frac12R(u^{(1)},X)Y-\frac16R(u^{(1)},Y)X\right)^{\overset{(3)}{v}_3},\\
        \overset{(3)}\nabla_{X^{\overset{(3)}{h}_1}}Y^{\overset{(3)}{h}_1}&=-\frac14\left(R(X,Y)u^{(1)}\right)^{\overset{(3)}{v}_3},\\
        \overset{(3)}\nabla_{X^{\overset{(3)}{h}_0}}Y^{\overset{(3)}{h}_1}
        &=\frac12\left(R(u^{(1)},Y)X\right)^{\overset{(3)}{h}_0}+(\nabla_XY)^{\overset{(3)}{h}_1}
        -\left(\frac14R(X,Y)u^{(1)}+\frac14R(X,u^{(1)})Y\right)^{\overset{(3)}{h}_2}\\
        &\quad-\left(\frac14R(X,Y)u^{(2)}+\frac14(\nabla_{u^{(1)}}R)(X,Y)u^{(1)}+\frac1{12}R(Y,u^{(2)})X\right.\\
        &\qquad\left.+\frac1{12}(\nabla_{u^{(1)}}R)(Y,u^{(1)})X\right)^{\overset{(3)}{v}_3},\\
        \overset{(3)}\nabla_{X^{\overset{(3)}{h}_0}}Y^{\overset{(3)}{h}_0}
        &=(\nabla_XY)^{\overset{(3)}{h}_0}-\frac12\left(R(X,Y)u^{(1)}\right)^{\overset{(3)}{h}_1}
        -\left(\frac14R(X,Y)u^{(2)}+\frac12(\nabla_{u^{(1)}}R)(X,Y)u^{(1)}\right)^{\overset{(3)}{h}_2}\\
        &\quad-\left(\frac1{12}R(X,Y)u^{(3)}+\frac16(\nabla_{u^{(1)}}R)(X,Y)u^{(2)}-\frac16(\nabla_{u^{(2)}}R)(X,Y)u^{(1)}\right.\\
        &\qquad\left.+\frac14(\nabla_{u^{(1)}}^2R)(X,Y)u^{(1)}\right)^{\overset{(3)}{v}_3},
\end{align*}
for all $X,Y\in\mathfrak X(M)$, where $u^{(1)}=\overset{(3)}{F}_0(u)$, $u^{(2)}=\overset{(3)}{F}_1(u)$, $u^{(3)}=\overset{(3)}{F}_2(u)$, and $u\in T^{(3)}M$.
\end{theorem}

Let $\Gamma$ be a curve on $T^{(3)}M$. Let $q=\tau_3\circ\Gamma$, and let
\[
        V^{(1)}=\overset{(3)}{F}_0\circ\Gamma,
        \qquad
        V^{(2)}=\overset{(3)}{F}_1\circ\Gamma,
        \qquad
        V^{(3)}=\overset{(3)}{F}_2\circ\Gamma.
\]
We also write
\[
        Y^{(\alpha)}=\overset{(3)}{K}_\alpha\circ\dot\Gamma,
        \qquad \alpha=1,2,3.
\]
Then
\begin{equation}\label{eq: decomposition gamma T^3}
        \dot\Gamma=\dot q^{\overset{(3)}{h}_0}+(Y^{(1)})^{\overset{(3)}{h}_1}+(Y^{(2)})^{\overset{(3)}{h}_2}+(Y^{(3)})^{\overset{(3)}{v}_3}.
\end{equation}

\begin{theorem}\label{teo: geod eq 3}
A curve $\Gamma$ on $T^{(3)}M$ is a Riemannian geodesic of the $3$-Sasaki metric if and only if
\begin{align*}
        0&=\nabla_{\dot q}\dot q+R(V^{(1)},Y^{(1)})\dot q
        +\frac12R(V^{(2)},Y^{(2)})\dot q+(\nabla_{V^{(1)}}R)(V^{(1)},Y^{(2)})\dot q\\
        &\quad+\frac16R(V^{(3)},Y^{(3)})\dot q
        +\frac13(\nabla_{V^{(1)}}R)(V^{(2)},Y^{(3)})\dot q
        -\frac13(\nabla_{V^{(2)}}R)(V^{(1)},Y^{(3)})\dot q\\
        &\quad+\frac12(\nabla_{V^{(1)}}^2R)(V^{(1)},Y^{(3)})\dot q
        +\frac12R(Y^{(1)},Y^{(2)})V^{(1)}+
        \frac12R(V^{(1)},Y^{(2)})Y^{(1)}\\
        &\quad+\frac12R(V^{(2)},Y^{(3)})Y^{(1)}
        +\frac12(\nabla_{V^{(1)}}R)(V^{(1)},Y^{(3)})Y^{(1)}
        +\frac16R(Y^{(1)},V^{(2)})Y^{(3)}\\
        &\quad+\frac16(\nabla_{V^{(1)}}R)(Y^{(1)},V^{(1)})Y^{(3)}
        +R(Y^{(2)},Y^{(3)})V^{(1)}
        +\frac13R(V^{(1)},Y^{(2)})Y^{(3)},\\
        0&=\nabla_{\dot q}Y^{(1)}
        +\frac12R(V^{(1)},\dot q)Y^{(2)}
        +\frac12R(V^{(1)},Y^{(2)})\dot q
        +\frac12R(V^{(2)},Y^{(3)})\dot q\\
        &\quad+\frac12(\nabla_{V^{(1)}}R)(V^{(1)},Y^{(3)})\dot q
        -\frac16(\nabla_{V^{(1)}}R)(\dot q,Y^{(3)})V^{(1)}
        -\frac16R(\dot q,Y^{(3)})V^{(2)}\\
        &\quad+\frac12R(V^{(1)},Y^{(3)})Y^{(1)},\\
        0&=\nabla_{\dot q}Y^{(2)}+R(V^{(1)},\dot q)Y^{(3)}+\frac13R(\dot q,Y^{(3)})V^{(1)},\\
        0&=\nabla_{\dot q}Y^{(3)}.
\end{align*}
\end{theorem}

The proof is given in Appendix \ref{app:sasaki-third-order-geodesics}.

\begin{remark}\label{cor: hgeod3}
If $\Gamma$ is a $3$-horizontal geodesic on $T^{(3)}M$, then $\overset{(3)}{\tau}_2\circ\Gamma$ is a geodesic on $T^{(2)}M$.
\end{remark}

If $\Gamma=j^3q$ is the third-jet lift of a curve $q$ on $M$, then
\[
        V^{(1)}=\dot q,
        \qquad
        V^{(2)}=\nabla_{\dot q}\dot q,
        \qquad
        V^{(3)}=\nabla^2_{\dot q}\dot q,
\]
and
\[
        \overset{(3)}{K}_1(\dot\Gamma)=\nabla_{\dot q}\dot q,
        \qquad
        \overset{(3)}{K}_2(\dot\Gamma)=\frac1{2!}\nabla^2_{\dot q}\dot q,
        \qquad
        \overset{(3)}{K}_3(\dot\Gamma)=\frac1{3!}\nabla^3_{\dot q}\dot q.
\]

The following corollary to Theorem \ref{teo: geod eq 3} can be proven by similar arguments to those used in the proof of Corollary  \ref{cor: 2jetcurve}.
\begin{corollary}\label{cor: 1jetgeod3}
Let $q$ be a curve on $M$. The third-jet lift $\Gamma=j^3q$ is a geodesic on $T^{(3)}M$ if and only if $q$ is a geodesic on $M$.
\end{corollary}

\subsection{Geodesics on $(T^{(k)}M,\overset{(k)}g)$}

The full geodesic system on $T^{(k)}M$ becomes combinatorially large.  However, the final two components admit a uniform closed form, which is sufficient for the jet lift characterization below.

\begin{lemma}\label{lem: nablaww}
Let $\overset{(k)}\nabla$ be the Levi--Civita connection of $(T^{(k)}M,\overset{(k)}g)$. Then, for all $W\in\mathfrak X(T^{(k)}M)$,
\begin{align}
        \overset{(k)}{K}_k\circ\left(\overset{(k)}\nabla_WW\right)
        &=\nabla_{\tau_{k*}\circ W}(\overset{(k)}{K}_k\circ W),\label{eq: nablaxx}\\
        \overset{(k)}{K}_{k-1}\circ\left(\overset{(k)}\nabla_WW\right)
        &=\nabla_{\tau_{k*}\circ W}(\overset{(k)}{K}_{k-1}\circ W)
        +\frac1kR(u^{(1)},\overset{(k)}{K}_{k-1}\circ W)(\tau_{k*}\circ W)\nonumber\\
        &\quad+\frac{k-1}{k}R(u^{(1)},\tau_{k*}\circ W)(\overset{(k)}{K}_{k-1}\circ W),\label{eq: nablaxx2}
\end{align}
where $u^{(1)}=\overset{(k)}{F}_0(u)$, $u\in T^{(k)}M$.
\end{lemma}

\begin{proof}
Let $W\in\mathfrak X(T^{(k)}M)$.  Using Equation \eqref{eq: conKk-app}, the bracket formulas involving $\overset{(k)}{v}_k$-lifts, and Proposition \ref{prop: Levi--Civita kSasaki}, we obtain
\begin{align*}
        \overset{(k)}{K}_k\circ\left(\overset{(k)}\nabla_WW\right)
        &=2\overset{(k)}{K}_k\circ\left(\overset{(k)}\nabla_{\overset{(k)}{h}_0\circ W}(\overset{(k)}{v}_k\circ W)\right)
        +\overset{(k)}{K}_k\circ\left([\overset{(k)}{v}_k\circ W,\overset{(k)}{h}_0\circ W]\right)\\
        &=\nabla_{\tau_{k*}\circ W}(\overset{(k)}{K}_k\circ W).
\end{align*}
This proves Equation \eqref{eq: nablaxx}.

For the next component, let $Z\in\mathfrak X(M)$.  Applying the Koszul formula gives
\begin{align*}
        \llangle\overset{(k)}\nabla_WW,Z^{\overset{(k)}{h}_{k-1}}\rrangle
        &=(\overset{(k)}{h}_0\circ W)\llangle W,Z^{\overset{(k)}{h}_{k-1}}\rrangle
        -\llangle \overset{(k)}{h}_{k-1}\circ W,(\nabla_{\tau_{k*}\circ W}Z)^{\overset{(k)}{h}_{k-1}}\rrangle\\
        &\quad+\frac1k\llangle \overset{(k)}{v}_k\circ W,
        (R(\overset{(k)}{K}_{k-1}\circ W,Z)u^{(1)}+(k-1)R(\overset{(k)}{K}_{k-1}\circ W,u^{(1)})Z)^{\overset{(k)}{v}_k}\rrangle.
\end{align*}
Using the curvature symmetries, this becomes
\begin{align*}
        \llangle\overset{(k)}\nabla_WW,Z^{\overset{(k)}{h}_{k-1}}\rrangle
        &=\left\langle \nabla_{\tau_{k*}\circ W}(\overset{(k)}{K}_{k-1}\circ W),Z\right\rangle\\
        &\quad+\frac1k\left\langle R(u^{(1)},\tau_{k*}\circ W)(\overset{(k)}{K}_{k-1}\circ W),Z\right\rangle\\
        &\quad+\frac{k-1}{k}\left\langle R(u^{(1)},\overset{(k)}{K}_{k-1}\circ W)(\tau_{k*}\circ W),Z\right\rangle.
\end{align*}
Since this holds for all $Z$, Equation \eqref{eq: nablaxx2} follows.
\qed \end{proof}

Let $\Gamma$ be a curve on $T^{(k)}M$ and let $q=\tau_k\circ\Gamma$.  We denote by
\[
        V^{(\alpha)}=\overset{(k)}{F}_{\alpha-1}\circ\Gamma,
        \qquad \alpha=1,\ldots,k,
\]
the vector fields along $q$ determined by the connection tower vector bundle structure, and by
\[
        Y^{(\alpha)}=\overset{(k)}{K}_\alpha\circ\dot\Gamma,
        \qquad \alpha=1,\ldots,k,
\]
the components of the velocity vector field. Then
\begin{equation}\label{eq: decomposition gamma general T^k}
        \dot\Gamma=\dot q^{\overset{(k)}{h}_0}+(Y^{(1)})^{\overset{(k)}{h}_1}+\cdots+(Y^{(k)})^{\overset{(k)}{v}_k}.
\end{equation}

\begin{lemma}\label{lem: C}
Let $\Gamma$ be a curve on $T^{(k)}M$. If $\Gamma$ is a geodesic on $T^{(k)}M$, then
\begin{align}
        \nabla_{\dot q}Y^{(k-1)}
        +\frac1kR(V^{(1)},Y^{(k-1)})\dot q
        +\frac{k-1}{k}R(V^{(1)},\dot q)Y^{(k-1)}&=0,\label{eq: 2geoeq}\\
        \nabla_{\dot q}Y^{(k)}&=0.\label{eq: 2geoeq1}
\end{align}
\end{lemma}

\begin{proof}
This follows immediately by applying Lemma \ref{lem: nablaww} to $W=\dot\Gamma$ and using $\overset{(k)}\nabla_{\dot\Gamma}\dot\Gamma=0$.
\qed \end{proof}

\begin{corollary}\label{cor: k lift}
Let $q$ be a curve on $M$. If the $k$-jet lift $\Gamma=j^kq$ is a geodesic on $T^{(k)}M$, then $\Gamma$ is $k$-horizontal.
\end{corollary}

\begin{proof}
Suppose that $\Gamma=j^kq$ is a geodesic.  Then, by Lemma \ref{lem: C} and Equation \eqref{eq: decomposition curve T^k}, the vector field
\[
        A:=\nabla_{\dot q}^k\dot q
\]
satisfies
\[
        A+\frac1kR(\dot q,A)\dot q=0,
        \qquad
        \nabla_{\dot q}A=0.
\]
Taking the inner product of the first equation with $\dot q$ gives $\langle A,\dot q\rangle=0$, since $R(\dot q,A)$ is skew-adjoint.  Since $A$ is parallel along $q$, repeated differentiation yields
\[
        0=\frac{d^r}{dt^r}\langle A,\dot q\rangle
        =\left\langle A,\nabla_{\dot q}^r\dot q\right\rangle,
        \qquad r=0,1,\ldots,k.
\]
Taking $r=k$ gives $\|A\|^2=0$. Hence $A=0$, which is precisely the condition that $\Gamma$ is $k$-horizontal.
\qed \end{proof}

\begin{proposition}\label{Prop: horizontalgeo}
If $\Gamma$ is a $k$-horizontal geodesic on $T^{(k)}M$, then $\overset{(k)}{\tau}_{\!k-1}\circ\Gamma$ is a geodesic on $T^{(k-1)}M$.
\end{proposition}

\begin{proof}
Suppose that $\Gamma$ is a $k$-horizontal geodesic on $T^{(k)}M$, that is, $\overset{(k)}\nabla_{\dot\Gamma}\dot\Gamma=0$ and $\overset{(k)}{K}_k\circ\dot\Gamma=0$. Then $\dot\Gamma=\overset{(k)}{n}_k\circ\dot\Gamma$. Let $\widetilde\Gamma=\overset{(k)}{\tau}_{\!k-1}\circ\Gamma$.  Its velocity vector field is
\[
        \dot{\widetilde\Gamma}=\overset{(k)}{\tau}_{\!k-1*}\circ \overset{(k)}{n}_k\circ\dot\Gamma.
\]
By Lemma \ref{lemma: nablahh},
\[
        0=\overset{(k)}{K}_\mu\circ\left(\overset{(k)}\nabla_{\overset{(k)}{n}_k\circ\dot\Gamma}(\overset{(k)}{n}_k\circ\dot\Gamma)\right)
        =\overset{(k-1)}{K}_{\!\!\!\!\mu}\circ\left(\overset{(k-1)}\nabla_{\!\!\!\!\dot{\widetilde\Gamma}}\dot{\widetilde\Gamma}\right),
\]
for $\mu=0,1,\ldots,k-1$.  Hence $\overset{(k-1)}\nabla_{\!\!\!\!\dot{\widetilde\Gamma}}\dot{\widetilde\Gamma}=0$.
\qed \end{proof}

\begin{remark}\label{rem: hor geod}
It follows that an $\alpha$-horizontal geodesic on $T^{(k)}M$ projects onto a geodesic on $T^{(\alpha-1)}M$, for each $\alpha=1,\ldots,k$.
\end{remark}

\begin{remark}
The results in Proposition \ref{Prop: horizontalgeo} and Remark \ref{rem: hor geod} also follow from the fact that the projections $\overset{(k)}{\tau}_{\!\alpha}$ are Riemannian submersions; see \cite[Proposition 3.1]{Herm1960}.  The proof above records the corresponding connection map argument.
\end{remark}

\begin{proposition}\label{prop: kth-jet-geodesic}
Let $q$ be a curve on $M$. The $k$-jet $\Gamma=j^kq$ is a geodesic on $T^{(k)}M$ if and only if $q$ is a geodesic on $M$.
\end{proposition}

\begin{proof}
Suppose that $\Gamma=j^kq$ is a geodesic on $T^{(k)}M$. By Corollary \ref{cor: k lift}, $\Gamma$ is $k$-horizontal.  Therefore, by Proposition \ref{Prop: horizontalgeo}, the curve $\overset{(k)}{\tau}_{\!k-1}\circ\Gamma$ is a geodesic on $T^{(k-1)}M$. Since this projection is the $(k-1)$-jet of $q$, the argument can be repeated.  Inductively, all lower jet lifts are geodesics, and in particular $q$ is a geodesic on $M$.

Conversely, if $q$ is a geodesic on $M$, then all higher covariant derivatives $\nabla^r_{\dot q}\dot q$ vanish for $r\geq1$. Hence the velocity field of $j^kq$ is $1$-horizontal, and Lemma \ref{lemma: nablahh} gives $\overset{(k)}\nabla_{\dot\Gamma}\dot\Gamma=0$.
\qed \end{proof}

\appendix
\makeatletter
\@addtoreset{theorem}{section}
\makeatother
\renewcommand{\thetheorem}{\thesection.\arabic{theorem}}
\section{Background material for connection towers}\label{app:connection-background}

\subsection{Adapted-basis details}\label{app:adapted-basis-details}

This appendix records the coordinate relations between the coordinate bases and the adapted bases associated with a connection map.

The adapted covectors are
\[
        \delta q^{(\alpha)i}
        =
        dq^{(\alpha)i}
        +(K_1)^i_jdq^{(\alpha-1)j}
        +\cdots+
        (K_\alpha)^i_jdq^{(0)j},
        \qquad \alpha=1,\ldots,k,
\]
and $\delta q^{(0)i}=dq^{(0)i}$.  Let $\delta/\delta q^{(\alpha)i}$ denote the dual basis.  The dual coefficients $(C_\alpha)^i_j$ are defined by
\[
        \frac{\delta}{\delta q^{(\alpha)i}}
        =
        \frac{\partial}{\partial q^{(\alpha)i}}
        -(C_1)^j_i\frac{\partial}{\partial q^{(\alpha+1)j}}
        -\cdots
        -(C_{k-\alpha})^j_i\frac{\partial}{\partial q^{(k)j}}.
\]

\begin{proposition}[Miron \cite{MironBook1997}]\label{prop: coordinate basis to adapted basis}
The dual coefficients satisfy
\[
        dq^{(\alpha)i}
        \left(
        \frac{\delta}{\delta q^{(\mu)j}}
        \right)
        =-(C_{\alpha-\mu})^i_j
\]
for all $\alpha=1,\ldots,k$, $\mu=0,\ldots,\alpha-1$, and $i,j=1,\ldots,n$.  In particular,
\[
        dq^{(\alpha)i}
        =
        \delta q^{(\alpha)i}
        -(C_1)^i_j\delta q^{(\alpha-1)j}
        -\cdots
        -(C_\alpha)^i_j\delta q^{(0)j},
\]
for all $\alpha=1,\ldots,k$.  Similarly,
\[
        \frac{\partial}{\partial q^{(\alpha)i}}
        =
        \frac{\delta}{\delta q^{(\alpha)i}}
        +(K_1)^j_i\frac{\delta}{\delta q^{(\alpha+1)j}}
        +\cdots+
        (K_{k-\alpha})^j_i\frac{\delta}{\delta q^{(k)j}}.
\]
\end{proposition}

\begin{proof}
The result follows by inverting the triangular change-of-basis matrix relating the coordinate covectors $dq^{(\alpha)i}$ to the adapted covectors $\delta q^{(\alpha)i}$.  The vector-field identity is the dual triangular relation.
\qed \end{proof}

\begin{proposition}[Miron \cite{MironBook1997}]\label{prop: dual-coefficient-recursion}
The dual coefficients are constructed recursively in terms of the connection coefficients by
\begin{equation}\label{eq: dual coefficients recursion}
        (C_\alpha)^i_j
        =
        (K_\alpha)^i_j
        -
        \sum_{\mu=1}^{\alpha-1}
        (K_\mu)^i_l(C_{\alpha-\mu})^l_j,
        \qquad
        (C_1)^i_j=(K_1)^i_j.
\end{equation}
\end{proposition}

\begin{proof}
From the definition of the adapted covector basis,
\[
        \delta q^{(\alpha)i}
        =
        dq^{(\alpha)i}
        +(K_1)^i_l dq^{(\alpha-1)l}
        +\cdots+
        (K_\alpha)^i_l dq^{(0)l}.
\]
Evaluating this on $\delta/\delta q^{(0)j}$ gives
\[
        0
        =
        dq^{(\alpha)i}\left(\frac{\delta}{\delta q^{(0)j}}\right)
        +(K_1)^i_l
        dq^{(\alpha-1)l}\left(\frac{\delta}{\delta q^{(0)j}}\right)
        +\cdots+
        (K_\alpha)^i_j.
\]
Using Proposition \ref{prop: coordinate basis to adapted basis} yields
\[
        -(C_\alpha)^i_j
        -(K_1)^i_l(C_{\alpha-1})^l_j
        -\cdots
        -(K_{\alpha-1})^i_l(C_1)^l_j
        +(K_\alpha)^i_j=0,
\]
which is precisely Equation \eqref{eq: dual coefficients recursion}.
\qed \end{proof}

\subsection{Technical lemmas for lifted vector fields}\label{app:lift-technical-lemmas}

This appendix contains the technical lift identities used in the computation of Lie brackets.

\begin{proof}[Proof of Lemma \ref{lemma: prop}]
The linearity identities follow directly from the defining characterization of
the lift.  Indeed, the maps $\tau_{k\ast}, \overset{(k)}{K}_1,\ldots,\overset{(k)}{K}_k$
are fiberwise linear on each tangent space \(T_uT^{(k)}M\).  Hence the unique
vector whose \(\alpha\)-component is prescribed and whose remaining components
are zero depends linearly on the prescribed vector in \(T_{\tau_k(u)}M\).  This
gives
\[
        (X+Y)^{\overset{(k)}{h}_\alpha}=X^{\overset{(k)}{h}_\alpha}+Y^{\overset{(k)}{h}_\alpha}, \qquad (cX)^{{\overset{(k)}{h}_\alpha}} = cX^{\overset{(k)}{h}_\alpha}.
\]
Similarly, multiplying the prescribed vector by \(f(\tau_k(u))\) gives
\[
        (fX)^{\overset{(k)}{h}_\alpha}=(f\circ\tau_k)X^{\overset{(k)}{h}_\alpha}.
\]
Since
\[
        \frac{\delta}{\delta q^{(\alpha)j}}=(\partial_j)^{\overset{(k)}{h}_\alpha},
        \qquad \alpha=0,\ldots,k,
\]
the coordinate formula follows by applying the preceding identity to
\(X=X^j\partial_j\):
\[
        X^{\overset{(k)}{h}_\alpha}=(X^j\circ\tau_k)(\partial_j)^{\overset{(k)}{h}_\alpha}
        =(X^j\circ\tau_k)\frac{\delta}{\delta q^{(\alpha)j}}.
\]
The final formulas follow immediately from the fact that $Y(f \circ \tau_k) = df(\tau_{k\ast} Y)$ for all $Y \in \mathfrak{X}(T^{(k)}M)$. 
\qed \end{proof}

\begin{lemma}\label{lemma: Brackets fX gY}
Let $X,Y\in\mathfrak X(M)$ and $f,g\in C^\infty(M)$.  Then
\[
\begin{aligned}
\big[(fX)^{\overset{(k)}{h}_\alpha},(gY)^{\overset{(k)}{h}_\beta}\big]
={}&
\delta_{0\alpha}(f\circ\tau_k)(Xg)Y^{\overset{(k)}{h}_\beta}
-\delta_{0\beta}(g\circ\tau_k)(Yf)X^{\overset{(k)}{h}_\alpha} +(f\circ\tau_k)(g\circ\tau_k)\big[X^{\overset{(k)}{h}_\alpha},Y^{\overset{(k)}{h}_\beta}\big],
\end{aligned}
\]
for $\alpha,\beta=0,\ldots,k$.
\end{lemma}

\begin{proof}
Using Lemma \ref{lemma: fX},
\[
\begin{aligned}
\big[(fX)^{\overset{(k)}{h}_\alpha},(gY)^{\overset{(k)}{h}_\beta}\big]
={}&
(f\circ\tau_k)X^{\overset{(k)}{h}_\alpha}(g\circ\tau_k)Y^{\overset{(k)}{h}_\beta}
-(g\circ\tau_k)Y^{\overset{(k)}{h}_\beta}(f\circ\tau_k)X^{\overset{(k)}{h}_\alpha}
+(f\circ\tau_k)(g\circ\tau_k)\big[X^{\overset{(k)}{h}_\alpha},Y^{\overset{(k)}{h}_\beta}\big].
\end{aligned}
\]
The result follows from Lemma \ref{lemma: prop}.
\qed \end{proof}

The $\overset{(k)}{h}_\alpha$-lifts of vector fields on $M$ to $T^{(k)}M$ are compatible with the projections of the connection tower.

\section{Riemannian identities and the Levi--Civita-induced tower}\label{app:riemannian-lc-tower}
\subsection{Riemannian identities used in Section~\ref{sec:riemannian-induced-connection-tower}}
\label{app:riemannian-identities}

This appendix records the curvature identities used in the computations of Section~\ref{sec:riemannian-induced-connection-tower}.  We use the convention
\[
        R(X,Y)Z=\nabla_X\nabla_YZ-\nabla_Y\nabla_XZ-\nabla_{[X,Y]}Z.
\]

\begin{lemma}\label{lemma: curvature symmetries}
The following identities hold for all $X,Y,Z,W\in\mathfrak X(M)$:
\begin{align*}
&\textit{1. (First skew-symmetry identity)}\quad R(X,Y)+R(Y,X)=0;\\
&\textit{2. (Second skew-symmetry identity)}\quad
\langle R(X,Y)Z,W\rangle+\langle R(X,Y)W,Z\rangle=0;\\
&\textit{3. (Symmetry by pairs identity)}\quad
\langle R(X,Y)Z,W\rangle=\langle R(W,Z)Y,X\rangle;\\
&\textit{4. (First Bianchi identity)}\quad
R(X,Y)Z+R(Y,Z)X+R(Z,X)Y=0.
\end{align*}
\end{lemma}

The covariant derivative of $R$ in the direction of $X$ is defined by
\begin{equation}\label{eq:rg2}
\begin{aligned}
(\nabla_XR)(Y,Z)W={}&\nabla_X(R(Y,Z)W)-R(\nabla_XY,Z)W\\
&-R(Y,\nabla_XZ)W-R(Y,Z)\nabla_XW.
\end{aligned}
\end{equation}
We regard $\nabla R$ as the $(4,1)$-tensor field determined by $(\nabla R)(X;\cdot)=\nabla_XR$.  Similarly, $\nabla^2R$ denotes the $(5,1)$-tensor field
\[
        (\nabla^2R)(X,Y;\cdot)=\nabla^2_{XY}R.
\]
The curvature endomorphism also acts on tensor fields; in particular,
\begin{align*}
(R(X,Y)R)(A,B)C={}&R(X,Y)R(A,B)C-R(R(X,Y)A,B)C\\
&-R(A,R(X,Y)B)C-R(A,B)R(X,Y)C.
\end{align*}

\begin{lemma}\label{lemma: 2nd curvature symmetries}
The following identities hold for all $X,Y,V,Z,W\in\mathfrak X(M)$:
\begin{align*}
&\textit{1. (First skew-symmetry identity)}\quad
(\nabla_XR)(Y,Z)W+(\nabla_XR)(Z,Y)W=0;\\
&\textit{2. (Second skew-symmetry identity)}\quad
\langle(\nabla_XR)(Y,Z)W,V\rangle+\langle(\nabla_XR)(Y,Z)V,W\rangle=0;\\
&\textit{3. (Symmetry by pairs identity)}\quad
\langle(\nabla_XR)(Y,Z)W,V\rangle=\langle(\nabla_XR)(V,W)Z,Y\rangle;\\
&\textit{4. (First Bianchi identity)}\quad
(\nabla_XR)(Y,Z)W+(\nabla_XR)(Z,W)Y+(\nabla_XR)(W,Y)Z=0;\\
&\textit{5. (Second Bianchi identity)}\quad
(\nabla_XR)(Y,Z)W+(\nabla_YR)(Z,X)W+(\nabla_ZR)(X,Y)W=0.
\end{align*}
\end{lemma}

\begin{lemma}\label{lemma: 3rd curvature symmetries}
The following identities hold for all $X,Y,A,B,C,D\in\mathfrak X(M)$:
\begin{align*}
&\textit{1. (First skew-symmetry identity)}\quad
(\nabla^2R)(X,Y;A,B)+(\nabla^2R)(X,Y;B,A)=0;\\
&\textit{2. (Second skew-symmetry identity)}\quad
\langle(\nabla^2R)(X,Y;A,B)C,D\rangle
+\langle(\nabla^2R)(X,Y;A,B)D,C\rangle=0;\\
&\textit{3. (Symmetry by pairs identity)}\quad
\langle(\nabla^2R)(X,Y;A,B)C,D\rangle
+\langle(\nabla^2R)(X,Y;D,C)B,A\rangle=0;\\
&\textit{4. (First Bianchi identity)}\quad
(\nabla^2R)(X,Y;A,B)C+(\nabla^2R)(X,Y;B,C)A+(\nabla^2R)(X,Y;C,A)B=0;\\
&\textit{5. (Second Bianchi identity)}\quad
(\nabla^2R)(X,Y;A,B)C+(\nabla^2R)(X,A;B,Y)C+(\nabla^2R)(X,B;Y,A)C=0;\\
&\textit{6. (Nested curvature identity)}\quad
(\nabla^2R)(X,Y;A,B)C-(\nabla^2R)(Y,X;A,B)C=(R(X,Y)R)(A,B)C.
\end{align*}
\end{lemma}

\begin{lemma}\label{lemma: commute cov deriv family}
Let $\gamma_s:t\mapsto\gamma_s(t)=\gamma(s,t)$ be a smooth family of curves on $M$, and let $Z$ be any smooth family of vector fields along $\gamma$. Then
\[
        \nabla_{\frac{\partial\gamma}{\partial s}}\nabla_{\frac{\partial\gamma}{\partial t}}Z
        =\nabla_{\frac{\partial\gamma}{\partial t}}\nabla_{\frac{\partial\gamma}{\partial s}}Z
        +R\left(\frac{\partial\gamma}{\partial s},\frac{\partial\gamma}{\partial t}\right)Z.
\]
\end{lemma}

\begin{proof}[Proof of Lemmas~\ref{lemma: curvature symmetries}--\ref{lemma: commute cov deriv family}]
The identities in Lemma~\ref{lemma: curvature symmetries} are the standard curvature
symmetries of the Levi--Civita connection.  Lemma~\ref{lemma: 2nd curvature symmetries}
follows by covariantly differentiating these identities and using $\nabla g=0$, together
with the second Bianchi identity.  Lemma~\ref{lemma: 3rd curvature symmetries} follows
by applying a second covariant derivative and commuting the two covariant derivatives;
the commutator is precisely the curvature action on the tensor $R$.  Finally,
Lemma~\ref{lemma: commute cov deriv family} is the defining commutation identity for
covariant derivatives along a two-parameter family, obtained from the definition of the
curvature tensor and the torsion-freeness of the Levi--Civita connection.
\qed \end{proof}

\section{Lie bracket computations}\label{app:bracket-computations}

\subsection{Weak fiberwise-linearity}\label{app:weak-fiberwise-linearity}

Let $0_p$ be the null element in the fiber $\tau_k^{-1}(p)$ for all $p\in M$, that is, $0_p=j^k_0(q)$, where $q$ is the constant curve through $p$. We say that a function $f\in C^\infty(T^{(k)}M)$ is \textit{weakly fiberwise linear} if $f(0_p)=0$ for all $p\in M$. Similarly, a vector field $X\in\Gamma(TT^{(k)}M)$ is weakly fiberwise linear if $X_{0_p}=0$ for all $p\in M$.

It follows from the definition of the Levi--Civita-induced connection coefficients that $(K_\alpha)^i_j$ and $\partial (K_\alpha)^i_j/\partial q^{(0)l}$ are weakly fiberwise linear for all $\alpha=1,\ldots,k$ and all $i,j,l$.

\begin{lemma}\label{lemma: fiberlinear}
The vector field
\[
        \left[\frac{\delta}{\delta q^{(0)i}},\frac{\delta}{\delta q^{(0)j}}\right]
\]
is weakly fiberwise linear.
\end{lemma}

\begin{proof}
Let $f\in C^\infty(T^{(k)}M)$. At $0_p$, all connection coefficients vanish, and therefore
\[
        \frac{\delta f}{\delta q^{(0)i}}(0_p)
        =\frac{\partial f}{\partial q^{(0)i}}(0_p).
\]
Thus $\delta f/\delta q^{(0)i}$ is weakly fiberwise linear if and only if $\partial f/\partial q^{(0)i}$ is weakly fiberwise linear. Using also the weak fiberwise linearity of $\partial (K_\alpha)^l_i/\partial q^{(0)j}$, we obtain
\[
        \frac{\delta}{\delta q^{(0)i}}\left(\frac{\delta f}{\delta q^{(0)j}}\right)(0_p)
        =\frac{\partial^2 f}{\partial q^{(0)i}\partial q^{(0)j}}(0_p).
\]
Interchanging $i$ and $j$ and subtracting gives
\[
        \left[\frac{\delta}{\delta q^{(0)i}},\frac{\delta}{\delta q^{(0)j}}\right]_{0_p}(f)=0.
\]
Since this holds for every $f$, the bracket vanishes at $0_p$.
\qed \end{proof}

\subsection{Third-order bracket computations}
\label{app:third-order-brackets}

We now record the computations that lead to Theorem~\ref{thm: liebrac3}.

\begin{lemma}\label{lemma: Brackets h2 3}
 Let  $X, Y \in \mathfrak{X}(M)$. Then
$$\overset{(3)}{K}_3\big( [X^{\overset{(3)}{h}_0}, Y^{\overset{(3)}{h}_2}]_u\big) = -R(X_p,u^{(1)})Y_p - \frac13 R(u^{(1)},Y_p)X_p, \quad \overset{(3)}{K}_3\big([X^{\overset{(3)}{h}_1}, Y^{\overset{(3)}{h}_2}]_u\big)  = \overset{(3)}{K}_3\big([X^{\overset{(3)}{h}_2}, Y^{\overset{(3)}{h}_2}]_u\big)  = 0,$$
 where $u^{(1)}=\overset{(3)}{F}_0(u)$, $u \in T^{(3)}_pM$, $p\in M$.
    \end{lemma}

\begin{proof}
Using Equation \eqref{eq: coordinate basis to adapted basis vector field} and Lemma \ref{lemma: Brackets vk}, it is easily verified that $[X^{\overset{(3)}{h}_2}, Y^{\overset{(3)}{h}_2}] = 0$ for all $X, Y \in \mathfrak{X}(M)$. Moreover,
        \begin{align*}
            \left[ \frac{\delta}{\delta q^{(1)i}}, \frac{\delta}{\delta q^{(2)l}}\right] &= \left[\frac{\partial}{\partial q^{(1)i}} - (K_1)^j_i \frac{\delta}{\delta q^{(2)j}} - (K_2)^j_i \frac{\partial}{\partial q^{(3)j}}, \ \frac{\delta}{\delta q^{(2)l}}\right] \\
            &= \left[\frac{\partial}{\partial q^{(1)i}}, \frac{\delta}{\delta q^{(2)l}}\right] + \frac{\delta (K_2)^j_i}{\delta q^{(2)l}} \frac{\partial}{\partial q^{(3)j}} \\
            &= \left[\frac{\partial}{\partial q^{(1)i}}, \frac{\partial}{\partial q^{(2)l}} - (K_1)^j_l \frac{\partial}{\partial q^{(3)j}}\right] + \frac{\partial (K_1)^j_i}{\partial q^{(1)l}} \frac{\partial}{\partial q^{(3)j}} \\
            &= \left(\frac{\partial (K_1)^j_i}{\partial q^{(1)l}} - \frac{\partial (K_1)^j_l}{\partial q^{(1)i}}\right) \frac{\partial}{\partial q^{(3)j}} \\
            &= 0.
        \end{align*}
Finally,
        \begin{align*}
            \left[ \frac{\delta}{\delta q^{(0)i}}, \frac{\delta}{\delta q^{(2)l}}\right] &= \left[\frac{\partial}{\partial q^{(0)i}} - (K_1)^j_i \frac{\delta}{\delta q^{(1)j}} - (K_2)^j_i \frac{\delta}{\delta q^{(2)j}} - (K_3)^j_i \frac{\partial}{\partial q^{(3)j}}, \ \frac{\delta}{\delta q^{(2)l}}\right] \\
            &= \left[\frac{\partial}{\partial q^{(0)i}}, \ \frac{\delta}{\delta q^{(2)l}} \right] + \frac{\partial (K_2)^j_i}{\partial q^{(2)l}} \frac{\delta}{\delta q^{(2)j}} + \frac{\delta (K_3)^j_i}{\delta q^{(2)l}} \frac{\partial}{\partial q^{(3)j}} \\
            &= \left[\frac{\partial}{\partial q^{(0)i}}, \ \frac{\partial}{\partial q^{(2)l}} - (K_1)^j_l \frac{\partial}{\partial q^{(3)j}}\right] + \frac{\partial (K_2)^j_i}{\partial q^{(2)l}} \frac{\delta}{\delta q^{(2)j}} + \left(\frac{\partial (K_3)^j_i}{\partial q^{(2)l}} - (K_1)^m_l \frac{\partial (K_3)^j_i}{\partial q^{(3)m}}\right) \frac{\partial}{\partial q^{(3)j}} \\
            &= \frac{\partial (K_1)^j_i}{\partial q^{(1)l}} \frac{\delta}{\delta q^{(2)j}} + \left(\frac{\partial (K_3)^j_i}{\partial q^{(2)l}} - (K_1)^m_l \frac{\partial (K_1)^j_i}{\partial q^{(1)m}} - \frac{\partial (K_1)^j_l}{\partial q^{(0)i}}\right) \frac{\partial}{\partial q^{(3)j}}.
        \end{align*}
        Hence,
        \begin{align*}
            \overset{(3)}{K}_3\circ \left[ \frac{\delta}{\delta q^{(0)i}}, \frac{\delta}{\delta q^{(2)l}}\right]&= \left(\frac{\partial (K_3)^j_i}{\partial q^{(2)l}} - (K_1)^m_l \frac{\partial (K_1)^j_i}{\partial q^{(1)m}} - \frac{\partial (K_1)^j_l}{\partial q^{(0)i}}\right) \partial_j.
        \end{align*}
By direct calculation,
    \begin{align*}
        \frac{\partial (K_3)^j_i}{\partial q^{(2)l}} &= q^{(1)a}dq^j\left(\nabla_{\partial_a} \nabla_{\partial_l} \partial_i\right)+ q^{(1)a}\frac13 R^j_{lai},
    \end{align*}
 \begin{align*}
        (K_1)^m_l \frac{\partial (K_1)^j_i}{\partial q^{(1)m}}+\frac{\partial (K_1)^j_l}{\partial q^{(0)i}}&=  dq^j(\nabla_{\partial_i} \nabla_{\partial_a} \partial_l)
    \end{align*}
         Therefore,
        \begin{align*}
            \overset{(3)}{K}_3\left( \left[ \frac{\delta}{\delta q^{(0)i}}, \frac{\delta}{\delta q^{(2)l}}\right]_u\right)
            &= \left[\frac13 q^{(1)a} R_{lai}^j + q^{(1)a} dq^j(\nabla_{\partial_a} \nabla_{\partial_l} \partial_i - \nabla_{\partial_i} \nabla_{\partial_a} \partial_l)\right]\partial_j\Big{\vert}_{p} \\
            &= q^{(1)a} \left[R_{ail}^j + \frac13 R_{lai}^j \right]\partial_j\Big{\vert}_{p},
        \end{align*}
    from which the result follows upon an application of the first Bianchi identity.
\qed \end{proof}

\begin{lemma}\label{lemma: Brackets h1 3}
   Let $X, Y \in \mathfrak{X}(M)$. Then
    \small{\begin{align*}
      \overset{(3)}{K}_3\big(  [X^{\overset{(3)}{h}_0}, Y^{\overset{(3)}{h}_1}]_u\big) &=  -\frac12\left(R(X_p, Y_p)u^{(2)} + \frac13 R(Y_p, u^{(2)})X_p +(\nabla_{u^{(1)}} R)(X_p, Y_p)u^{(1)} + \frac13 (\nabla_{u^{(1)}} R)(Y_p, u^{(1)})X_p\right),\\
     \overset{(3)}{K}_3\big( [X^{\overset{(3)}{h}_1}, Y^{\overset{(3)}{h}_1}]_u\big)  &= -\frac12 R(X_p,Y_p)u^{(1)},
    \end{align*}}
where $u^{(1)}=\overset{(3)}{F}_0(u)$, $u^{(2)}= \overset{(3)}{F}_1(u)$, $u \in T^{(3)}_pM$, $p\in M$.
\end{lemma}

\begin{proof}
Applying the same method as in the previous lemma, we get

    \begin{align*}
        \left[ \frac{\delta}{\delta q^{(1)i}}, \frac{\delta}{\delta q^{(1)l}}\right] &= \left[\frac{\partial}{\partial q^{(1)i}} - (K_1)^j_i \frac{\delta}{\delta q^{(2)j}} - (K_2)^j_i \frac{\partial}{\partial q^{(3)j}}, \ \frac{\delta}{\delta q^{(1)l}}\right] \\
        &= \left[\frac{\partial}{\partial q^{(1)i}}, \frac{\delta}{\delta q^{(1)l}}\right] + \frac{\delta (K_1)^j_i}{\delta q^{(1)l}}\frac{\delta}{\delta q^{(2)j}} + \frac{\delta (K_2)^j_i}{\delta q^{(1)l}}\frac{\partial}{\partial q^{(3)j}} \\
        &= \left[\frac{\partial}{\partial q^{(1)i}}, \frac{\partial}{\partial q^{(1)l}} - (K_1)_l^j \frac{\delta}{\delta q^{(2)j}} - (K_2)_l^j \frac{\partial}{\partial q^{(3)j}}\right] + \frac{\delta (K_1)^j_i}{\delta q^{(1)l}}\frac{\delta}{\delta q^{(2)j}} + \frac{\delta (K_2)^j_i}{\delta q^{(1)l}}\frac{\partial}{\partial q^{(3)j}} \\
        &= -(K_1)^j_l\left[\frac{\partial}{\partial q^{(1)i}}, \frac{\delta}{\delta q^{(2)j}} \right] + \left(\frac{\partial (K_1)^j_i}{\partial q^{(1)l}} -\frac{\partial (K_1)^j_l}{\partial q^{(1)i}}\right)\frac{\delta}{\delta q^{(2)j}} + \left(\frac{\delta (K_2)^j_i}{\delta q^{(1)l}} - \frac{\partial (K_2)^j_l}{\partial q^{(1)i}}\right)\frac{\partial}{\partial q^{(3)j}} \\
        &= \left(\frac{\partial (K_1)^j_i}{\partial q^{(1)l}} -\frac{\partial (K_1)^j_l}{\partial q^{(1)i}}\right)\frac{\delta}{\delta q^{(2)j}} + \left(\frac{\delta (K_2)^j_i}{\delta q^{(1)l}} + (K_1)^m_l \frac{\partial(K_1)^j_m}{\partial q^{(1)i}} - \frac{\partial (K_2)^j_l}{\partial q^{(1)i}}\right)\frac{\partial}{\partial q^{(3)j}}.
    \end{align*}
    Hence,
    \begin{align*}
        \overset{(3)}{K}_3\circ\left[ \frac{\delta}{\delta q^{(1)i}}, \frac{\delta}{\delta q^{(1)l}}\right]  &= \left(\frac{\partial (K_2)^j_i}{\partial q^{(1)l}} -  \frac{\partial (K_2)^j_l}{\partial q^{(1)i}}\right)\partial_j. \\
    \end{align*}
 Thus, we obtain
    \begin{align*}
        \overset{(3)}{K}_3\left( \left[ \frac{\delta}{\delta q^{(1)i}}, \frac{\delta}{\delta q^{(1)l}}\right]_u \right) &= \frac12 q^{(1)a} (R_{lai}^j + R_{ail}^j)\partial_j\Big{\vert}_{p} \\
        &= \frac12 q^{(1)a} R_{lia}^j\partial_j\Big{\vert}_{p},
    \end{align*}
 after applying  the first Bianchi identity.

Next,
\begin{align*}
    \left[\frac{\delta}{\delta q^{(0)i}}, \frac{\delta}{\delta q^{(1)l}} \right] &= \left[\frac{\partial}{\partial q^{(0)i}} - (K_1)^j_i \frac{\delta}{\delta q^{(1)j}} - (K_2)^j_i \frac{\delta}{\delta q^{(2)j}} - (K_3)^j_i \frac{\partial}{\partial q^{(3)j}}, \  \frac{\delta}{\delta q^{(1)l}} \right] \\
    &= \left[\frac{\partial}{\partial q^{(0)i}}, \frac{\delta}{\delta q^{(1)l}}  \right] + \frac{\delta (K_1)^j_i}{\delta q^{(1)l}} \frac{\delta}{\delta q^{(1)j}} + \frac{\delta (K_2)^j_i}{\delta q^{(1)l}} \frac{\delta}{\delta q^{(2)j}} + \frac{\delta (K_3)^j_i}{\delta q^{(1)l}} \frac{\partial}{\partial q^{(3)j}} - (K_1)^j_i \left[\frac{\delta}{\delta q^{(1)j}}, \ \frac{\delta}{\delta q^{(1)l}} \right] \\
    &= \left[\frac{\partial}{\partial q^{(0)i}}, \frac{\delta}{\delta q^{(1)l}}  \right] + \frac{\delta (K_1)^j_i}{\delta q^{(1)l}} \frac{\delta}{\delta q^{(1)j}} + \frac{\delta (K_2)^j_i}{\delta q^{(1)l}} \frac{\delta}{\delta q^{(2)j}} + \left(\frac{\delta (K_3)^j_i}{\delta q^{(1)l}} + \frac12 (K_1)^m_i q^{(1)a} R_{mla}^j\right) \frac{\partial}{\partial q^{(3)j}}.
\end{align*}
Moreover,
\begin{align*}
    \left[\frac{\partial}{\partial q^{(0)i}},\ \frac{\delta}{\delta q^{(1)l}}  \right] &= \left[\frac{\partial}{\partial q^{(0)i}}, \frac{\partial}{\partial q^{(1)l}} - (K_1)^j_l \frac{\delta}{\delta q^{(2)j}} - (K_2)^j_l \frac{\partial}{\partial q^{(3)j}} \right] \\
    &=  -\frac{\partial(K_1)^j_l}{\partial q^{(0)i}} \frac{\delta}{\delta q^{(2)j}} -\frac{\partial(K_2)^j_l}{\partial q^{(0)i}} \frac{\partial}{\partial q^{(3)j}} - (K_1)^j_l \left[\frac{\partial}{\partial q^{(0)i}}, \frac{\delta}{\delta q^{(2)j}}\right] \\
    &= -\frac{\partial (K_1)^j_l}{\partial q^{(0)i}} \frac{\delta}{\delta q^{(2)j}} + \left((K_1)^m_l \frac{\partial (K_1)^j_m}{\partial q^{(0)i}} - \frac{\partial (K_2)^j_l}{\partial q^{(0)i}} \right)\frac{\partial}{\partial q^{(3)j}}.
\end{align*}
Hence, we obtain
\begin{align*}
    \left[\frac{\delta}{\delta q^{(0)i}}, \frac{\delta}{\delta q^{(1)l}} \right] = &\frac{\partial (K_1)^j_i}{\partial q^{(1)l}} \frac{\delta}{\delta q^{(1)j}} + \left(\frac{\delta (K_2)^j_i}{\delta q^{(1)l}} -  \frac{\partial (K_1)^j_l}{\partial q^{(0)i}}\right)\frac{\delta}{\delta q^{(2)j}} \\
    &\qquad + \left((K_1)^m_l \frac{\partial (K_1)^j_m}{\partial q^{(0)i}} - \frac{\partial (K_2)^j_l}{\partial q^{(0)i}} + \frac{\delta (K_3)^j_i}{\delta q^{(1)l}} + \frac12 (K_1)^m_i q^{(1)a} R_{mla}^j\right) \frac{\partial}{\partial q^{(3)j}},
\end{align*}
from which it follows that
\begin{align*}
    \overset{(3)}{K}_3\circ\left[\frac{\delta}{\delta q^{(0)i}}, \frac{\delta}{\delta q^{(1)l}} \right] &= \left((K_1)^m_l \frac{\partial (K_1)^j_m}{\partial q^{(0)i}} - \frac{\partial (K_2)^j_l}{\partial q^{(0)i}} + \frac{\delta (K_3)^j_i}{\delta q^{(1)l}} + \frac12 (K_1)^m_i q^{(1)a} R_{mla}^j\right) \partial_j \\
    &= \Big{(} (K_1)^m_l \frac{\partial (K_1)^j_m}{\partial q^{(0)i}} - \frac{\partial (K_2)^j_l}{\partial q^{(0)i}} + \frac12 (K_1)^m_i q^{(1)a} R_{mla}^j + \frac{\partial (K_3)^j_i}{\partial q^{(1)l}} \\
    &\qquad - (K_1)^m_l \frac{\partial (K_3)^j_i}{\partial q^{(2)m}} + \left( (K_1)_r^m (K_1)^r_l - (K_2)^m_l \right) \frac{\partial (K_1)^j_i}{\partial q^{(1)m}} \Big{)}\partial_j.
\end{align*}
By direct calculation, we get
{\small \begin{align*}
    (K_1)^m_l \frac{\partial (K_1)^j_m}{\partial q^{(0)i}} &= q^{(1)a}q^{(1)b} \left(dq^m(\nabla_{\partial_a} \partial_l) dq^j(\nabla_{\partial_i} \nabla_{\partial_b} \partial_m) - dq^m(\nabla_{\partial_a}\partial_l) dq^r(\nabla_{\p_b}\p_m) dq^j(\nabla_{\p_i} \p_r)\right), \\
    \frac{\p(K_2)^j_l}{\p q^{(0)i}} &= q^{(2)a}\left[dq^j(\nabla_{\p_i} \nabla_{\p_a} \p_l) - dq^m(\nabla_{\p_a} \p_l) dq^j(\nabla_{\p_i} \p_m) \right] \\
    & \qquad\qquad +\frac12 q^{(1)a} q^{(1)b} \left[dq^j(\nabla_{\p_i} \nabla_{\p_a} \nabla_{\p_b} \p_l) - dq^m(\nabla_{\p_a} \nabla_{\p_b} \p_l)dq^j(\nabla_{\p_i} \p_m) \right], \\
    \frac{\p (K_3)^j_i}{\p q^{(1)l}} &= q^{(2)a} \left[ dq^j(\nabla_{\p_l} \nabla_{\p_a} \p_i) + \frac13 R^j_{ali}\right] +\frac16 q^{(1)a} q^{(1)b} dq^j(\nabla_{\p_a} \nabla_{\p_l} \nabla_{\p_b} \p_i + \nabla_{\p_b}\nabla_{\p_a} \nabla_{\p_l}\p_i + \nabla_{\p_l}\nabla_{\p_a}\nabla_{\p_b}\p_i) \\
    &= q^{(2)a} \left[ dq^j(\nabla_{\p_l} \nabla_{\p_a} \p_i) + \frac13 R^j_{ali}\right] + \frac12 q^{(1)a} q^{(1)b}dq^j (\nabla_{\p_a} \nabla_{\p_l} \nabla_{\p_b} \p_i) \\
    & \qquad  + \frac16 q^{(1)a} q^{(1)b} dq^j(\nabla_{\p_a}(R(\p_b, \p_l)\p_i) + R(\p_l, \p_a)\nabla_{\p_b}\p_i), \\
    (K_1)^m_l \frac{\p (K_3)^j_i}{\p q^{(2)m}} &= q^{(1)a}q^{(1)b} \left( dq^m(\nabla_{\p_a} \p_l) dq^j(\nabla_{\p_b} \nabla_{\p_m} \p_i) + \frac13 dq^m(\nabla_{\p_a}\p_l) R^j_{mbi}\right), \\
    (K_1)^m_r (K_1)^r_l - (K_2)^m_l &= -q^{(2)a} dq^m(\nabla_{\p_a}\p_l) + q^{(1)a}q^{(1)b} \left(dq^m(\nabla_{\p_a}\p_r) dq^r(\nabla_{\p_b} \p_l) - \frac12 dq^m(\nabla_{\p_a}\nabla_{\p_b}\p_l)\right).
\end{align*}}
Collecting all terms with leading coefficients $q^{(2)a}$, we obtain
$$q^{(2)a} \left[dq^j(\nabla_{\p_l}\nabla_{\p_a}\p_i) - dq^j(\nabla_{\p_i} \nabla_{\p_a} \p_l) + \frac13 R_{ali}^j\right]\p_j = q^{(2)a} \left[R_{lia}^j + \frac13 R_{ali}^j \right]\partial_j.$$
Moreover, we note that $q^{(2)m} = \frac12\left(F_1^m - q^{(1)a} q^{(1)b} dq^m(\nabla_{\p_a}\p_b)\right)$, so that
\begin{align*}
    q^{(2)a} \left[R_{lia}^j + \frac13 R_{ali}^j \right]\partial_j = \frac12F_1^m \left[ R_{lim}^j + \frac13 R^j_{mli}\right]\partial_j - \frac12 q^{(1)a}q^{(1)b} dq^m(\nabla_{\p_a}\p_b) \left[ R_{lim}^j + \frac13 R^j_{mli}\right]\partial_j.
\end{align*}
Collecting all the terms with leading coefficients $q^{(1)a}q^{(1)b}$ now, we obtain
\begin{align*}
    q^{(1)a}q^{(1)b} &\Big{[}dq^m(\nabla_{\p_a} \p_l) dq^j(\nabla_{\p_i} \nabla_{\p_b}\p_m - \nabla_{\p_b}\nabla_{\p_m}\p_i) - \frac12 dq^j(\nabla_{\p_i} \nabla_{\p_a} \nabla_{\p_b} \p_l -  \nabla_{\p_a} \nabla_{\p_l} \nabla_{\p_b} \p_i) \\
    &\qquad + \frac12 dq^m(\nabla_{\p_b}\p_i) R^j_{mla} + \frac16 dq^m(\nabla_{\p_b}\p_i)R_{lam}^j - \frac13 dq^m(\nabla_{\p_b}\p_l)R^j_{mai} + \frac16 dq^j(\nabla_{\p_a} (R(\p_b, \p_l)\p_i)) \\
    &\quad - \frac12 dq^m(\nabla_{\p_a}\p_b) R_{lim}^j - \frac16 dq^m(\nabla_{\p_a}\p_b) R^j_{mli}\big{]}\partial_j \\
    = q^{(1)a} q^{(1)b} &\Big{[}dq^m(\nabla_{\p_b}\p_l) R^j_{iam} + \frac12 R_{a; lib}^j + \frac12 dq^m(\nabla_{\p_b}\p_l)R_{mia}^j + \frac12 dq^m(\nabla_{\p_b}\p_i)R_{lma}^j + \frac12 dq^m(\nabla_{\p_a}\p_b) R_{lim}^j \\
    &\quad - dq^m(\nabla_{\p_b}\p_l)R^j_{iam} + \frac12 dq^m(\nabla_{\p_b}\p_i) R^j_{mla} + \frac16 dq^m(\nabla_{\p_b}\p_i)R^j_{lam} - \frac13 dq^m(\nabla_{\p_b}\p_l) R^j_{mai} \\
    &\quad - \frac12 dq^m(\nabla_{\p_a}\p_b) R_{lim}^j - \frac16 dq^m(\nabla_{\p_a}\p_b)R^j_{mli} + \frac16 R_{a; bli} + \frac16 dq^m(\nabla_{\p_a}\p_b)R_{mli}^j \\
    & \quad + \frac16 dq^m(\nabla_{\p_b}\p_l)R^j_{ami} + \frac16 dq^m(\nabla_{\p_b}\p_i)R^j_{alm}\Big{]}\partial_j.
\end{align*}
Grouping terms with common leading coefficients, we see that all terms with leading coefficients $dq^m(\nabla_{\p_b}\p_l)$, $dq^m(\nabla_{\p_a}\p_b)$, and $dq^m(\nabla_{\p_b}\p_i)$ cancel due to the first Bianchi identity. Hence,
{\small
\begin{align*}
    \overset{(3)}{K}_3 \circ \left[\frac{\delta}{\delta q^{(0)i}}, \frac{\delta}{\delta q^{(1)l}} \right] &= \left(\frac12F_1^m \left[ R_{lim}^j + \frac13 R^j_{mli}\right] + \frac12 F_0^aF_0^b \Big{[}R^j_{a; lib} + \frac13 R^j_{a; bli} \Big{]} \right)\partial_j,
\end{align*}}
from which the result follows.
\qed \end{proof}

\begin{lemma}\label{lemma: Brackets h0h0 3}
Let $X, Y \in \mathfrak{X}(M)$. Then
    \small{\begin{align*}
        \overset{(3)}{K}_3([X^{\overset{(3)}{h}_0}, Y^{\overset{(3)}{h}_0}]_u) &= -\frac1{3!} R(X_p,Y_p)u^{(3)} - \frac13 (\nabla_{u^{(1)}} R)(X_p,Y_p)u^{(2)}+ \frac13 (\nabla_{u^{(2)}}R)(X_p,Y_p)u^{(1)} - \frac12 (\nabla^2_{u^{(1)}} R)(X_p,Y_p)u^{(1)},
    \end{align*}}\normalsize
    where $u^{(1)}=\overset{(3)}{F}_0(u)$, $u^{(2)}=\overset{(3)}{F}_1(u)$, $u^{(3)}=\overset{(3)}{F}_2(u)$, $u \in T^{(3)}_pM$, $p\in M$.
\end{lemma}

\begin{proof}
First, observe that
\begin{equation}\label{eq: proofh0h0 3}
\left[\frac{\delta}{\delta q^{(0)i}},  \frac{\delta}{\delta q^{(0)j}}\right]=F_0^mR_{ijm}^l\frac{\delta}{\delta q^{(1)l}}-\left(\frac12 F_1^mR_{ijm}^l+ F_0^aF_0^bR_{a;ijb}^l\right)\frac{\delta}{\delta q^{(2)l}}+S_{(0)i(0)j}^{(3)l}\frac{\partial}{\partial q^{(3)l}},
\end{equation}
where  $S_{(0)i(0)j}^{(3)l}\in C^\infty(T^{(3)}M)$.
Then
\begin{align*}
    \left[ \frac{\partial}{\partial q^{(3)m}}, \left[\frac{\delta}{\delta q^{(0)i}},  \frac{\delta}{\delta q^{(0)j}}\right]\right] &= \frac{\p S_{(0)i(0)j}^{(3)l}}{\p q^{(3)m}} \frac{\p}{\p q^{(3)l}},
\end{align*}
On the other hand, from the Jacobi identity,

\small{
\begin{align*}
    \left[ \frac{\partial}{\partial q^{(3)m}}, \left[\frac{\delta}{\delta q^{(0)i}},  \frac{\delta}{\delta q^{(0)j}}\right]\right] &= \left[\frac{\delta}{\delta q^{(0)j}}, \left[\frac{\delta}{\delta q^{(0)i}}, \frac{\partial}{\partial q^{(3)m}}\right] \right] -\left[\frac{\delta}{\delta q^{(0)i}}, \left[\frac{\delta}{\delta q^{(0)j}}, \frac{\partial}{\partial q^{(3)m}}\right] \right] \\
    &= \left[\frac{\delta}{\delta q^{(0)j}}, \left(\nabla_{\partial_i}\partial_m \right)^{\overset{(3)}{v}_3} \right] - \left[\frac{\delta}{\delta q^{(0)i}}, \left(\nabla_{\partial_j}\partial_m \right)^{\overset{(3)}{v}_3} \right] \\
    &= \left(\nabla_{\p_j} \nabla_{\p_i} \p_m - \nabla_{\p_i} \nabla_{\p_j} \p_m \right)^{\overset{(3)}{v}_3} \\
    &= -R_{ijm}^l \frac{\partial}{\partial q^{(3)l}},
\end{align*}}
\normalsize
so that $ \displaystyle \frac{\p S_{(0)i(0)j}^{(3)l}}{\p q^{(3)m}} = -R_{ijm}^l$.
It follows that
\begin{equation}\label{eq: S0,0(3) 1}
S_{(0)i(0)j}^{(3)l}= -\frac 1 {3!}R^l_{ijm} F_2^m + \Phi^l_{ij},
\end{equation}
for some functions $\Phi^l_{ij}\in C^\infty(T^{(3)}M)$ satisfying $\frac{\p \Phi^l_{ij}}{\p q^{(3)\lambda}}=0$, for all $\lambda=1,\dots,n$.

Similarly,
\small{
\begin{align*}
    \left[ \frac{\delta}{\delta q^{(2)m}}, \left[\frac{\delta}{\delta q^{(0)i}},  \frac{\delta}{\delta q^{(0)j}}\right] \right]_u &= \left[\frac{\delta}{\delta q^{(0)j}}, \left[\frac{\delta}{\delta q^{(0)i}}, \frac{\delta}{\delta q^{(2)m}}\right] \right]_u -\left[\frac{\delta}{\delta q^{(0)i}}, \left[\frac{\delta}{\delta q^{(0)j}}, \frac{\delta}{\delta q^{(2)m}}\right] \right]_u\\
    &= \left[\frac{\delta}{\delta q^{(0)j}}, \left(\nabla_{\partial_i}\partial_m \right)^{\overset{(3)}{h}_2} -\left(R(\partial_i, u^{(1)})\partial_m + \frac13 R(u^{(1)}, \partial_m)\partial_i\right)^{\overset{(3)}{v}_3}\right]_u\\
    & \quad - \left[\frac{\delta}{\delta q^{(0)i}}, \left(\nabla_{\partial_j}\partial_m \right)^{\overset{(3)}{h}_2} - \left(R(\partial_j, u^{(1)})\partial_m + \frac13 R(u^{(1)}, \partial_m)\p_j\right)^{\overset{(3)}{v}_3} \right]_u \\
    &= \left(R(\p_j, \p_i)\p_m\right)_u^{\overset{(3)}{h}_2} - \left(R(\p_j, u^{(1)})\nabla_{\p_i} \p_m \right.\\
    & \quad \left.+ \frac13 R(u^{(1)}, \nabla_{\p_i}\p_m)\p_j -  R(\p_i, u^{(1)})\nabla_{\p_j} \p_m - \frac13 R(u^{(1)}, \nabla_{\p_j}\p_m)\p_i\right)_u^{\overset{(3)}{v}_3} \\
    &= \left[\frac{\delta}{\delta q^{(0)j}}, \left(\nabla_{\partial_i}\partial_m \right)^{\overset{(3)}{h}_2} \right]_u- \left[\frac{\delta}{\delta q^{(0)i}}, \left(\nabla_{\partial_j}\partial_m \right)^{\overset{(3)}{h}_2}\right]_u\\
    &\quad -\left[\frac{\delta}{\delta q^{(0)j}},(R(\partial_i, u^{(1)})\partial_m + \frac13 R(u^{(1)}, \partial_m)\partial_i)^{\overset{(3)}{v}_3}\right]_u +\left[\frac{\delta}{\delta q^{(0)i}}, (R(\partial_j, u^{(1)})\partial_m + \frac13 R(u^{(1)}, \partial_m)\partial_j)^{\overset{(3)}{v}_3} \right]_u
    \end{align*}}\normalsize
Then, we calculate the first two Lie brackets.
\small{
\begin{align*}
 \left[\frac{\delta}{\delta q^{(0)j}}, \left(\nabla_{\partial_i}\partial_m \right)^{\overset{(3)}{h}_2} \right]_u- \left[\frac{\delta}{\delta q^{(0)i}}, \left(\nabla_{\partial_j}\partial_m \right)^{\overset{(3)}{h}_2}\right]_u=
    &\left(R(\p_j, \p_i)\p_m\right)_u^{\overset{(3)}{h}_2} - \left(R(\p_j, u^{(1)})\nabla_{\p_i} \p_m + \frac13 R(u^{(1)}, \nabla_{\p_i}\p_m)\p_j \right. \\
    &\left. -  R(\p_i, u^{(1)})\nabla_{\p_j} \p_m - \frac13 R(u^{(1)}, \nabla_{\p_j}\p_m)\p_i\right)_u^{\overset{(3)}{v}_3}
\end{align*}}\normalsize
Note that we cannot directly use Lemma \ref{lemma: Brackets vk} to calculate the two final Lie brackets, as also happened in the proof of Lemma \ref{lemma: Brackets h0h0}. Rather, we must use the following similar identity.
\begin{equation}\label{eq: lie brac Xh0 Rv3}
    \left[X^{\overset{(3)}{h}_0}, (R(Y, u^{(1)})Z)^{\overset{(3)}{v}_3}\right]_u
    = \left((\nabla_X R)(Y,u^{(1)})Z + R(\nabla_X Y, u^{(1)})Z + R(Y, u^{(1)})\nabla_X Z \right)_u^{\overset{(3)}{v}_3}, X, Y, Z \in  \mathfrak{X}(M)
\end{equation}
By making use of the symmetries of the curvature tensor, we obtain similar results when $u^{(1)}$ is in any of the other arguments of $R$. Then, we apply these results to calculate the two final Lie brackets. Combining the four Lie brackets, we obtain
{\small
\begin{align*}
    \overset{(3)}{v}_3\left(\left[ \frac{\delta}{\delta q^{(2)m}}, \left[\frac{\delta}{\delta q^{(0)i}},  \frac{\delta}{\delta q^{(0)j}}\right]\right]_u\right) =
   \left( (\nabla_{\p_i}R)(\p_j,  u^{(1)})\p_m -(\nabla_{\p_j} R)(\p_i,  u^{(1)})\p_m +\frac13 (\nabla_{\p_i}R)( u^{(1)}, \p_m)\p_j- \frac13 (\nabla_{\p_j}R)( u^{(1)}, \p_m)\p_i\right)_u^{\overset{(3)}{v}_3}\!.
\end{align*}}\normalsize

On the other hand, from Equation \eqref{eq: proofh0h0 3}, we also have
\begin{align*}
     \overset{(3)}{K}_3\circ\left[ \frac{\delta}{\delta q^{(2)m}}, \left[\frac{\delta}{\delta q^{(0)i}},  \frac{\delta}{\delta q^{(0)j}}\right]\right] &= \frac{\delta S_{ij}^{(3)l}}{\delta q^{(2)m}} \frac{\partial}{\partial q^{(3)l}} \\
     &= \left(-\frac 1 {3!}R_{ija}^l \frac{\delta F_2^a}{\delta q^{(2)m}}+ \frac{\p \Phi^{(3)l}_{ij}}{\p q^{(2)m}} \right) \frac{\partial}{\partial q^{(3)l}} \\
     &=\frac{\p \Phi^{(3)l}_{ij}}{\p q^{(2)m}} \frac{\partial}{\partial q^{(3)l}}.
\end{align*}
\normalsize
Then, it follows that
\begin{equation}\label{eq: phi }
    \Phi^l_{ij} = \frac 1 2 F_0^{a} F_1^m\left(R_{i; jam}^l - R_{j; iam}^l + \frac13 R_{i; amj}^l - \frac13 R^l_{j; ami} \right) + \Psi^l_{ij},
\end{equation}
for some functions  $\Psi^l_{ij} \in C^\infty(T^{(3)}M)$ satisfying  $\frac{\p \Psi^l_{ij}}{\p q^{(2)\lambda}}= \frac{\p \Psi^l_{ij}}{\p q^{(3)\lambda}} = 0$, for all $\lambda=1,\dots,n$. For the sake of simplicity, in what follows, we will use the notation $Q^l_{ijam}=R_{i; jam}^l - R_{j; iam}^l + \frac13 R_{i; amj}^l - \frac13 R^l_{j; ami}$.
Hence,
\begin{equation}\label{eq: S0,0(3) 2}
S_{(0)i(0)j}^{(3)l}= -\frac 1 {3!}R^l_{ijm} F_2^m + \frac 1 2 F_0^{a} F_1^mQ^l_{ijam} + \Psi^l_{ij}.
\end{equation}

Finally,
{\small \begin{align*}
     \overset{(3)}{v}_3 \circ \left[ \frac{\delta}{\delta q^{(1)m}}, \left[\frac{\delta}{\delta q^{(0)i}},  \frac{\delta}{\delta q^{(0)j}}\right]\right] &= \frac{\delta S^{(3)l}_{ij}}{\delta q^{(1)m}} \frac{\p}{\p q^{(3)l}} -F_0^b R_{ijb}^k\left[ \frac{\delta}{\delta q^{(1)m}}, \frac{\delta}{\delta q^{(1)k}}\right] \\
    &=   \left[R_{ijk}^l \left( (K_2)^k_m - \frac 1 {3!}\frac{\p F_2^k}{\p q^{(1)m}} \right)+\frac 1 2 F_1^k Q_{ijmk}^l   + \frac{\p \Psi^l_{ij}}{\p q^{(1)m}} + \frac12 F_0^a  F_0^bR_{mk a}^l R_{ijb}^k  \right]\frac{\partial}{\partial q^{(3)l}}.
\end{align*}}
Moreover, since
\begin{align*}
    (K_2)^k_m - \frac 1 {3!}\frac{\p F_2^k}{\p q^{(1)m}} &=  \frac1{3!} F_0^aF_0^bR_{amb}^k,
\end{align*}
we conclude that
{\small \begin{align*}
     \overset{(3)}{v}_3\circ \left[ \frac{\delta}{\delta q^{(1)m}}, \left[\frac{\delta}{\delta q^{(0)i}},  \frac{\delta}{\delta q^{(0)j}}\right]\right] &=   \left(\frac 1 2 F_1^k Q_{ijmk}^l   + \frac1{6} F_0^aF_0^bR_{amb}^k+\frac12 F_0^a  F_0^bR_{mk a}^l R_{ijb}^k+ \frac{\p \Psi^l_{ij}}{\p q^{(1)m}}  \right)\frac{\partial}{\partial q^{(3)l}}.
\end{align*}}
Hence, using the notation
$$C = \overset{(3)}{K}_3 \circ \left[ \frac{\delta}{\delta q^{(1)m}}, \left[\frac{\delta}{\delta q^{(0)i}},  \frac{\delta}{\delta q^{(0)j}}\right]\right]$$
with  $C=C_{mij}^l\p_l$, we have
$$\frac{\p \Psi^l_{ij}}{\p q^{(1)m}} = -\frac 1 2 F_1^k Q_{ijmk}^l   - \frac1{6} F_0^aF_0^bR_{amb}^k-\frac12 F_0^a  F_0^bR_{mk a}^l R_{ijb}^k+ C_{mij}^l,$$
from which it follows that
\begin{equation}\label{eq: S0,0(3) 3}
\Psi^l_{ij}= -\frac 1 2 F_0^mF_1^k Q_{ijmk}^l   - \frac1{6} F_0^MF_0^aF_0^bR_{amb}^k-\frac12 F_0^MF_0^a  F_0^bR_{mk a}^l R_{ijb}^k+ F_0^mC_{mij}^l+h_{ij}^l,
\end{equation}
for some functions  $h_{ij}^l \in C^\infty(T^{(3)}M)$ satisfying  $\frac{\p h_{ij}^l}{\p q^{(1)m}}=\frac{\p h_{ij}^l}{\p q^{(2)m}}= \frac{\p h_{ij}^l}{\p q^{(3)m} }= 0$, for all $m=1,\dots,n$.

 In summary, from Equations \eqref{eq: S0,0(3) 2}-\eqref{eq: S0,0(3) 3},
 we have
$$
S_{(0)i(0)j}^{(3)l}= -\frac 1 {3!}R^l_{ijm} F_2^m + \frac 1 2 F_0^{a} F_1^mQ^l_{ijam}  -\frac 1 2 F_0^mF_1^k Q_{ijmk}^l   - \frac1{6} F_0^MF_0^aF_0^bR_{amb}^k-\frac12 F_0^MF_0^a  F_0^bR_{mk a}^l R_{ijb}^k+ F_0^mC_{mij}^l+h_{ij}^l,
$$
where $C_{mij}^l$  are the coefficient functions of $C$.

Finally, through Lemma \ref{lemma: fiberlinear}, we know that $S_{(0)i(0)j}^{(3)l}$ is weakly
fiberwise linear, which implies that
\begin{equation}\label{eq: S 3 final}
S_{(0)i(0)j}^{(3)l}= -\frac 1 {3!}R^l_{ijm} F_2^m + \frac 1 2 F_0^{a} F_1^mQ^l_{ijam}  -\frac 1 2 F_0^mF_1^k Q_{ijmk}^l   - \frac1{6} F_0^MF_0^aF_0^bR_{amb}^k-\frac12 F_0^MF_0^a  F_0^bR_{mk a}^l R_{ijb}^k+ F_0^mC_{mij}^l.
\end{equation}
On the other hand, to obtain the functions $C_{mij}^l$, we apply once again the Jacobi identity. Thus, we get
\begin{equation}\label{eq: Jacobi 3}
\left[ \frac{\delta}{\delta q^{(1)m}}, \left[\frac{\delta}{\delta q^{(0)i}},  \frac{\delta}{\delta q^{(0)j}}\right]\right] = \left[\frac{\delta}{\delta q^{(0)j}}, \left[\frac{\delta}{\delta q^{(0)i}}, \frac{\delta}{\delta q^{(1)m}} \right] \right] - \left[\frac{\delta}{\delta q^{(0)i}}, \left[\frac{\delta}{\delta q^{(0)j}}, \frac{\delta}{\delta q^{(1)m}} \right] \right].
\end{equation}
Now,  we calculate
{\small
\begin{align*}
    \left[\frac{\delta}{\delta q^{(0)j}}, \left[\frac{\delta}{\delta q^{(0)i}}, \frac{\delta}{\delta q^{(1)m}} \right] \right] &= \left[\frac{\delta}{\delta q^{(0)j}}, \left(\nabla_{\p_i} \p_m \right)^{\overset{(3)}{h}_1} \right] + \frac12 \left[\frac{\delta}{\delta q^{(0)j}}, \left(R(u^{(1)}, \p_i)\p_m - R(\p_i, \p_m)u^{(1)} \right)^{\overset{(3)}{h}_2} \right] \\
    &\quad - \frac12 \left[ \frac{\delta}{\delta q^{(0)j}}, \left(R(\p_i, \p_m)^{(2)} + \frac13 R(\p_m, u^{(2)})\p_i + (\nabla_{u^{(1)}} R)(\p_i, \p_m)u^{(1)} + \frac13 (\nabla_{u^{(1)}} R)(\p_m, u^{(1)})\p_i \right)^{\overset{(3)}{v}_3}\right].
    \end{align*}}
       The first Lie bracket simplifies as follows.
{\small\begin{align*}
\left[\frac{\delta}{\delta q^{(0)j}}, \left(\nabla_{\p_i} \p_m \right)^{\overset{(3)}{h}_1} \right] &= (\nabla_{\p_j} \nabla_{\p_i} \p_m)^{\overset{(3)}{h}_1} + \frac12\left(R(u^{(1)}, \p_j)\nabla_{\p_i}\p_m - R(\p_j, \nabla_{\p_i}\p_m)u^{(1)} \right)^{\overset{(3)}{h}_2} \\
- \frac12 \Big{(}&R(\p_j, \nabla_{\p_i}\p_m)u^{(2)} + \frac13 R(\nabla_{\p_i}\p_m, u^{(2)})\p_j + (\nabla_{u^{(1)}} R)(\p_j, \nabla_{\p_i}\p_m)u^{(1)} + \frac13 (\nabla_{u^{(1)}} R)(\nabla_{\p_i}\p_m, u^{(1)})\p_j \Big{)}^{\overset{(3)}{v}_3}
 \end{align*}}
  With regard to the second Lie bracket, we have
{\small\begin{align*}
&\Big{[}\frac{\delta}{\delta q^{(0)j}}, \Big{(}R(u^{(1)}, \p_i)\p_m - R(\p_i, \p_m)u^{(1)} \Big{)}^{\overset{(3)}{h}_2} \Big{]} =\\
&\qquad \Big{[} (\nabla_{\p_j} R)(u^{(1)}, \p_i)\p_m + R(u^{(1)}, \nabla_{\p_j}\p_i)\p_m + R(u^{(1)}, \p_i)\nabla_{\p_j}\p_m - (\nabla_{\p_j}R)(\p_i, \p_m)u^{(1)} - R(\nabla_{\p_j} \p_i, \p_m)u^{(1)}- R(\p_i, \nabla_{\p_j}\p_m)u^{(1)}\Big{]}^{\overset{(3)}{h}_2} \\
&\qquad -\Big{[} R(\p_j, u^{(1)})R(u^{(1)}, \p_i)\p_m - R(\p_j, u^{(1)})R(\p_i, \p_m)u^{(1)} + \frac13 R(u^{(1)}, R(u^{(1)}, \p_i)\p_m)\p_j - \frac13 R(u^{(1)}, R(\p_i, \p_m)u^{(1)})\p_j \Big{]}^{\overset{(3)}{v}_3}.
\end{align*}}

 Regarding the third Lie bracket, we first need to see how terms involving $u^{(3)}$ transform within the Lie bracket. For all $X, Y,Z \in \mathfrak{X}(M)$, we have
\begin{align*}
    \left[X^{\overset{(3)}{h}_0}, (R(Y, Z)u^{(2)})^{\overset{(3)}{v}_3} \right] &= \left[X^{\overset{(3)}{h}_0}, F_1^\lambda (R(Y,Z)\p_\lambda)^{\overset{(3)}{v}_3} \right] \\
    &= X^{\overset{(3)}{h}_0}(F_1^\lambda)  (R(Y,Z)\p_\lambda)^{\overset{(3)}{v}_3}  + F_1^\lambda \left[X^{\overset{(3)}{h}_0}, (R(Y,Z)\p_\lambda)^{\overset{(3)}{v}_3} \right]\\
    &= X^i \frac{\delta F_1^\lambda}{\delta q^{(0)i}}  (R(Y,Z)\p_\lambda)^{\overset{(3)}{v}_3}  + F_1^\lambda \left[X^{\overset{(3)}{h}_0}, (R(Y,Z)\p_\lambda)^{\overset{(3)}{v}_3} \right].
\end{align*}
Then, we use Equations \eqref{eq: coef der F}-\eqref{eq: coef delta der F vanish} to obtain
$$\frac{\delta F_1^\lambda}{\delta q^{(0)i}}=\frac{\partial F_1^\lambda}{\partial q^{(0)i}} - 2(K_2)^\lambda_i,$$
which we simplify due to
$$\frac{\p F_1^\lambda}{\p q^{(0)i}} =2(K_2)^\lambda_i -F_1^\lambda\Gamma_{ai}^\lambda+ q^{(1)a}q^{(1)b}R^\lambda_{iab}.$$
Thus, we apply the result to the first term and
use  Lemma~\ref{lemma: Brackets vk} to calculate the second term in order to have
\begin{align*}
    \left[X^{\overset{(3)}{h}_0}, (R(Y, Z)u^{(2)})^{\overset{(3)}{v}_3} \right] &= (-F_1^\lambda\Gamma_{ai}^\lambda+ q^{(1)a}q^{(1)b}R^\lambda_{iab})X^i   (R(Y,Z)\p_\lambda)^{\overset{(3)}{v}_3}  + F_1^\lambda \left(\nabla_X(R(Y,Z)\p_\lambda)\right)^{\overset{(3)}{v}_3} .
\end{align*}
Hence,
{\small \begin{equation}\label{eq: formula 1}
    \left[X^{\overset{(3)}{h}_0}, (R(Y, Z)u^{(2)})^{\overset{(3)}{v}_3} \right] = \left(R(Y,Z)R(X,u^{(1)})u^{(1)} + (\nabla_X R)(Y,Z)u^{(2)} + R(\nabla_X Y, Z)u^{(2)} + R(Y, \nabla_X Z)u^{(2)} \right)^{\overset{(3)}{v}_3}
\end{equation}}
Similarly,
\begin{equation}\label{eq: formula 2}
\left[X^{\overset{(3)}{h}_0}, (R(Y,u^{(2)})Z)^{\overset{(3)}{v}_3} \right] =\left(R(Y, R(X,u^{(1)})u^{(1)})Z + (\nabla_X R)(Y, u^{(2)})Z + R(\nabla_X Y, u^{(2)})Z + R(Y, u^{(2)})\nabla_X Z\right)^{\overset{(3)}{v}_3}
\end{equation}
Next,
{\small \begin{align*}
    [X^{\overset{(3)}{h}_0}, (\nabla_{u^{(1)}} R)(Y, Z)u^{(1)})^{\overset{(3)}{v}_3}] &= [X^{\overset{(3)}{h}_0}, F_0^aF_0^b ((\nabla_{\p_a} R)(Y, Z)\p_b)^{\overset{(3)}{v}_3}] \\
    &= X^{\overset{(3)}{h}_0}(F_0^aF_0^b) ((\nabla_{\p_a} R)(Y, Z)\p_b)^{\overset{(3)}{v}_3} + q^{(1)a}q^{(1)b} [X^{\overset{(3)}{h}_0}, ((\nabla_{\p_a} R)(Y, Z)\p_b)^{\overset{(3)}{v}_3}] \\
    &= [-q^{(1)a} (\nabla_{\nabla_X \p_a} R)(Y, Z)V - q^{(1)b}(\nabla_{V} R)(Y, Z)\nabla_X \p_b + q^{(1)a} q^{(1)b} \nabla_X ((\nabla_{\p_a} R)(Y,Z)\p_b)]^{\overset{(3)}{v}_3}
\end{align*}}
and we obtain
\begin{equation}\label{eq: formula 3}
    [X^{\overset{(3)}{h}_0}, (\nabla_{u^{(1)}} R)(Y, Z)u^{(1)})^{\overset{(3)}{v}_3}]  = [(\nabla^2 R)(X,u^{(1)}; Y,Z)u^{(1)} + (\nabla_{u^{(1)}} R)(\nabla_X Y,Z)u^{(1)} + (\nabla_{u^{(1)}} R)(Y, \nabla_X Z)u^{(1)}]^{\overset{(3)}{v}_3}.
\end{equation}
Now, we apply the identities \eqref{eq: formula 1}-\eqref{eq: formula 2}-\eqref{eq: formula 3} to simplify the third Lie bracket and use the identities given in Lemma \ref{lemma: 2nd curvature symmetries}.
\begin{align*}
    \left[ \frac{\delta}{\delta q^{(0)j}}, \left(R(\p_i, \p_m)u^{(2)} + \frac13 R(\p_m, u^{(2)})\p_i + (\nabla_{u^{(1)}} R)(\p_i, \p_m)u^{(1)} + \frac13 (\nabla_{u^{(1)}} R)(\p_m, u^{(1)})\p_i \right)^{\overset{(3)}{v}_3}\right] = \\
 \Big{[}R(\p_i, \p_m)R(\p_j, u^{(1)})u^{(1)} + (\nabla_{\p_j}R)(\p_i, \p_m)u^{(2)} + R(\nabla_{\p_j}\p_i, \p_m)u^{(2)} + R(\p_i, \nabla_{\p_j}\p_m)u^{(2)} \qquad \quad\\
     + \frac13 R(\p_m, R(\p_j, u^{(1)})u^{(1)})\p_i + \frac13 (\nabla_{\p_j}R)(\p_m, u^{(2)})\p_i + \frac13 R(\nabla_{\p_j}\p_m, u^{(2)})\p_i + \frac13 R(\p_i, u^{(2)})\nabla_{\p_j}\p_i \\
     + (\nabla_{u^{(1)}} R)(\p_i, \nabla_{\p_j} \p_m)u^{(1)} + (\nabla_{u^{(1)}} R)(\nabla_{\p_j} \p_i, \p_m)u^{(1)} + \frac13 (\nabla_{u^{(1)}} R)(\nabla_{\p_j} \p_m, u^{(1)})\p_i \qquad \qquad \quad \\+ \frac13 (\nabla_{u^{(1)}} R)(\p_m, u^{(1)})\nabla_{\p_j} \p_i
     + (\nabla^2 R)(\p_j, u^{(1)}; \p_i, \p_m)u^{(1)} + \frac13 (\nabla^2 R)(\p_j, u^{(1)}; \p_m, u^{(1)})\p_i \Big{]}^{\overset{(3)}{v}_3}  \quad
\end{align*}
Combining the three Lie brackets to simplify the two Lie brackets in Equation \eqref{eq: Jacobi 3}, we obtain
{\small \begin{align*}
    C(u) &= \frac12 \Big{[}  -R(\p_j, u^{(1)})R(u^{(1)},\p_i)\p_m +  R(\p_j, V)R(\p_i, \p_m)u^{(1)}- \frac13 R(u^{(1)}, R(u^{(1)},\p_i)\p_m)\p_j+ \frac13 R(u^{(1)}, R(\p_i, \p_m)u^{(1)})\p_j \\
    &\quad - R(\p_i, \p_m)R(\p_j, u^{(1)})u^{(1)} - (\nabla_{\p_j} R)(\p_i, \p_m)u^{(2)} - \frac13 R(\p_m, R(\p_j, u^{(1)})u^{(1)})\p_i - \frac13 (\nabla_{\p_j} R)(\p_m, u^{(2)})\p_i \\
    & \quad + R(\p_i, u^{(1)})R(u^{(1)},\p_j)\p_m - R(\p_i, u^{(1)})R(\p_j, \p_m)u^{(1)} + \frac{1}{3} R(u^{(1)}, R(u^{(1)},\p_j)\p_m)\p_i- \frac{1}{3} R(u^{(1)}, R(\p_j, \p_m)u^{(1)})\p_i \\
&\quad + R(\p_j, \p_m)R(\p_i, u^{(1)})u^{(1)} + (\nabla_{\p_i} R)(\p_j, \p_m)u^{(2)} + \frac{1}{3} R(\p_m, R(\p_i, u^{(1)})u^{(1)})\p_j + \frac{1}{3} (\nabla_{\p_i} R)(\p_m, u^{(2)})\p_j \\
    &\quad - (\nabla^2 R)(\p_j; u^{(1)}, \p_i, \p_m)u^{(1)} - \frac13 (\nabla^2 R)(\p_j, u^{(1)}; \p_m, u^{(1)})\p_i + (\nabla^2 R)(\p_i, u^{(1)}; \p_j, \p_m)u^{(1)} + \frac{1}{3} (\nabla^2 R)(\p_i, u^{(1)}; \p_m, u^{(1)})\p_j \Big{]}_p
\end{align*}
}
Now, if we conjugate  Equation \eqref{eq: S 3 final} with the previous equation and  take into account Lemma \ref{lemma: fX}, we conclude that
{\small
\begin{align*}
    \overset{(3)}{K}_3\circ [X^{\overset{(3)}{h}_0}, Y^{\overset{(3)}{h}_0}] &= \frac12\Big{[} -\frac1{3} R(X, Y)u^{(3)} + 3R(Y, u^{(1)})R(X,u^{(1)})u^{(1)} - 3R(X,u^{(1)})R(Y,u^{(1)})u^{(1)} + R(u^{(1)}, R(X, u^{(1)})u^{(1)})Y \\
    &\quad - R(u^{(1)}, R(Y,u^{(1)})u^{(1)})X - R(u^{(1)},R(X,Y)u^{(1)})u^{(1)} + (\nabla_{X} R)(Y, u^{(1)})u^{(2)}  - (\nabla_{Y} R)(X, u^{(1)})u^{(2)} \\
    &\quad + \frac{1}{3} (\nabla_{X} R)(u^{(1)}, u^{(2)})Y - \frac13(\nabla_{Y} R)(u^{(1)}, u^{(2)})X + (\nabla^2 R)(X, u^{(1)}; Y, u^{(1)})u^{(1)} - (\nabla^2 R)(Y, u^{(1)}; X, u^{(1)})u^{(1)}\Big{]}.
\end{align*}
}
 We will now make use of the identities given in Lemma \ref{lemma: curvature symmetries} to group as many terms together as possible.
Hence,
\begin{align*}
   & 3R(Y, u^{(1)})R(X,u^{(1)})u^{(1)} - 3R(X,u^{(1)})R(Y,u^{(1)})u^{(1)}
    + R(u^{(1)}, R(X, u^{(1)})u^{(1)})Y- R(u^{(1)}, R(Y,u^{(1)})u^{(1)})X \\
    &- R(u^{(1)},R(X,Y)u^{(1)})u^{(1)} \\
    &= (R(Y,u^{(1)})R)(X,u^{(1)})u^{(1)} - (R(X,u^{(1)})R)(Y,u^{(1)})u^{(1)}.
\end{align*}
Then, using the identities given in Lemma \ref{lemma: 3rd curvature symmetries}, we obtain
\begin{align*}
&  (R(Y,u^{(1)})R)(X,u^{(1)})u^{(1)} - (R(X,u^{(1)})R)(Y,u^{(1)})u^{(1)} \\
    &= (\nabla^2 R)(Y,u^{(1)}; X,u^{(1)})u^{(1)} - (\nabla^2 R)(X,u^{(1)}; Y,u^{(1)})u^{(1)} - (\nabla^2 R)(u^{(1)},u^{(1)}; X,Y)u^{(1)}.
\end{align*}

Therefore,
{\small
\begin{align*}
   \overset{(3)}{K}_3\circ [X^{\overset{(3)}{h}_0}, Y^{\overset{(3)}{h}_0}] &= -\frac1{3!} R(X,Y)u^{(3)} - \frac12 (\nabla^2 R)(u^{(1)},u^{(1)}; X,Y)u^{(1)} \\
    &\quad + \frac12 \left[(\nabla_{X} R)(Y, u^{(1)})u^{(2)} -  (\nabla_{Y} R)(X, u^{(1)})u^{(2)} + \frac{1}{3} (\nabla_{X} R)(u^{(1)}, u^{(2)})Y - \frac13(\nabla_{Y} R)(u^{(1)}, u^{(2)})X \right]
\end{align*}
}
Next, from the first and second Bianchi identities given in Lemma \ref{lemma: 2nd curvature symmetries}, we get
\begin{align*}
    (\nabla_X R)(Y,u^{(1)})&u^{(2)} - (\nabla_Y R)(X,u^{(1)})u^{(2)} + \frac13 (\nabla_X R)(u^{(1)}, u^{(2)})Y - \frac1{3} (\nabla_Y R)(u^{(1)}, u^{(2)})X \\
    &= -\frac23 (\nabla_{u^{(1)}} R)(X,Y)u^{(2)}-\frac1{3} (\nabla_{u^{(2)}}R)(X,Y)u^{(1)}
\end{align*}
Finally, using the definition of the second covariant derivative of $R$, we conclude that
\begin{align*}
    \overset{(3)}{K}_3\circ [X^{\overset{(3)}{h}_0}, Y^{\overset{(3)}{h}_0}] &= -\frac1{3!} R(X,Y)u^{(3)} - \frac13 (\nabla_{u^{(1)}} R)(X,Y)u^{(2)} + \frac13 (\nabla_{u^{(2)}}R)(X,Y)u^{(1)} - \frac12 (\nabla^2_{u^{(1)}} R)(X,Y)u^{(1)}
\end{align*}
\qed \end{proof}

\begin{proof}[Proof of Theorem~\ref{thm: liebrac3}]
The identities involving $\overset{(3)}{v}_3$ follow from Lemma~\ref{lemma: Brackets vk}.  The brackets involving $\overset{(3)}{h}_2$ are given by Lemma~\ref{lemma: Brackets h2 3}; the brackets involving $\overset{(3)}{h}_1$ are given by Lemma~\ref{lemma: Brackets h1 3}; and the $\overset{(3)}{h}_0$-$\overset{(3)}{h}_0$ bracket is given by Lemma~\ref{lemma: Brackets h0h0 3}.  The lower components are inherited from the second-order bracket formulas through Lemma~\ref{lemma: K mu}.
\qed \end{proof}
\subsection{General reduction of $k$-vertical bracket components}
\label{app:general-bracket-reduction}

This appendix gives the coordinate reduction used to obtain the general partial bracket formulas in Propositions~\ref{prop: Brackets hk-1} and~\ref{prop: Brackets h_0 hk-1}.

\begin{lemma}\label{lem: Liebrac}
Consider the vector fields defined by
$$\frac{\overset{(k-1)\qquad}{\delta_{k-1}}}{\delta q^{(\alpha)i}}=\frac{\overset{(k)}{\delta}}{\delta q^{(\alpha)i}}+(C_{k-\alpha})_{i}^j\frac{\overset{(k)}{\delta}}{\delta q^{(k)j}}, \;  \alpha=0,1, \ldots, k-1,$$
where $(C_{\alpha})_{i}^l$ are the dual coefficients given by (\ref{eq: dual coefficients recursion}). Then
\begin{align*}
\overset{(k)}{K}_k\circ \bigg[\frac{\overset{(k)}{\delta}}{\delta q^{(\alpha)i}},\frac{\overset{(k)}{\delta}}{\delta q^{(\beta)l}}\bigg]
=&\Bigg(\frac{\overset{(k-1)}{\delta} (C_{k-\alpha})_i^j}{\delta q^{(\beta)l}}-\frac{\overset{(k-1)}{\delta} (C_{k-\beta})_l^j}{\delta q^{(\alpha)i}}+(C_{k-\alpha})_i^a\frac{\overset{(k)}{\delta} (C_{k-\beta})_l^j}{\delta q^{(k)a}}-(C_{k-\beta})_l^a\frac{\overset{(k)}{\delta} (C_{k-\alpha})_i^j}{\delta q^{(k)a}}\Bigg) \partial_j\\&
+ \overset{(k)}{K}_k\circ \bigg[\frac{\overset{(k-1)\qquad}{\delta_{k-1}}}{\delta q^{(\alpha)i}},\frac{\overset{(k-1) \qquad}{\delta_{k-1}}}{\delta q^{(\beta)l}}\bigg],
\end{align*}
for all $ \alpha,\beta=0, 1, \ldots, k-1$.
\end{lemma}

\begin{proof}
First of all, we observe that each vector field $\displaystyle  \frac{\overset{(k)}{\delta}}{\delta q^{(\alpha)i}}$ decomposes into the auxiliary  vector field $\displaystyle \frac{\overset{(k-1)\qquad}{\delta_{k-1}}}{\delta q^{(\alpha)i}}$ and the $k$-vertical vector fields $\displaystyle \frac{\overset{(k)}{\delta}}{\delta q^{(k)j}}$, $j=1, \ldots, n$, as follows.
$$\frac{\overset{(k)}{\delta}}{\delta q^{(\alpha)i}}=\frac{\overset{(k-1)\qquad}{\delta_{k-1}}}{\delta q^{(\alpha)i}}-(C_{k-\alpha})_{i}^j\frac{\overset{(k)}{\delta}}{\delta q^{(k)j}}$$
Hence, we have
\begin{align*}
\bigg[\frac{\overset{(k)}{\delta}}{\delta q^{(\alpha)i}},\frac{\overset{(k)}{\delta}}{\delta q^{(\beta)l}}\bigg]=&
\bigg[\frac{\overset{(k)}{\delta}}{\delta q^{(\alpha)i}},\frac{\overset{(k-1)\qquad}{\delta_{k-1}}}{\delta q^{(\beta)i}}-(C_{k-\beta})_{i}^j\frac{\overset{(k)}{\delta}}{\delta q^{(k)j}}\bigg]\\
=&\left[\frac{\overset{(k)}{\delta}}{\delta q^{(\alpha)i}},\frac{\overset{(k-1)\qquad}{\delta_{k-1}}}{\delta q^{(\beta)i}}\right]-
\frac{\overset{(k)}{\delta} (C_{k-\beta})_l^j}{\delta q^{(\alpha)i}}\frac{\delta}{\delta q^{(k)j}}- (C_{k-\beta})_{i}^j\left[\frac{\overset{(k)}{\delta}}{\delta q^{(\alpha)i}},\frac{\overset{(k)}{\delta}}{\delta q^{(k)j}}\right].
\end{align*}
Moreover,
\begin{align*}
\left[\frac{\overset{(k)}{\delta}}{\delta q^{(\alpha)i}},\frac{\overset{(k-1)\qquad}{\delta_{k-1}}}{\delta q^{(\beta)i}}\right]=&
\left[\frac{\overset{(k-1)\qquad}{\delta_{k-1}}}{\delta q^{(\alpha)i}}-(C_{k-\alpha})_{i}^j\frac{\overset{(k)}{\delta}}{\delta q^{(k)j}},\frac{\overset{(k-1)\qquad}{\delta_{k-1}}}{\delta q^{(\beta)i}}\right]\\
=&\left[\frac{\overset{(k-1)\qquad}{\delta_{k-1}}}{\delta q^{(\alpha)i}},\frac{\overset{(k-1)\qquad}{\delta_{k-1}}}{\delta q^{(\beta)i}}\right]+
\frac{\overset{(k-1)\qquad}{\delta_{k-1}}\!\!(C_{k-\alpha})_i^j}{\delta q^{(\beta)l}} \frac{\delta}{\delta q^{(k)j}}- (C_{k-\alpha})_i^j\left[\frac{\overset{(k)}{\delta}}{\delta q^{(k)i}},\frac{\overset{(k-1)\qquad}{\delta_{k-1}}}{\delta q^{(\beta)i}}\right]
\end{align*}
and, since

\begin{align*}
\frac{\overset{(k-1)\qquad}{\delta_{k-1}}\!\!(C_{k-\alpha})_i^j}{\delta q^{(\beta)l}}
= \frac{\overset{(k)}{\delta} (C_{k-\alpha})_i^j}{\delta q^{(\beta)i}}+(C_{k-\beta})_{i}^j\frac{\overset{(k)}{\delta} (C_{k-\alpha})_i^j}{\delta q^{(k)j}}
\end{align*}
and
\begin{align*}
\left[\frac{\overset{(k)}{\delta}}{\delta q^{(k)i}},\frac{\overset{(k-1)\qquad}{\delta_{k-1}}}{\delta q^{(\beta)i}}\right]
=0,
\end{align*}
we  conclude that

\begin{align*}
\left[\frac{\overset{(k)}{\delta}}{\delta q^{(\alpha)i}},\frac{\overset{(k-1)\qquad}{\delta_{k-1}}}{\delta q^{(\beta)i}}\right]=&
\left[\frac{\overset{(k-1)\qquad}{\delta_{k-1}}}{\delta q^{(\alpha)i}}-(C_{k-\alpha})_{i}^j\frac{\overset{(k)}{\delta}}{\delta q^{(k)j}},\frac{\overset{(k-1)\qquad}{\delta_{k-1}}}{\delta q^{(\beta)i}}\right]\\
=&\left[\frac{\overset{(k-1)\qquad}{\delta_{k-1}}}{\delta q^{(\alpha)i}},\frac{\overset{(k-1)\qquad}{\delta_{k-1}}}{\delta q^{(\beta)i}}\right]+
\frac{\overset{(k-1)\qquad}{\delta_{k-1}}\!\!(C_{k-\alpha})_i^j}{\delta q^{(\beta)l}} \frac{\delta}{\delta q^{(k)j}}- (C_{k-\alpha})_i^j\left[\frac{\overset{(k)}{\delta}}{\delta q^{(k)i}},\frac{\overset{(k-1)\qquad}{\delta_{k-1}}}{\delta q^{(\beta)i}}\right]\\
=&\left(\frac{\overset{(k)}{\delta} (C_{k-\alpha})_i^j}{\delta q^{(\beta)i}}+(C_{k-\beta})_{i}^j\frac{\overset{(k)}{\delta} (C_{k-\alpha})_i^j}{\delta q^{(k)j}}+
\frac{\overset{(k-1)\qquad}{\delta_{k-1}}\!\!(C_{k-\alpha})_i^j}{\delta q^{(\beta)l}}\right) \frac{\delta}{\delta q^{(k)j}}.
\end{align*}

On the other hand, we have
\begin{align*}
\left[\frac{\overset{(k)}{\delta}}{\delta q^{(\alpha)i}},\frac{\overset{(k)}{\delta}}{\delta q^{(k)j}}\right]=&\left[\frac{\overset{(k-1)\qquad}{\delta_{k-1}}}{\delta q^{(\alpha)i}}-(C_{k-\alpha})_{i}^j\frac{\overset{(k)}{\delta}}{\delta q^{(k)j}},\frac{\overset{(k)}{\delta}}{\delta q^{(k)j}}\right]\\
=&\left[\frac{\overset{(k-1)\qquad}{\delta_{k-1}}}{\delta q^{(\alpha)i}},\frac{\overset{(k)}{\delta}}{\delta q^{(k)j}}\right]+\frac{\overset{(k)}{\delta} (C_{k-\alpha})_{i}^j}{\delta q^{(k)j}}\frac{\overset{(k)}{\delta}}{\delta q^{(k)j}}\\
=&\frac{\overset{(k)}{\delta} (C_{k-\alpha})_{i}^j}{\delta q^{(k)j}}\frac{\overset{(k)}{\delta}}{\delta q^{(k)j}}.
\end{align*}
Hence,
\begin{align*}
\left[\frac{\overset{(k)}{\delta}}{\delta q^{(\alpha)i}},\frac{\overset{(k)}{\delta}}{\delta q^{(\beta)l}}\right]
=&\left(\frac{\overset{(k-1)}{\delta} (C_{k-\alpha})_i^j}{\delta q^{(\beta)l}}-\frac{\overset{(k-1)}{\delta} (C_{k-\beta})_l^j}{\delta q^{(\alpha)i}}+(C_{k-\alpha})_i^a\frac{\overset{(k)}{\delta} (C_{k-\beta})_l^j}{\delta q^{(k)a}}-(C_{k-\beta})_l^a\frac{\overset{(k)}{\delta} (C_{k-\alpha})_i^j}{\delta q^{(k)a}}\right) \frac{\overset{(k)}{\delta}}{\delta q^{(k)j}}\\&
+ \left[\frac{\overset{(k-1)\qquad}{\delta_{k-1}}}{\delta q^{(\alpha)i}},\frac{\overset{(k-1) \qquad}{\delta_{k-1}}}{\delta q^{(\beta)l}}\right],
\end{align*}
and the result follows immediately.
\qed \end{proof}

\begin{proposition}\label{prop: coordinteLiebrac}
Let $S_{(\alpha)i, (\beta)l}^{(\mu)j}$ be the $j$th-coordinate function  of $\overset{(k)}{K}_\mu\circ \left[\frac{\overset{(k)}{\delta}}{\delta q^{(\alpha)i}},\frac{\overset{(k)}{\delta}}{\delta q^{(\beta)l}}\right]$, for each $ \alpha,\beta, \mu=0, 1, \ldots, k$. Then
\begin{align*}
S_{(\alpha)i, (\beta)l}^{(k)j}=&\frac{\overset{(k-1)}{\delta} (C_{k-\alpha})_i^j}{\delta q^{(\beta)l}}-\frac{\overset{(k-1)}{\delta} (C_{k-\beta})_l^j}{\delta q^{(\alpha)i}}+(C_{k-\alpha})_i^a\frac{\overset{(k)}{\delta} (C_{k-\beta})_l^j}{\delta q^{(k)a}}-(C_{k-\beta})_l^a\frac{\overset{(k)}{\delta} (C_{k-\alpha})_i^j}{\delta q^{(k)a}}
+ \sum_{a=1}^{n}\sum_{\mu=0}^{k-1}S_{(\alpha)i, (\beta)l}^{(\mu)a}(C_{k-\mu})_a^j,
\end{align*}
for all $ \alpha,\beta=0, 1, \ldots, k-1$.
\end{proposition}

\begin{proof}
Note that
 $\overset{(k)}{\tau}_{\!k-1\ast}(\frac{\overset{(k-1)\qquad}{\delta_{k-1}}}{\delta q^{(\alpha)i}})=\frac{\overset{(k-1)}{\delta}}{\delta q^{(\alpha)i}}.$ Then

\begin{align*}
\overset{(k)}{K}_k\circ \big[\frac{\overset{(k-1)\qquad}{\delta_{k-1}}}{\delta q^{(\alpha)i}},\frac{\overset{(k-1) \qquad}{\delta_{k-1}}}{\delta q^{(\beta)l}}\big]
=&\sum_{\mu=0}^{k-1}S_{(\alpha)i, (\beta)l}^{(\mu)j}\overset{(k)}{K}_k(\frac{\overset{(k-1)\qquad}{\delta_{k-1}}}{\delta q^{(\mu)j}})=\sum_{a=1}^{n}\sum_{\mu=0}^{k-1}S_{(\alpha)i, (\beta)l}^{(\mu)a}(C_{k-\mu})_a^j\partial_j.
\end{align*}
Now, if we combine this equation with Lemma \ref{lem: Liebrac}, the result follows.
\qed \end{proof}

\begin{proof}[Proof of Proposition~\ref{prop: Brackets hk-1}]
Using Proposition \ref{prop: coordinteLiebrac}, we get
           $$\overset{(k)}{K}_k\circ \left[\frac{\delta}{\delta q^{(k-1)i}},  \frac{\delta}{\delta q^{(k-1)l}}\right]=  \Bigg( \frac{\partial (C_{1})_i^j}{\partial q^{(k-1)l}}-\frac{\partial (C_{1})_l^j}{\partial q^{(k-1)i}}+(C_{1})_i^a\frac{\partial (C_{1})_l^j}{\partial q^{(k)a}}-(C_{1})_l^a\frac{\partial (C_{1})_i^j}{\partial q^{(k)a}}+ \sum_{a=1}^{n}\sum_{\mu=0}^{k-1}S_{(k-1)i, (k-1)l}^{(\mu)a}(C_{k-\mu})_a^j\Bigg)\partial_j$$
and, since $S_{(k-1)i, (k-1)l}^{(\mu)a}=0$, for $\mu=0,1, \ldots, k-1$, and $(C_{1})_i^j$ do not depend on $q^{(\lambda)a}$, for $\lambda>1$, the result is valid for $\alpha=1$ (for all $k\in \N$).
         Now, we suppose that the result is valid for $\alpha-1$ (for all $k\in \N$) and $\alpha<k-1$. Then, $\phantom{T}^kS_{(k-\alpha)i, (k-1)l}^{(\mu)a}=\phantom{T}^{k-1}S_{(k-(\alpha-1))i, (k-1)l}^{(\mu)a}=0$ and
           $$\overset{(k)}{K}_k\circ \left[\frac{\delta}{\delta q^{(k-\alpha)i}},  \frac{\delta}{\delta q^{(k-1)l}}\right]=  \Bigg( \frac{\partial (C_{\alpha})_i^j}{\partial q^{(k-1)l}}-\frac{\overset{(k-1)}{\delta} (C_{1})_l^j}{\delta q^{(k-\alpha)i}}+(C_{\alpha})_i^a\frac{\partial (C_{1})_l^j}{\partial q^{(k)a}}-(C_{1})_l^a\frac{\partial (C_{\alpha})_i^j}{\partial q^{(k)a}}+ \sum_{a=1}^{n}\sum_{\mu=0}^{k-1}S_{(k-\alpha)i, (k-1)l}^{(\mu)a}(C_{k-\mu})_a^j\Bigg)\partial_j$$
  is still zero, so the result is valid for $\alpha$. This means that the identity is hereditary  in $\alpha$ (up to $\alpha=k-2$). Therefore, we conclude the result is valid for $\alpha<k-1$.

  Finally, for $\alpha=k-1$, and because $\phantom{T}^kS_{(1)i, (k-1)l}^{(\mu)a}=\phantom{T}^{k-1}S_{(1)i, (k-1)l}^{(\mu)a}=0$, for $\mu=0,1, \ldots, k-1$, we have
 {\small \begin{align*}
           \overset{(k)}{K}_k\circ \left[\frac{\delta}{\delta q^{(1)i}},  \frac{\delta}{\delta q^{(k-1)l}}\right]=&
            \Bigg( \frac{\partial (C_{k-1})_i^j}{\partial q^{(k-1)l}}-\frac{\overset{(k-1)}{\delta} (C_{1})_l^j}{\delta q^{(1)i}}+(C_{k-1})_i^a\frac{\partial (C_{1})_l^j}{\partial q^{(k)a}}-(C_{1})_l^a\frac{\partial (C_{\alpha})_i^j}{\partial q^{(k)a}}+ \sum_{a=1}^{n}\sum_{\mu=0}^{k-1}S_{(1)i, (k-1)l}^{(\mu)a}(C_{k-\mu})_a^j\Bigg)\partial_j\\
            = &
             \Bigg( \frac{\partial (K_{k-1})_i^j}{\partial q^{(k-1)l}}-\frac{\overset{(k-1)}{\delta} (K_{1})_l^j}{\delta q^{(1)i}}\Bigg)\partial_j\\
           = & \Bigg( \frac{\partial (K_{k-1})_i^j}{\partial q^{(k-1)l}}-\frac{\partial (K_{1})_l^j}{\partial q^{(1)i}}\Bigg)\partial_j,
         \end{align*}
 }
         which is zero due to Equation  \ref{eq: partialK}.
\qed \end{proof}

\begin{proof}[Proof of Proposition~\ref{prop: Brackets h_0 hk-1}]
Due to  Lemma \ref{lemma: Brackets vk}, we have that $\phantom{T}^k\!S_{(0)i, (k-1)l}^{(k-1)j}= \phantom{T}^{k-1}\!\!S_{(0)i, (k-1)l}^{(k-1)j}=\Gamma_{il}^j$ and $\phantom{T}^k\!S_{(0)i, (k-1)l}^{(\mu)j}= \phantom{T}^{k-1}\!\!S_{(0)i, (k-1)l}^{(\mu)j}=0$, for $\mu=0, 1, \ldots, k-2$.

On the other hand, using Proposition \ref{prop: coordinteLiebrac} and Equations \eqref{eq: partialK}--\eqref{eq: partialK-1}, we get
\small{
\begin{align*}
\phantom{T}^kS_{(0)i, (k-1)l}^{(k-1)j}  %& =  \frac{\overset{(k-1)}{\delta}(C_{k})_i^j}{\delta q^{(k-1)l}}-\frac{\overset{(k-1)}{\delta} (C_{1})_l^j}{\delta q^{(0)i}}+(C_{k})_i^a\frac{\partial (C_{1})_l^j}{\partial q^{(k)a}}-(C_{1})_l^a\frac{\partial (C_{k})_i^j}{\partial q^{(k)a}}+ \sum_{a=1}^{n}\sum_{\mu=0}^{k-1}S_{(0)i, (k-1)l}^{(\mu)a}(C_{k-\mu})_a^j \\
                & =   \frac{\partial(C_{k})_i^j}{\partial q^{(k-1)l}}-\frac{\partial (C_{1})_l^j}{\partial q^{(0)i}}+(K_1)_i^a\frac{\partial (C_{1})_l^j}{\partial q^{(1)a}}-(C_{1})_l^a\frac{\partial (C_{k})_i^j}{\partial q^{(k)a}}+ S_{(0)i, (k-1)l}^{(k-1)a}(C_{1})_a^j \\
                & =  -q^{(1)a} dq^j\left(\nabla_{\partial_i} \nabla_{\partial_a} \partial_l\right) +\frac 1 kq^{(1)a} dq^j\left(\nabla_{\partial_l} \nabla_{\partial_a} \partial_i+(k-1)\nabla_{\partial_a} \nabla_{\partial_l} \partial_i\right)
                \\
                & = \frac 1 kq^{(1)a} dq^j\left(\nabla_{\partial_l} \nabla_{\partial_i} \partial_a -\nabla_{\partial_i} \nabla_{\partial_l} \partial_a\right) +\frac{k-1} kq^{(1)a} dq^j\left(\nabla_{\partial_a} \nabla_{\partial_i} \partial_l-\nabla_{\partial_i} \nabla_{\partial_a} \partial_l\right)
                \\
                & = -\frac 1 kq^{(1)a}\left(R_{ila}^j+(k-1)R_{ial}^j\right).
              \end{align*}
              }
\qed \end{proof}

\section{Koszul computations for the $k$-Sasaki metric}\label{app:koszul-computations}
\subsection{Metric reductions for the $k$-Sasaki metric}\label{app:sasaki-koszul-reductions}

This appendix contains the general Koszul-formula reductions used in Section 4.

\begin{proof}[Proof of Lemma \ref{lem: appKoszul}]
The result follows from the metric compatibility of the Levi--Civita connection on $M$, together with Lemma \ref{lemma: prop} and the identity
\[
        \llangle Y^{\overset{(k)}{h}_\beta},Z^{\overset{(k)}{h}_\mu}\rrangle
        =\delta_{\beta\mu}\langle Y,Z\rangle\circ\tau_k.
\]
Indeed, only the $\overset{(k)}{h}_0$-lift differentiates functions pulled back from $M$ non-trivially, and so
\[
        X^{\overset{(k)}{h}_\alpha}\big(\llangle Y^{\overset{(k)}{h}_\beta},Z^{\overset{(k)}{h}_\mu}\rrangle\big)
        =\delta_{\alpha0}\delta_{\beta\mu}X\langle Y,Z\rangle\circ\tau_k.
\]
Metric compatibility on $M$ gives the stated formula.
\qed \end{proof}

\begin{proof}[Proof of Lemma \ref{lemma: nablahh}]
To prove the first $k-1$ identities, it is enough to show that
\begin{equation}\label{eq:connk-app}
        \llangle \overset{(k)}\nabla_{X^{\overset{(k)}{h}_\alpha}}Y^{\overset{(k)}{h}_\beta},Z^{\overset{(k)}{h}_\mu}\rrangle_k
        =\llangle \overset{(k-1)}\nabla_{X^{\overset{(k-1)}{h}_\alpha}}Y^{\overset{(k-1)}{h}_\beta},Z^{\overset{(k-1)}{h}_\mu}\rrangle_{k-1},
\end{equation}
for all $X,Y,Z\in\mathfrak X(M)$ and $\alpha,\beta,\mu=0,1,\ldots,k-1$.  This follows by applying the Koszul formula \eqref{eq: Koszul}, Lemma \ref{lem: appKoszul}, and the compatibility of the lift brackets under $\overset{(k)}{\tau}_{\!k-1}$.

For the last identity, let $W$ be $k$-horizontal.  By the Koszul formula, Lemma \ref{lem: appKoszul}, and the bracket formulas involving $\overset{(k)}{v}_k$-lifts,
\[
        \llangle\overset{(k)}\nabla_WW,Z^{\overset{(k)}{v}_k}\rrangle_k=0,
        \qquad Z\in\mathfrak X(M).
\]
Thus
\begin{equation}\label{eq: conKk-app}
        \overset{(k)}{K}_k\circ\left(\overset{(k)}\nabla_WW\right)=0.
\end{equation}
Polarizing this identity with $W=X^{\overset{(k)}{h}_\alpha}+Y^{\overset{(k)}{h}_\beta}$ gives
\[
        2\overset{(k)}{K}_k\circ\left(\overset{(k)}\nabla_{X^{\overset{(k)}{h}_\alpha}}Y^{\overset{(k)}{h}_\beta}\right)
        =\overset{(k)}{K}_k\circ\left([X^{\overset{(k)}{h}_\alpha},Y^{\overset{(k)}{h}_\beta}]\right),
\]
which proves the claim.
\qed \end{proof}

\begin{proof}[Proof of Proposition \ref{prop: Levi--Civita kSasaki}]
Let $X,Y,Z\in\mathfrak X(M)$.  Using the Koszul formula and the bracket identities involving $\overset{(k)}{v}_k$-lifts, one obtains
\[
        \llangle \overset{(k)}\nabla_{X^{\overset{(k)}{h}_\alpha}}Y^{\overset{(k)}{v}_k},Z^{\overset{(k)}{v}_k}\rrangle=0,
        \qquad \alpha=1,\ldots,k-1,
\]
and
\begin{align*}
        2\llangle \overset{(k)}\nabla_{X^{\overset{(k)}{h}_0}}Y^{\overset{(k)}{v}_k},Z^{\overset{(k)}{v}_k}\rrangle
        &=X^{\overset{(k)}{h}_0}\llangle Y^{\overset{(k)}{v}_k},Z^{\overset{(k)}{v}_k}\rrangle
        +\llangle[X^{\overset{(k)}{h}_0},Y^{\overset{(k)}{v}_k}],Z^{\overset{(k)}{v}_k}\rrangle
        -\llangle[X^{\overset{(k)}{h}_0},Z^{\overset{(k)}{v}_k}],Y^{\overset{(k)}{v}_k}\rrangle\\
        &=2\langle\nabla_XY,Z\rangle.
\end{align*}
This gives the $k$-vertical components.  Finally, applying the same formula with $Z^{\overset{(k)}{h}_\mu}$, $\mu=0,\ldots,k-1$, gives
\[
        \llangle \overset{(k)}\nabla_{X^{\overset{(k)}{v}_k}}Y^{\overset{(k)}{v}_k},Z^{\overset{(k)}{h}_\mu}\rrangle=0,
\]
which yields the remaining components.
\qed \end{proof}

\subsection{Third-order Levi--Civita computations}\label{app:sasaki-third-order-lc}

\begin{proof}[Proof of Theorem \ref{teo: Levi--Civita Sasaki 3}]
The proof is obtained by the same method as Theorem \ref{teo: Levi--Civita Sasaki 2}.  Lemma \ref{lemma: nablahh} gives the components involving only $\overset{(3)}{h}_\alpha$-lifts from the third-order bracket formulas in Theorem \ref{thm: liebrac3}.  Proposition \ref{prop: Levi--Civita kSasaki} gives the $\overset{(3)}{v}_3$-components of the derivatives involving one $3$-vertical lift.  It remains only to compute the lower components of these latter derivatives.

For example, the $\overset{(3)}{h}_0$-component of $\overset{(3)}\nabla_{X^{\overset{(3)}{h}_0}}Y^{\overset{(3)}{v}_3}$ is obtained by pairing against $Z^{\overset{(3)}{h}_0}$:
\begin{align*}
        \llangle \overset{(3)}\nabla_{X^{\overset{(3)}{h}_0}}Y^{\overset{(3)}{v}_3},Z^{\overset{(3)}{h}_0}\rrangle
        &=-\llangle Y^{\overset{(3)}{v}_3},\overset{(3)}\nabla_{X^{\overset{(3)}{h}_0}}Z^{\overset{(3)}{h}_0}\rrangle\\
        &=\left\langle Y,
        \frac1{12}R(X,Z)u^{(3)}
        +\frac16(\nabla_{u^{(1)}}R)(X,Z)u^{(2)}
        -\frac16(\nabla_{u^{(2)}}R)(X,Z)u^{(1)}\right.\\
        &\qquad\left.
        +\frac14(\nabla_{u^{(1)}}^2R)(X,Z)u^{(1)}
        \right\rangle\\
        &=\left\langle
        \frac1{12}R(u^{(3)},Y)X
        +\frac16(\nabla_{u^{(1)}}R)(u^{(2)},Y)X
        -\frac16(\nabla_{u^{(2)}}R)(u^{(1)},Y)X\right.\\
        &\qquad\left.
        +\frac14(\nabla_{u^{(1)}}^2R)(u^{(1)},Y)X,
        Z\right\rangle.
\end{align*}
Thus
\[
        \tau_{3*}\left(\overset{(3)}\nabla_{X^{\overset{(3)}{h}_0}}Y^{\overset{(3)}{v}_3}\right)
        =\frac1{12}R(u^{(3)},Y)X
        +\frac16(\nabla_{u^{(1)}}R)(u^{(2)},Y)X
        -\frac16(\nabla_{u^{(2)}}R)(u^{(1)},Y)X
        +\frac14(\nabla_{u^{(1)}}^2R)(u^{(1)},Y)X.
\]
The remaining components are obtained analogously by applying the Koszul formula against $\overset{(3)}{h}_0$-, $\overset{(3)}{h}_1$-, $\overset{(3)}{h}_2$-, and $\overset{(3)}{v}_3$-lifts and using the bracket formulas of Section 3.  Combining all components gives the theorem.
\qed \end{proof}

\section{Geodesic computations}\label{app:geodesic-computations}

\subsection{Third-order geodesic computations}\label{app:sasaki-third-order-geodesics}

\begin{proof}[Proof of Theorem \ref{teo: geod eq 3}]
Using Equation \eqref{eq: decomposition gamma T^3} and the torsion-free property of $\overset{(3)}\nabla$, the covariant acceleration is
\begin{align*}
        \overset{(3)}\nabla_{\dot\Gamma}\dot\Gamma
        &=\overset{(3)}\nabla_{\dot q^{\overset{(3)}{h}_0}}\dot q^{\overset{(3)}{h}_0}
        +2\overset{(3)}\nabla_{\dot q^{\overset{(3)}{h}_0}}(Y^{(1)})^{\overset{(3)}{h}_1}
        -[\dot q^{\overset{(3)}{h}_0},(Y^{(1)})^{\overset{(3)}{h}_1}]\\
        &\quad+2\overset{(3)}\nabla_{\dot q^{\overset{(3)}{h}_0}}(Y^{(2)})^{\overset{(3)}{h}_2}
        -[\dot q^{\overset{(3)}{h}_0},(Y^{(2)})^{\overset{(3)}{h}_2}]
        +2\overset{(3)}\nabla_{\dot q^{\overset{(3)}{h}_0}}(Y^{(3)})^{\overset{(3)}{v}_3}
        -[\dot q^{\overset{(3)}{h}_0},(Y^{(3)})^{\overset{(3)}{v}_3}]\\
        &\quad+\overset{(3)}\nabla_{(Y^{(1)})^{\overset{(3)}{h}_1}}(Y^{(1)})^{\overset{(3)}{h}_1}
        +2\overset{(3)}\nabla_{(Y^{(1)})^{\overset{(3)}{h}_1}}(Y^{(2)})^{\overset{(3)}{h}_2}
        -[(Y^{(1)})^{\overset{(3)}{h}_1},(Y^{(2)})^{\overset{(3)}{h}_2}]\\
        &\quad+2\overset{(3)}\nabla_{(Y^{(1)})^{\overset{(3)}{h}_1}}(Y^{(3)})^{\overset{(3)}{v}_3}
        -[(Y^{(1)})^{\overset{(3)}{h}_1},(Y^{(3)})^{\overset{(3)}{v}_3}]
        +2\overset{(3)}\nabla_{(Y^{(2)})^{\overset{(3)}{h}_2}}(Y^{(3)})^{\overset{(3)}{v}_3}
        -[(Y^{(2)})^{\overset{(3)}{h}_2},(Y^{(3)})^{\overset{(3)}{v}_3}].
\end{align*}
The terms involving $\overset{(3)}\nabla_{(Y^{(2)})^{\overset{(3)}{h}_2}}(Y^{(2)})^{\overset{(3)}{h}_2}$ and $\overset{(3)}\nabla_{(Y^{(3)})^{\overset{(3)}{v}_3}}(Y^{(3)})^{\overset{(3)}{v}_3}$ vanish by Theorem \ref{teo: Levi--Civita Sasaki 3}.  Applying Theorem \ref{teo: Levi--Civita Sasaki 3} and the bracket formulas from Section 3 to the remaining terms gives the four component equations displayed in Theorem \ref{teo: geod eq 3}.
\qed \end{proof}



\bibliographystyle{abbrv}%{ieeetr}
\bibliography{bibliopaper1}

\begin{table*}[t]
\centering
\caption{Notation used throughout the paper.}
\label{tab:notation}
\renewcommand{\arraystretch}{1.18}
\begin{tabularx}{\textwidth}{>{\raggedright\arraybackslash}p{0.23\textwidth}
                            >{\raggedright\arraybackslash}X}
\toprule
Notation & Meaning \\
\midrule
\(T^{(k)}M\) & The \(k\)-th order tangent bundle of \(M\), whose elements are \(k\)-jets \(j^k_0(q)\) of curves \(q\) at \(0\). \\

\(\tau_k:T^{(k)}M\to M\) & The canonical projection onto the base point, given by \(\tau_k(j^k_0(q))=q(0)\). \\

\(\overset{(k)}{\tau}_\alpha:T^{(k)}M\to T^{(\alpha)}M\) & The canonical truncation map, given by \(\overset{(k)}{\tau}_\alpha(j^k_0(q))=j^\alpha_0(q)\), for \(0\leq \alpha\leq k\). \\

\((q^{(0)i},q^{(1)i},\ldots,q^{(k)i})\) & Induced coordinates on \(T^{(k)}M\), with
\(q^{(\alpha)i}=\displaystyle\frac{1}{\alpha!}\frac{d^\alpha q^i}{dt^\alpha}(0)\). \\

\(j^kq\) & The \(k\)-jet of a curve \(q\), given by \(t\mapsto j^k_0(q_t)\), where \(q_t(s)=q(t+s)\). \\

\(J\) & The canonical almost-tangent structure on \(T^{(k)}M\). In coordinates,
\[
J=\frac{\partial}{\partial q^{(1)i}}\otimes dq^{(0)i}
+\cdots+
\frac{\partial}{\partial q^{(k)i}}\otimes dq^{(k-1)i}.
\] \\

\(TM_\oplus^k\) & The Whitney sum of $k$ tangent bundles, i.e. \(TM_\oplus^k := \bigoplus_{i=1}^k TM\). \\

\(V_\alpha\) & The \(\alpha\)-vertical bundle, defined as \(V_\alpha=\operatorname{Im}(J^\alpha)=\ker(\tau^{\alpha-1}_{k*})\), for \(\alpha=1,\ldots,k\). \\

\(\overset{(k)}{K}=(\overset{(k)}{K}_1,\ldots,\overset{(k)}{K}_k)\) & A connection map on \(T^{(k)}M\). When compatible with the tower of projections, it is called a connection tower. \\

% \(\overset{(\alpha)}{K}_\mu\) & The \(\mu\)-th component of the connection map on \(T^{(\alpha)}M\). For a connection tower, these components are compatible under the truncation maps. \\

\(\overset{(k)}{H}_0\) & The horizontal bundle \(\ker(\overset{(k)}{K})\subset TT^{(k)}M\). \\

\(\overset{(k)}{H}_\alpha\) & The \(\alpha\)-th component of the associated multiconnection, defined by \(\overset{(k)}{H}_\alpha=J^\alpha(\overset{(k)}{H}_0)\), for \(\alpha=0,\ldots,k-1\). \\

\(\overset{(k)}{N}_\alpha\) & The \(\alpha\)-horizontal bundle
$\overset{(k)}{N}_\alpha=\overset{(k)}{H}_0\oplus\cdots\oplus \overset{(k)}{H}_{\alpha-1}.$ \\

\(\overset{(k)}{h}_\alpha\), \(\overset{(k)}{v}_\alpha\), \(\overset{(k)}{n}_\alpha\)  & Projection operators onto \(\overset{(k)}{H}_\alpha\), \(\overset{(k)}{V}_\alpha\), and \(\overset{(k)}{N}_\alpha\), respectively. \\

\(\overset{(k)}\delta q^{(\alpha)i}\) & The adapted coframe determined by a connection map, defined such that
\(
\overset{(k)}{K}_\alpha=\overset{(k)}\delta q^{(\alpha)i}\otimes\partial_i .
\) \\

\(\displaystyle \frac{\overset{(k)}\delta}{\delta q^{(\alpha)i}}\) & The adapted frame dual to \(\overset{(k)}\delta q^{(\alpha)i}\). \\

\(X^{\overset{(k)}{h}\alpha}\) & The $\overset{(k)}{h}_\alpha$-lift of a vector field \(X\in\mathfrak X(M)\) to \(T^{(k)}M\), characterized by
\(\overset{(k)}{K}_\alpha(X^{\overset{(k)}{h}\alpha})=X\circ\tau_k\) and \(\overset{(k)}{K}_\beta(X^{h^k_\alpha})=0\) for \(\beta\neq\alpha\). \\

\(X^{\overset{(k)}{v}_k}\) & The \(k\)-vertical lift of \(X\) to \(T^{(k)}M\). This is also denoted by \(X^{\overset{(k)}{h}_k}\) when treating all lift components uniformly. \\

\(\overset{(k)}{d_T}\) & Tulczyjew's operator
\(
\overset{(k)}{d_T}:T^{(k)}M\to TT^{(k-1)}M\) defined by
$\displaystyle\overset{(k)}{d_T}(j^k_0q)
=
\frac{d}{dt}\Big|_{t=0}(j^{k-1}q)(t).$
\\

\(\overset{(k)}{F}:T^{(k)}M\to (TM)^k_\oplus\) & The diffeomorphism induced by a connection tower which pulls back the vector bundle structure of \((TM)_\oplus^k\) to \(T^{(k)}M\). \\

\(\nabla\) & The Levi--Civita connection of the Riemannian metric \(g\) on \(M\). \\

\(R\) & The curvature tensor of \(\nabla\), with the convention:
\(
R(X,Y)Z=\nabla_X\nabla_YZ-\nabla_Y\nabla_XZ-\nabla_{[X,Y]}Z.
\) \\

\(\overset{(k)}{g}\) & The \(k\)-Sasaki metric on \(T^{(k)}M\) induced by the adapted splitting and the base metric \(g\). \\

\(\overset{(k)}{\nabla}\) & The Levi--Civita connection of the \(k\)-Sasaki metric \(\overset{(k)}{g}\). \\

\bottomrule
\end{tabularx}
\end{table*}

%\bibliographystyle{cas-model2-names}

%\bibliography{cas-refs}



\end{document}